\documentclass[a4paper,11pt, twoside,reqno]{amsart}

\usepackage{fullpage}
\usepackage[T1]{fontenc}
\usepackage{multicol}

\usepackage[dvipsnames]{xcolor}
\usepackage[
    colorlinks=true,
    linkcolor=MidnightBlue,
    citecolor=MidnightBlue,
    urlcolor=MidnightBlue,
    pagebackref=true
]{hyperref}

\usepackage{enumerate}
\usepackage{bm}
\usepackage{tikz-cd}
\usepackage{amsmath,amssymb,amsfonts,amsthm}

\theoremstyle{plain}
\newtheorem{thm}{Theorem}[section]
\newtheorem{cor}[thm]{Corollary}
\newtheorem{conjecture}[thm]{Conjecture}
\newtheorem{lemma}[thm]{Lemma}
\newtheorem{prop}[thm]{Proposition}

\theoremstyle{definition}
\newtheorem{defin}[thm]{Definition}

\theoremstyle{remark}
\newtheorem{rem}[thm]{Remark}

\newcommand{\F}{\mathbb{F}}
\newcommand{\Z}{\mathbb{Z}}

\newcommand{\Gal}{\mathrm{Gal}}
\newcommand{\bfh}{\bm{h}}

\newcommand{\bfa}{\bm{a}}
\newcommand{\bfb}{\bm{b}}
\newcommand{\bu}{\bm{u}}
\newcommand{\bfe}{\bm{e}}

\newcommand{\bfv}{\bm{v}}
\newcommand{\bfx}{\bm{x}}
\newcommand{\bfy}{\bm{y}}
\newcommand{\bfz}{\bm{z}}

\newcommand{\bfc}{\bm{c}}
\newcommand{\bfd}{\bm{d}}

\newcommand{\ZZ}{\mathbb{Z}}
\newcommand{\RR}{\mathbb{R}}
\newcommand{\QQ}{\mathbb{Q}}

\newcommand{\bfbe}{\boldsymbol{\beta}}

\DeclareMathOperator{\sing}{Sing}

\DeclareMathOperator{\meas}{Vol}

\DeclareMathOperator{\codim}{codim}

\DeclareMathOperator{\Fr}{Fr}
\DeclareMathOperator{\G}{Gal}
\DeclareMathOperator{\Br}{Br}
\DeclareMathOperator{\GL}{GL}

\DeclareMathOperator{\rank}{rank}
\DeclareMathOperator{\im}{Im}
\DeclareMathOperator{\Tr}{Tr}

\DeclareMathOperator{\GG}{\mathbb{G}}
\DeclareMathOperator{\A}{\mathbb{A}}
\DeclareMathOperator{\PP}{\mathbb{P}}
\DeclareMathOperator{\Sing}{Sing}
\DeclareMathOperator{\Mat}{Mat}
\DeclareMathOperator{\Spec}{Spec}

\renewcommand{\bar}{\overline}

\usepackage{etoolbox}

\makeatletter
\patchcmd{\@starttoc}
  {\@startsection{}\@M\z@{\linespacing\@plus\linespacing}%
    {.5\linespacing}{\centering\contentsnamefont}{#2}}
  {}
  {}
  {\PackageError{toc}{Could not remove the contents heading}{}}
\makeatother
\makeatletter
\patchcmd{\@starttoc}
  {\addvspace{32\p@\@plus14\p@}}
  {}
  {}
  {\PackageError{toc}{Could not remove trailing TOC space}{}}
\makeatother

\newcommand{\x}{\bm{x}}
\newcommand{\dd}{\mathrm{d}}
\newcommand{\y}{\bm{y}}

\newcommand{\ve}{\varepsilon}
\renewcommand{\le}{\leqslant}
\renewcommand{\leq}{\leqslant}
\renewcommand{\ge}{\geqslant}
\renewcommand{\geq}{\geqslant}
\newcommand{\FF}{\mathbb{F}}

\newcommand{\rad}{\operatorname{rad}}
\newcommand{\diag}{\operatorname{diag}}

\newcommand{\lcm}{\operatorname{lcm}}
\newcommand{\bfw}{\bm{w}}
\newcommand{\Qbar}{\overline{\mathbb{Q}}}
\newcommand{\CC}{\mathbb{C}}
\newcommand{\SL}{\operatorname{SL}}
\newcommand{\cone}{\widehat}

\numberwithin{equation}{section}


\begin{document}

\title{The Artin--Davenport conjecture on cubic forms}

\begin{abstract}
We show that
any integral cubic form in ten or more variables
has a nontrivial integer zero.
Like earlier authors on this topic,
we use geometric methods to produce a zero in some cases,
and the circle method in all other cases.
New ingredients are required on both fronts:
in the first case,
bounds on the dimension of singular loci and incidence varieties,
and in the second case,
a stratified geometric sieve, 
together with stratification results for point counts over finite fields and a 
cube-root Hensel lifting analysis.
\end{abstract}

\author{Dante Bonolis}
\address{TU Graz\\ Steyrergasse  30\\ 8010 Graz\\ Austria}
\email{dante.bonolis@tugraz.at}

\author{Tim Browning}
\address{ISTA\\
Am Campus 1\\
3400 Klosterneuburg\\
Austria}
\email{tdb@ist.ac.at}
\address{Department of Mathematics\\ Imperial College London\\
SW7 2AZ\\ 
United Kingdom}
\email{tim.browning@imperial.ac.uk}

\author{Jakob Glas}
\address{Leibniz Universit\"at Hannover\\ Welfengarten 1\\ 30167 Hannover\\ Germany}
\email{glas@math.uni-hannover.de}
\address{Laboratoire de math\'ematiques d’Orsay\\ Universit\'e Paris Saclay\\ Orsay\\ France}
\email{jakob.glas@universite-paris-saclay.fr}

\author{Victor Y. Wang}

\address{Institute of Mathematics\\
Academia Sinica\\
Taipei 106319\\
Taiwan}
\email{vywang@as.edu.tw}

\maketitle

\thispagestyle{empty}
\setcounter{tocdepth}{1}

\begin{center}
  \contentsnamefont Contents
\end{center}
{\small
\begin{multicols}{2}
\tableofcontents

\end{multicols}
}

\section{Introduction}
\label{sec:intro}

In his  
presidential address to the London Mathematical Society 
in November 1959, 
Davenport~\cite{Dav60} conjectures that any integral cubic form in $n\geq 10$ variables always possesses a nontrivial integer zero. 
The exact same conjecture was attributed to Artin by Lang~\cite{lang}, in his address to the
American Mathematical Society in February 1960.  
For $n\leq 9$,  there exist cubic forms which fail to have 
nontrivial integer zeros. For example, 
if $K/\QQ$ is a cubic number field with 
$\ZZ$-basis $\omega_1,\omega_2,\omega_3$ for the ring of integers $\mathfrak{o}_K$, and if 
$p$ is a rational prime that is inert in $\mathfrak{o}_K$, then the nonary cubic form 
$$
\mathsf{N}_{K/\QQ}(x_1,x_2,x_3)+p\mathsf{N}_{K/\QQ}(x_{4},x_5,x_6)+p^{2}\mathsf{N}_{K/\QQ}(x_{7},x_8,x_9)
$$
has only the trivial zero over $\QQ_p$, where
$\mathsf N_{K/\QQ}(x,y,z)=N_{K/\QQ}(x\omega_1+y\omega_2+z\omega_3)$
is the associated norm form.
Thanks to work of Demyanov \cite{dem} (when $p\neq 3$) and Lewis \cite{lewis} (in general), it has been known since the 1950s that cubic forms in $n\geq 10$ variables have a nontrivial $p$-adic zero for every prime $p$.

In 1959, Davenport \cite{Dav59} adapted the circle method to prove that $n\geq 32$ suffices for 
any integral cubic form in $n$ variables to possess a nontrivial integer zero.
This was first improved to $n\geq 29$ in 1962 and then to $n\geq 16$ in 1963 by Davenport \cite{Dav62, Dav63}, before being extended to $n\geq 14$ by Heath-Brown \cite{HB07} in 2007, using an averaged van der Corput variant of the circle method. The following is our main result.

\begin{thm}
\label{thm:main}
The Artin--Davenport conjecture is true. 
\end{thm}

Suppose that $X\subset \PP^{n-1}$ is a cubic hypersurface defined by 
a cubic form $F\in \ZZ[x_1,\dots,x_n]$.  The arithmetic of $X$ has been the subject of intense focus over the years, with many results that 
work towards the 
Artin--Davenport conjecture, or that apply in smaller dimensions. 
Compelling evidence is provided by  Heath-Brown \cite{hb10},
who establishes the conjecture for the special class of smooth cubic hypersurfaces $X$. It follows from Hooley \cite{hooley9,hooley9.2, Hoo13} that we have the same conclusion if $X$ has at worst ordinary double points. 

Assuming that $X$ is
geometrically integral, not equal to  a
cone, and defined over $\QQ$,
and that $n\ge 4$,
the existence of a single rational point on $X$ is
enough to demonstrate the $\QQ$-unirationality of $X$, 
and so the Zariski density of $X(\QQ)$ in $X$. 
Variants of this result have been known for many years, and can be found in the work of Colliot-Th\'el\`ene and Salberger
\cite[Proposition 1.3]{c-t-s} and Koll\'ar \cite{kollar}.

For a large class of possibly singular cubic hypersurfaces $X\subset \PP^{n-1}$, 
Colliot-Th\'el\`ene \cite[Appendix]{Bro10}
 has calculated the Brauer group
$\Br(Y)$ associated to a smooth  model $Y$ of
$X$. (This calculation is based on earlier unpublished work of Salberger.)
As a consequence of this and the conjecture of 
Colliot-Th\'el\`ene \cite{bud} that the Brauer--Manin obstruction should control the Hasse principle for 
rationally connected varieties, one is led to the following conjecture.

\begin{conjecture}
Let $X\subset\PP^{n-1}$ be a cubic
hypersurface defined over $\QQ$ which is not a cone, with $n\geq 5$ and 
singular locus which is empty or of codimension at least $4$ in $X$.
Then the Hasse principle holds for $X$.
\end{conjecture}

This conjecture is formulated in \cite[\S1]{Bro10}.
As a crucial  auxiliary result, we shall prove the following result in
Lemma \ref{LEM:433-incidence-bound}, which may prove to be of independent interest. 

\begin{thm}\label{thm:sing}
Let $n\geq 2$ and let  $X\subset\PP^{n-1}$ be a cubic
hypersurface defined over $\QQ$,  
with  singular locus of 
 codimension at most $\lfloor (n-1)/3 \rfloor$ in $X$.
Then $X(\QQ) \ne \emptyset$.
\end{thm}

We note here that $\lfloor (n-1)/3 \rfloor = 3$ for $10\le n\le 12$ and $\lfloor (n-1)/3 \rfloor= 4$ for $13\le n\le 15$.

\subsection*{Overview of the paper}
\label{sec:overview}

Our paper may be read linearly as usual,
but the following non-linear outline of the main ideas
may be useful.  Associated to 
a cubic form $F\in \ZZ[x_1,\dots,x_n]$
is the $n\times n$ Hessian matrix  $H(\bm{x})$, which is formed from the second partial derivatives of $F$.
Throughout this paper, let $|\bfx| := \max\{|x_1|,\dots,|x_n|\}$
for $\bfx\in \RR^n$.
We will say that the cubic form $F$ is {\em good} if 
for any $H\geq 1$ and any $\ve>0$ we have the upper bound
\[
\#\{\x\in\ZZ^{n}: |\x|\leq H,\,\rank H(\x)=r\}\ll H^{r+\ve},
\]
for each integer $0\leq r\leq n$.  A crucial step in Davenport's \cite{Dav63}
treatment of general cubic forms is a proof of the fact that forms
that fail to be good automatically possess nontrivial integer
zeros. In the alternative case when $F$ is good and $n\ge 16$, he shows that 
\begin{equation}\label{Eq:LowerBound.N(P)}
N(P) =\#\{ \bfx\in \ZZ^n\colon |\bm{x}-P\bm{x}_0|\leq P,\, F(\bm{x})=0\}\gg P^{n-3}
\end{equation}
via the circle method, where  $\bm{x}_0\in \RR^n$ is a suitable nonsingular zero of $F$. Our basic strategy follows the same line of attack. However, we go further in that we assume throughout that $X(\QQ)=\emptyset$. From this we are able to draw various consequences about the geometry of $X$, such as Theorem~\ref{thm:sing},  which are then used to control the exponential sums involved. More precisely, we show that if $X(\QQ)=\emptyset$, then  (a smoothed version of) $N(P)$ also satisfies \eqref{Eq:LowerBound.N(P)}.
This is clearly a contradiction to $X(\QQ)=\emptyset$, and hence a rational point must have existed from the outset. 

The circle method is introduced in \S\ref{sec:intro-circle}.
We shall adopt (a smoothed version of) the Kloosterman circle method approach that forms the backbone of the works by Heath-Brown \cite{hb10} and Hooley \cite{hooley9}, and which relies on an application 
of Poisson summation to bring into play 
the complete exponential sums
\begin{equation}\label{eq:defSq}
S_q(\bm{v})
:= \sum_{a\in (\ZZ/q\ZZ)^\ast} \sum_{\bfx\in (\ZZ/q\ZZ)^n}
e_q(aF(\bm{x})+\bm{v}\cdot \bm{x}),
\end{equation}
for integers $q\ge 1$ and vectors $\bfv\in \ZZ^n$. 
When $q$ is prime the estimation of these sums in \cite{hb10,hooley9} uses the full force of Deligne's resolution of the Weil conjectures \cite{Del}. 
This is one of several places where tight control over the singular locus of the cubic hypersurface $X$ enters the argument.
Our main idea can be summarised as using
the stratification theory of exponential sums
over finite fields
as a proxy for the breakdown in control over 
$S_q(\bm{v})$ whenever $X$ has a large-dimensional singular locus.

Works of Fouvry, Katz and Laumon \cite{KL85,FK01}
produce a stratification that is nearly sufficient for our purposes. Their construction is based on a general induction procedure and produces strata
\[\A_\ZZ^n = Z_0 \supset \cdots \supset Z_{n}=\{\mathbf 0\} ,
\]
with each $Z_j\subset\A_\ZZ^n$ of codimension at least $j$, with the property that  $S_p(\bm{v})\ll p^{\frac{n+j}{2}}$ if 
$\bm{v} \bmod{p}$ belongs to $\A^n(\FF_p)\setminus Z_{j}(\FF_p)$.
In our projective setting,
a different construction,
based on the theory of singular support
developed by Beilinson \cite{Bei},
Hu--Yang \cite{HY}
and Raskin--Smith~\cite{RS},
gives sharper control on the strata $Z_j$,
enabling us to use classical projective duality
to show that if $X(\QQ)$ were empty,
then the dimension of the Zariski closure of $Z_j(\QQ)$ in $Z_j$
would be strictly less than the standard bound for $\dim Z_j$ itself.
A similar idea with smooth $\QQ$-points on dual varieties
was previously used in the circle method for smooth cubic hypersurfaces in $9$ variables
by Hooley \cite[Lemma~2]{hooley9},
before he found a more robust argument avoiding any assumption on the dual variety
\cite{hooley9.3}.
The perspective of Raskin--Smith \cite{RS},
connecting the singular support to Gysin maps of singular varieties,
is convenient for us.
However, one could alternatively adopt a more Fourier-theoretic approach
based on Laumon~\cite{Laumon}, Sawin~\cite{Sawin} and Zhou~\cite{Zhou}.

It is well known that homogeneity of $F$ gives
$S_q(u\bfv)=S_q(\bfv)$ for all $u\in(\ZZ/q\ZZ)^\ast$,
whence the Chinese remainder theorem gives
\begin{equation}
 S_{q_1q_2}(\bfv)=S_{q_1}(\bfv)S_{q_2}(\bfv),
 \label{eq:q-exact-CRT}
\end{equation}
whenever $\gcd(q_1,q_2)=1$.
The circle method will lead us to estimate sums
of the shape
\begin{equation}\label{Eq: EnterTheDelta}
\sum_{\bfv \in \ZZ^n} \sum_q S_q(\bfv) I_q(\bfv)
= \sum_{\bfv \in \ZZ^n} \sum_{\textnormal{cubefull }r} S_r(\bfv)
\sum_{\substack{\textnormal{cube-free }d \\ \gcd(d,r)=1}} S_d(\bfv) I_{dr}(\bfv),
\end{equation}
where we temporarily write $I_q(\bfv)$
for a certain oscillatory integral.
See \eqref{eq:defERROR2} for the specific ranges involved.
Our overall strategy is then to combine precise stratifications for $|S_d(\bfv)|$,
on the one hand,
with weaker but more robust estimates for $|S_r(\bfv)|$ and $|I_{dr}(\bfv)|$.

The stratifications are given
in Lemma~\ref{lem:prime-square-stratification} for prime squares,
and in Appendix~\ref{sec:n10-singular-support} for primes.
A hybrid level-set decomposition for $|S_d(\bfv)|$,
based on the global stratum of $\bfv$
and the local strata-based factorisation type of $d$,
is given in \S\ref{sec:MarkV-construction}.
Moreover, in \S\ref{sec:MarkV-construction},
the level-set estimates are uniform over adelic boxes,
which turns out to be crucial for our approach.

Following ideas of Heath-Brown~\cite{hb10}, in \S\ref{sec:squarefull-average} we give
estimates for $|S_r(\bfv)|$ on average over $\bfv$ when $r$ is cubefull that depend roughly on $\bfv\bmod{r^{1/2}}$. However, for most of the paper, it turns out to be useful for us to use weaker estimates that depend on less information but are more sensitive to the stratification we use.
More precisely, they depend only roughly on $\bfv\bmod{r^{1/3}}$,
an insight that we formalise in \S\ref{sec:smith-work}
away from a finite set of primes $p$.
The remaining primes are handled separately in \S\ref{sec:fixed-primes}
by a localisation trick.

Moreover, the integral estimates in \S\ref{sec:intro-circle} tend to control
a large shifted box of vectors $\bfv$ at once,
not just a single real point $\bfv$.
An adelic box $\mathcal{A}$ of vectors $\bfv$,
on which $|S_r(\bfv)|$ and $|I_{dr}(\bfv)|$ are controlled,
is then fertile ground for stratification methods to take place:
bounds via point counts,
sieves,
and traces.
Most of our cubefull and integral estimates
may be fruitfully restricted to various natural subsets $\mathcal{S} \subset \mathcal{A}$,
giving bounds that are sensitive to the density of $\mathcal{S}$ in $\mathcal{A}$.
A notable exception is the Poisson summation based squarefull bound (Lemma~\ref{lem:onion}), which is reasonably strong in short boxes
but often weak in small subvarieties thereof.
Our work will require combining different estimates in different ranges.

In Lemma \ref{LEM:433-incidence-bound} we prove Theorem
\ref{thm:sing} and obtain further  information about the singular loci of
our cubic hypersurface $X$
and its hyperplane sections,
assuming that $X(\QQ) = \emptyset$.
This leads to useful exponential sum estimates via work of Hooley, Katz, et al.,
as we recall in \S\ref{sec:pointwise}.

In \S\ref{sec:MarkV-construction} we prove the required
cube-free estimates in ten variables,
using a stratified geometric sieve, Proposition~\ref{lem:stratified-composite-sieve}.
In \S\ref{sec:n10-mixed-moduli},
this is combined with the corresponding
ten variable cubefull estimates of \S\ref{sec:smith-work}.
The proof of Theorem~\ref{thm:main} is then completed in \S\ref{sec:intro-circle}
by incorporating the integral factor,
as well as a main term calculation based on work of Bernert \cite{Ber25}.

\subsection*{Final remarks}
The role of stratifications can be illustrated already
for prime moduli.
Generically, we have $q\sim B^{3/2}$ and $|\bm{v}|\ll B^{1/2}$  in \eqref{Eq: EnterTheDelta}. Restricting to primes, our goal is then to show that
\[
\mathcal{T}(B) := \sum_{B^{3/2}<p\leq 2B^{3/2}}
\sum_{0<|\bm{v}|\ll B^{1/2}}
|p^{-n} S_p(\bfv)|
= o(1).
\]
If we had square-root cancellation $S_p(\bfv) \ll p^{(n+1)/2}$ everywhere,
then we would immediately get
$$\mathcal{T}(B) \ll B^{(3+n)/2} (B^{3/2})^{(1-n)/2}
= B^{(9-n)/4} = o(1)$$
for $n\ge 10$.
However, square-root cancellation does not hold pointwise.
The stratification
controls the extent to which failures of square-root cancellation can hold.
We can use the Schwartz--Zippel lemma to estimate the contribution from vectors 
$\bm{v}\in Z_j(\ZZ)$ in an expanding region, but this turns out not to be  enough. Rather, we shall  apply results towards the dimension growth conjecture, in the form given by Pila \cite{Pila} and Vermeulen \cite{Vermeulen}. For this we require more information on the possible degrees of the components of $Z_j$,
which is achieved
in \S\ref{sec:n10-second-moments}
via a second moment argument  that is inspired by work of
Bonolis--Pierce--Woo \cite{woo}.
Additionally,  by using projective duality and the  theory of singular support, we are able to show 
in Appendix~\ref{sec:n10-singular-support}
that rational points on $Z_j$ actually belong to a 
subvariety of codimension at least $j+1$,
under the assumption $X(\QQ)=\emptyset$. 

Many questions remain open, even for cubic forms in $n\ge 10$ variables over $\QQ$.
We have been able to show that for any such cubic form,
either geometric methods produce a $\QQ$-point,
or else the circle method produces a localised asymptotic formula.
It would be interesting to determine the weakest conditions
under which an asymptotic formula can be proven;
does goodness suffice?
A related question is that of obtaining lower bounds for point counts and a question of Wooley,
for which we refer to recent work of Kumaraswamy and Rome \cite{nick}.

\subsection*{Acknowledgments}

We are very grateful to
Roger Heath-Brown and 
Simon Rydin Myerson 
for helpful conversations about 
Hessians and 
tangent spaces,
and to Julian Demeio,
Matthew Hase-Liu and
Adeel Khan
for helpful discussions about
the technical apparatus of Appendix~\ref{sec:n10-singular-support}.
Parts of this work were undertaken during research visits 
to the {\em Institute for Science and Technology Austria} in Klosterneuburg (June 2024), 
 the {\em Centro di ricerca Matematica Ennio De Giorgi} in Pisa (January 2026), and 
the {\em Institute of Mathematics Academia Sinica}
in Taipei (June 2026). 
We are very grateful to these institutions for providing excellent working conditions. 
VYW was supported by
the European Union's Horizon~2020 research and innovation programme under the Marie Sk\l{}odowska-Curie Grant Agreement No.~101034413,
and by the National Science and Technology Council Project Grant 114-2115-M-001-010-MY2.

\subsection*{Supplementary Lean code}

The interested reader may consult
\begin{center}
\url{https://github.com/wangyangvictor/cubic-lean-mid}
\end{center}
for a Codex Lean auto-formalisation of Theorem~\ref{thm:main} and its proof,
modulo five inputs from the literature
 in ``elementary'' form.
This builds on the vast existing Lean infrastructure,
including the strong foundations of Mathlib in areas such as commutative algebra.

\subsection*{AI statement}

The structure of the proof and most of the new ideas in this paper are due to the authors and were developed from  June 2024 until September 2026.  However, a key ingredient in our work is a suitable upper bound for the dimension of the incidence variety
\[
I = \{(\bm{x},\bm{y})\in \A^n\times \A^n\colon H(\bm{x})\bm{y}=0\},
\]
with the inequalities $n\leq \dim I\leq 2n$ being essentially trivial.  When the cubic hypersurface $X$ is smooth it is straightforward to prove that 
$\dim I=n$, a fact that underpins  work of Heath-Brown \cite[Lemma~1]{hb10}.
In general, 
assuming only that $X(\QQ)=\emptyset$, we were able to prove the upper bound 
$\dim I\leq \lfloor(3n-2)/2\rfloor$. This was improved to 
$\dim I\leq \lfloor(3n-3)/2\rfloor$ in a private communication from Heath-Brown, which we believe would be enough to handle cubic forms in at least eleven variables, once combined with the other ingredients in our paper.  After prompting 
ChatGPT 6.0, we were led to the improved  bound 
$\dim I \le \lfloor (4n-3)/3\rfloor$ in
Lemma~\ref{LEM:433-incidence-bound}, under the same assumption 
$X(\QQ)=\emptyset$.
ChatGPT 6.0  suggested an approach via the rational Hilbert--Mumford criterion of Kempf~\cite[Corollary~4.3]{Kem78}, which we discuss below in Remark~\ref{RMK:GIT-view}.
Our work involves combining many different estimates that are effective in different ranges of parameters. 
ChatGPT 5.5, 5.6, 6.0 and 6.1 have been helpful with this process, 
as well as in 
proofreading our work, and conducting the Lean auto-formalisation of this work.  
We are very grateful to Mark Sellke at OpenAI for giving us access to the Pro models in 
May 2026.

\section{Technical lemmas}
\label{sec:technical-lemmas}

\subsection{Counting points over finite fields}

The following result provides point counts for cubic hypersurfaces over finite fields, under various assumptions.

\begin{lemma}\label{lem:Fq}
Let $\FF_q$ be a finite field of order $q$ and
let $X\subset \PP^{N}$ be a geometrically integral cubic hypersurface defined over $\FF_q$,
where $N\ge 2$.
If $\sigma=\dim \sing(X)\in \{-1,\dots,N-2\}$, where $\sing(X)$ is the singular locus of $X$, then
\begin{equation*}
\#X(\FF_q)-\frac{q^N-1}{q-1}
\ll_N q^{(N+\sigma)/2}.
\end{equation*}
Moreover,
$$
\#X(\FF_q)-\frac{q^N-1}{q-1}\ll_N
\begin{cases}
q^{N-2} & \text{ if $N\geq 3$ and $X$ is not a cone,}\\
q^{N-3/2} &\text{ otherwise.}
\end{cases}
$$
\end{lemma}

\begin{proof}
The first result can be extracted from
Theorem~1 of \cite{Hoo91}.
The second is the main result of 
\cite{Bro13} and the last follows from the  Lang--Weil estimate \cite{LW}.
\end{proof}

\begin{rem}
Although Lemma~\ref{lem:Fq} will suffice for the present paper, 
it might be possible to obtain sharper results for $N\ge 4$
by using \cite[Theorem~5.5]{Wan23}. 
\end{rem}

\subsection{Basic geometric facts}

Let us record some useful facts about geometric integrality,
cones, and singular loci.
The first part of \cite[Lemma~1]{Bro10} states the following.

\begin{prop}
\label{Bro-integral-cone}
Let $N\ge 3$.
If $X \subset \PP^N_\QQ$ is a cubic hypersurface
with $X(\QQ) = \emptyset$,
then $X$ is geometrically integral and is not a cone.
\end{prop}

\begin{rem}
\label{field-of-maximal-vertex}
It will be useful to know how to compute the
maximal vertex of a conical cubic hypersurface.
If $C(\bfy)$ is a cubic form in $N$ variables
over a perfect field $K$,
then its maximal vertex over $\bar{K}$
coincides with the base change of its maximal vertex over $K$.
This follows from Galois descent,
since the maximal vertex over a field $L/K$ corresponds to the $L$-vector space
\[
M(L) = \{\bfz\in L^N: C(\bfy+t\bfz) - C(\bfy) = 0\in L[\bfy,t]\},
\]
which satisfies $M(\bar{K})^{\Gal(\bar{K}/K)} = M(K)$.
This is particularly clean if $K$ has characteristic $\neq 2,3$,
in which case by Taylor expansion the condition in $M(L)$
is equivalent to $H_C(\bfz)=\bm{0}$.
\end{rem}

Thanks in part to Lemma~\ref{lem:Fq},
singular loci of hyperplane sections of a cubic hypersurface will be relevant.
The following result is best interpreted scheme-theoretically.

\begin{thm}[Zak]
\label{Thm: ZaksBoundSingLocus}
Let $Y\subset \PP^N$ be a geometrically integral hypersurface of degree $\ge 2$ over a field.
Let $H\subset \PP^N$ be a hyperplane.
Then $\dim \Sing(Y\cap H)\leq \dim\Sing(Y)+1$.
In particular, if $\dim Y - \dim\Sing(Y) \ge 4$,
then $Y\cap H$ is geometrically integral.
\end{thm}

\begin{proof}
This follows from Zak's bound
recorded in Theorem~2 of Katz's appendix to \cite{Hoo91}.
\end{proof}

For further results in this direction,
we refer the reader to Zak \cite[Theorem~1.7]{Zak93}.

\subsection{A geometric sieve in adelic boxes}

We need to count points uniformly in adelic boxes.
It will be convenient to have the following version of Noether normalisation at hand.

\begin{lemma}
\label{Le: NoetherNormalisation}
Let $Y\subset\A_\ZZ^N$ be a closed subscheme whose
generic fibre $Y_{\QQ}\subset\A_\QQ^N$
has dimension $0\le d\le N-1$.
Let $\pi\colon \A^N \to \A^d$ be the projection onto the first $d$ coordinates.
There exist
a matrix $A \in \Mat_{N\times N}(\ZZ) \cap \GL_N(\QQ)$,
an integer $C\ge 1$,
and polynomials $F_j\in \ZZ[x_1,\dots,x_d,x_{d+j}]$ for $1\le j\le N-d$,
such that the following hold:
\begin{enumerate}
\item The composite map
$Y_{\QQ} \xrightarrow{A} \A_\QQ^N \xrightarrow{\pi} \A_\QQ^d$
is finite.

\item For each $1\le j\le N-d$,
we have $F_j = Cx_{d+j}^C + O_{x_1,\dots,x_d}(x_{d+j}^{C-1})$,
and the polynomial $F_j(A\bfx)\in \ZZ[\bfx]=\ZZ[x_1,\dots,x_N]$
vanishes on $Y$.
\end{enumerate}
\end{lemma}

\begin{proof}
Since the field $\QQ$ is infinite,
it is a well-known consequence of
Noether normalisation
that there exists a dense open subvariety $U \subset \GL_N$ over $\QQ$
such that (1) holds for all $A\in U(\QQ)$,
as remarked in \cite[p.~284]{Eis95}.
Since $\Mat_{N\times N}(\ZZ) \cap \GL_N(\QQ)$ is Zariski dense in $\GL_N(\Qbar)$,
we can satisfy (1) with a matrix $A\in \Mat_{N\times N}(\ZZ) \cap \GL_N(\QQ)$.
It follows from (1) that each $x_{d+j}$ satisfies a monic non-constant polynomial relation in $\QQ[x_1,\dots,x_d][x_{d+j}]$ on the variety $A(Y_{\QQ})$.
After multiplying each relation by a suitable monomial in $x_{d+j}$, 
we are led to (2),
since every element of the ideal $I(Y_{\QQ}) = I(Y) \otimes \QQ$
is a nonzero rational scalar times an element of $I(Y)$.
\end{proof}

Lemma~\ref{Le: NoetherNormalisation} immediately gives
the following variant of the Schwartz--Zippel bound.

\begin{lemma}
\label{LEM:basically-zippel}
Let $Y\subset\A_\ZZ^N$ be a closed subscheme whose
generic fibre
has dimension $0\le d\le N$.
Let $m\ge1$ be an integer.
Let $\bfb\in\ZZ^N$,
$L\ge1$ and $\bfv_0\in\RR^N$.
Then
\begin{equation*}
\#\{\bfv\in Y(\ZZ):|\bfv-\bfv_0|\le L,\ 
\bfv\equiv\bfb\bmod m\}
\ll_Y (1+L/m)^d.
\end{equation*}
\end{lemma}

\begin{proof}
This is trivial if $d=N$.
If $d\le N-1$, use Lemma~\ref{Le: NoetherNormalisation}.
\end{proof}

The following version of the Ekedahl sieve
for large square-free moduli
also becomes available.
To ease notation we shall henceforth adopt the common notation $A\sim B$ to mean that 
$B<A\leq 2B$ for any $A,B>0$.

\begin{lemma}
\label{Le: StratGeoSieveStep1}
Let $Y\subset\A_\ZZ^N$ be a closed subscheme whose
generic fibre
has dimension at most $d$,
where $0\le d\le N$.
Let $m\ge1$ be an integer.
Let $\bfb\in\ZZ^N$,
$L\ge1$ and $\bfv_0\in\RR^N$.
If $S \gg 1 + L/m$, then
\begin{equation*}
\sum_{\substack{q\sim S \\
\gcd(q,m)=1}}
\sum_{\substack{\bfv\in \ZZ^N \setminus Y(\ZZ) \\
|\bfv-\bfv_0|\le L \\
\bfv\equiv\bfb\bmod m \\
\bfv\bmod q\in Y(\ZZ/q\ZZ)}}
\mu^2(q)
\ll_{Y,\ve} (LS+|\bfv_0|)^\ve (1 + L/m)^{d+1}.
\end{equation*}
\end{lemma}

\begin{proof}
Since the ideal $I(Y)$ is finitely generated,
the divisor bound on a suitable polynomial in $I(Y)$
for each vector $\bfv\in \ZZ^N \setminus Y(\ZZ)$
implies that
\begin{equation*}
\sum_{\substack{q\sim S \\
\gcd(q,m)=1}}
\sum_{\substack{\bfv\in \ZZ^N \setminus Y(\ZZ) \\
|\bfv-\bfv_0|\le L \\
\bfv\equiv\bfb\bmod m \\
\bfv\bmod q\in Y(\ZZ/q\ZZ)}}
\mu^2(q)
\ll_{Y,\ve}
\sum_{\substack{\textnormal{admissible }\bfv\in \ZZ^N \setminus Y(\ZZ) \\
|\bfv-\bfv_0|\le L \\
\bfv\equiv\bfb\bmod m}} (L+|\bfv_0|)^\ve,
\end{equation*}
where we call a vector $\bfv\in \ZZ^N$
admissible if there exists a square-free integer $q\sim S$
with $\gcd(q,m)=1$
such that $\bfv\bmod q\in Y(\ZZ/q\ZZ)$.
Lemma~\ref{Le: StratGeoSieveStep1} will follow if we can prove that
\begin{equation}
\label{n10:translated-composite-sieve}
f(Y) := \sum_{\substack{\textnormal{admissible }\bfv\in \ZZ^N \\
|\bfv-\bfv_0|\le L \\
\bfv\equiv\bfb\bmod m}} 1
\ll_{Y,\ve} (LS+|\bfv_0|)^\ve (1 + L/m)^{d+1}.
\end{equation}
We include vectors $\bfv\in Y(\ZZ)$ in \eqref{n10:translated-composite-sieve}
so that $f(Y)$ is monotone, in the sense that if $Y \subset Y'$ then $f(Y) \le f(Y')$.
The inequality \eqref{n10:translated-composite-sieve}
is trivial if $d=N$, so suppose that $d\le N-1$.
We may assume $\dim Y_\QQ = d$,
by adding a linear space of generic dimension $d$
to $Y$ if necessary.
After changing variables as in Lemma~\ref{Le: NoetherNormalisation},
and further enlarging $Y$ if necessary,
we may assume in \eqref{n10:translated-composite-sieve}
that $Y$ is the simultaneous zero locus
of integer polynomials of the form
$$F_j(x_1,\dots,x_d,x_{d+j}) = Cx_{d+j}^C + O_{x_1,\dots,x_d}(x_{d+j}^{C-1}),$$
for $1\le j\le N-d$,
where $C\ge 1$.
The contribution of $\bfv\in Y(\ZZ)$ to the left-hand side of
\eqref{n10:translated-composite-sieve}
is negligible by Lemma~\ref{LEM:basically-zippel}.
If $\bfv\notin Y(\ZZ)$, then let $j_*$ be the minimal index $j$
for which $F_j(v_1,\dots,v_d,v_{d+j}) \ne 0$.
Once $v_1,\dots,v_d,v_{d+j_*}$ are chosen,
there are
$O_{Y,\ve}((L+|\bfv_0|)^\ve)$ choices for $q\mid F_{j_*}(v_1,\dots,v_d,v_{d+j_*})$,
by the divisor bound.
For each $j\ne j_*$, the congruence $F_j(v_1,\dots,v_d,y_{d+j}) \equiv 0\bmod{q}$
has at most $\gcd(C,q) C^{\omega(q)}=O(q^\ve)$ solutions $y_{d+j}\bmod{q}$,
by the Chinese remainder theorem.
Each residue $y_{d+j}\bmod{q}$ has $O(1 + L/(qm)) = O(1)$
integral lifts $v_{d+j}$ satisfying
$v_{d+j}\equiv y_{d+j}\bmod q$ and
$v_{d+j}\equiv b_{d+j}\bmod m$,
since $\gcd(q,m)=1$ and $q\sim S \gg L/m$.
Thus the contribution to \eqref{n10:translated-composite-sieve}
from the index $j_*$
is $$\ll_{Y,\ve} \sum_{v_1,\dots,v_d,v_{d+j_*}} (L+|\bfv_0|)^\ve S^\ve
\ll_\ve (LS+|\bfv_0|)^{2\ve} (1+L/m)^{d+1}.$$
Summing over $1\le j_*\le N-d$ completes the proof of \eqref{n10:translated-composite-sieve}.
\end{proof}

\subsection{A geometric sieve for stratifications}

We now give a stratified geometric sieve,
concerning an arbitrary finite collection of closed subschemes.

\begin{prop}\label{lem:stratified-composite-sieve}
Let $Y_0,Y_1,\ldots,Y_s\subset\A_\ZZ^N$ be closed subschemes
whose generic fibres have dimensions at most $d_0,d_1,\dots,d_s$,
respectively.
Suppose that
\[
 N\ge d_0\ge d_1\ge\cdots\ge d_s\ge0.
\]
Let $\mathcal U \subset Y_0(\ZZ)$ and $0 \le \alpha \le d_0$.
Suppose that
uniformly over $\bfv_0\in \RR^N$,
$L\ge 1$,
$\bfa\in\ZZ^N$
and integers $q\ge 1$,
\begin{equation}
\label{eq:stratified-progression-hypothesis}
\#\{\bfv\in \mathcal U:
|\bfv-\bfv_0|\le L,\;
\bfv\equiv\bfa\bmod q\}
\ll_{\mathcal U,\ve}
(L+|\bfv_0|)^\ve (1+L/q)^\alpha.
\end{equation}
Let $m\ge 1$ be an integer.
Let $\bfb\in\ZZ^N$,
$L\ge1$ and $\bfv_0\in\RR^N$.
Let
$${\textstyle \mathcal C
:= \left\{\bfv\in \mathcal U \setminus \bigcup_{i=1}^s Y_i(\ZZ):
|\bfv-\bfv_0|\le L,\;
\bfv\equiv\bfb\bmod m\right\}.}$$
For $R_1,\ldots,R_s \ge 1/2$, let
\[
\mathcal N(\bm R)
:= \sum_{\substack{q_i\sim R_i \\
\gcd(q_i,q_j)=1\ (i\ne j) \\ \gcd(q_i,m)=1}}
\sum_{\bfv\in\mathcal C}
\prod_{i=1}^s \mu^2(q_i)
\bm1_{\bfv\bmod q_i\in Y_i(\ZZ/q_i\ZZ)}.
\]
Then
\begin{equation}\label{eq:stratified-composite-sieve}
\mathcal N(\bm R)
\ll_{\mathcal U,Y_0,\dots,Y_s,\alpha,\ve}
(LR_1\cdots R_s + |\bfv_0|)^\ve
\left(1
+ \frac{(L/m)^\alpha}{\prod_{i=1}^s R_i^{\alpha-d_i-1}}
+ \sum_{j=1}^s \frac{(L/m)^{d_j+1}}{\prod_{i=j+1}^s R_i^{d_j-d_i}}
\right),
\end{equation}
where empty products are interpreted as $1$.
\end{prop}

\begin{proof}
The contribution to $\mathcal N(\bm R)$
from any individual vector $\bfv\in\mathcal C$
is $\ll (L+|\bfv_0|)^\ve$,
by the divisor bound.
Thus \eqref{eq:stratified-composite-sieve} is trivial if $L \le m$.
We may therefore assume $L>m$.

By the Lang--Weil bound,
we have $\#Y_i(\FF_p) \ll p^{d_i}$ for all $p$.
If $R_1\cdots R_s\le L/m$,
then there are $\ll \prod_i q_i^{d_i+\ve}$
available residue classes $\bfv\bmod{q_1\cdots q_s}$.
Applying \eqref{eq:stratified-progression-hypothesis}
with $q = mq_1\cdots q_s \ll L$ gives
\begin{equation*}
\frac{\mathcal N(\bm R)}{(LR_1\cdots R_s + |\bfv_0|)^\ve}
\ll \sum_{q_i\sim R_i}
(L/q)^\alpha
\prod_i q_i^{d_i}
\ll \frac{(L/m)^\alpha}{\prod_{i=1}^s R_i^{\alpha-d_i-1}},
\end{equation*}
which is satisfactory.
If $R_1\cdots R_s > L/m$, choose $1\le j\le s$ such that
\[
R_{j+1}\cdots R_s \ll L/m \ll R_jR_{j+1}\cdots R_s.
\]
This is possible because $R_1,\dots,R_s \gg 1$ and $1 \ll L/m \ll R_1\cdots R_s$.
Given $\bfv$, the number of available moduli $q_i$
is $\ll (L+|\bfv_0|)^\ve$ for each $i<j$,
by the divisor bound.
On the other hand, there are $\ll \prod_{i>j} q_i^{d_i+\ve}$
available residue classes $\bfv\bmod{\prod_{i>j} q_i}$.
Lemma~\ref{Le: StratGeoSieveStep1},
with $M := m\prod_{i>j} q_i \gg L/R_j$ in place of $m$,
with $Y := Y_j$ and with $S := R_j \gg 1 + L/M$, gives
\begin{equation*}
\frac{\mathcal N(\bm R)}{(LR_1\cdots R_s + |\bfv_0|)^\ve}
\ll \sum_{q_i\sim R_i\ (i>j)} (1+L/M)^{d_j+1}
\prod_{i>j} q_i^{d_i}
\ll \sum_{q_i\sim R_i\ (i>j)} (L/M)^{d_j+1}
\prod_{i>j} q_i^{d_i}
\end{equation*}
since $L/M \gg 1$.
This is $\ll (L/m)^{d_j+1} \prod_{i>j} R_i^{d_i-d_j}$, which is satisfactory.
\end{proof}

\subsection{Oscillatory integrals}

We proceed by  looking at cubic  weighted exponential integrals. 
Let $w: \RR^n \rightarrow \RR_{\geq 0}$ be a smooth weight function
 of compact support.
Let $P\geq 1$ and 
consider the  exponential integral
\begin{equation}\label{eq:defIq}
 I(\theta;\bfbe)=\int w(\bfx/P)e(\theta F(\bfx)-\bfbe\cdot \bfx)\mathrm{d}\bfx,
\end{equation}
for arbitrary $\theta \in \RR$ and $\bfbe\in \RR^n.$ We extract the following result from \cite{41}.

\begin{lemma}\label{lem:Iq'}
Let $w(\bfx)=\omega((\bm{x}_0-\bfx)/\ell)$, where
$\omega\in C_c^\infty(\RR^n)$ is supported on $[-1,1]^n$,
and $\ell>0$ and $\bm{x}_0\in\RR^n$ are fixed.
Let $\ve>0$ and let $N\geq 1$.
Assume $1\le q\le P^2$ and $|\theta| \le q^{-1}P^{-1}$.
Then we have
$$
I(\theta;q^{-1}\bfv)
\ll_{F,\ve,N} |\bfv|^{-N},
$$
if $|\bfv|\gg P^3$. Alternatively, when
$|\bfv|\ll P^3$, we have
$$
I(\theta;q^{-1}\bfv)
\ll_{F,\ve,N}
P^{-N}+\meas\left\{\bfx\in\RR^n:
|\bfx/P-\bm{x}_0|\le 2\ell,\;
|q\theta\nabla F(\bfx)-\bfv|\le P^\ve V_0
\right\},
$$
where $V_0=qP^{-1}\max\{1,\sqrt{|\theta|P^3}\}$.
Finally, we have
$$
I(\theta;q^{-1}\bfv)
\ll_{F,\ve,N} P^{-N},
$$
if $|\bfv|\ge P^{2\ve}qP^{-1}\max\{1,|\theta|P^3\}$.
The implied constants above may depend on $\omega$, $\ell$ and $\bm{x}_0$
in addition to $F,\ve,N$.
\end{lemma}

\begin{proof}
We may assume that $P$ is sufficiently large, the bounded $P$ case
following from the trivial bound and integration by parts.
The first two parts directly follow from  taking $g=F$ and $H=O(1)$ in
\cite[Lemma~9 and its proof]{41}.
(Although \cite[Lemma~9]{41} is written without the condition $|\bfx/P-\bm{x}_0|\le 2\ell$,
it is clear from the proof that we may include it, since $w(\bfx/P)$
is supported on the strictly smaller box $|\bfx/P-\bm{x}_0|\le \ell$.)
For the final part, it suffices to show that the volume is over the empty  set in the second estimate 
for $I(\theta;q^{-1}\bfv)$,  
if $|\bfv|\geq P^{2\ve} qP^{-1} \max\left\{1,|\theta |P^3\right\}$. 
Note that $q\theta \nabla F(\bfx)\ll q|\theta|P^2$ if $|\bfx/P-\bm{x}_0|\le 2\ell$.  If 
$|\theta |P^3<1$, then the claim is obvious.  
If $|\theta |P^3\geq 1$, then the claim is also clear, since  then 
$|\bfv|\geq P^{2\ve} qP^{-1}\cdot |\theta |P^3=q |\theta |P^{2+2\ve}\gg q |\theta |P^{2}$, which would imply the contradictory upper bound
$|\bfv|\ll P^\ve V_0 = P^\ve  qP^{-1} \sqrt{|\theta |P^3}$.
\end{proof}

\section{Geometry of cubic forms}
\label{sec : geom_cubic_hyper}

In this section we collect some facts about the geometry of cubic forms over $\QQ$.
In later sections, we will be able to spread out much information from $\QQ$
to finite fields $\FF_p$,
for all but finitely many primes $p$.
Let us now introduce the main objects of interest.

Let $F\in \QQ[x_1,\dots,x_n]$ be a nonzero cubic form in $n\ge 2$ variables.
We make no further assumptions about $F$,
except where specified otherwise.
This level of generality will be useful for
a slicing argument in \S\ref{sec:dual-tangential-geometry}.
Let $$H(\bm{x})\in \Mat_{n\times n}(\QQ[x_1,\dots, x_n])$$
be the Hessian matrix of $F$,
as in \S\ref{sec:intro},
formed from the second partial derivatives of $F$.
This gives a vector-valued bilinear form
$H(\bm{x})\bm{y}=H(\bm{y})\bm{x}$ on $\QQ^n\times \QQ^n$.

Let $X\subset \PP^{n-1}$ be the closed subscheme $F=0$ over $\Qbar$.
In this section, we work with schemes
in order to avoid edge cases concerning the definition of $\sing X$,
and we work over $\Qbar$
in order to avoid cumbersome base-change notation.
Thus $X$ is a $\Gal(\Qbar/\QQ)$-stable scheme over $\Qbar$,
and $X(\QQ)$ consists of the $\Gal(\Qbar/\QQ)$-invariant $\Qbar$-points of $X$.
Let
\begin{equation}
\label{Eq: Def.hessian.Incidence}
I := \{(\bm{x},\bm{y})\in \A^n\times \A^n\colon H(\bm{x})\bm{y}=0\}
\end{equation}
be the Hessian incidence correspondence;
its scheme structure is of no matter to us.
On the other hand, given an integer $0\leq r \leq n$, define
\begin{equation}
\label{EQ:Hessian-Vr-locus}
V_r := \{\bm{x}\in \A^n\colon \rank H(\bm{x})\leq r\},
\end{equation}
with the \emph{reduced} scheme structure.
The tangent space $T_{\bfx}V_r$ at a point $\bfx\in V_r$ is thus defined
in terms of the \emph{radical} of the determinantal ideal associated to $V_r$.
Explicitly, $T_{\bfx}V_r$ may be identified with the set of $\Qbar[\ve]/(\ve^2)$-valued
points of $V_r$
whose reduction modulo $\ve$ is $\bfx$.

We define $r(F)$ to be the maximum of the rank of $H(\bm{x})$ as $\bm{x}$ ranges over $X(\Qbar)$. 
Building on work of Zak \cite{myguyzak},
we will show in Lemma~\ref{LEM:my-guy-slop} that $r(F) \ge 2\lceil n/3\rceil$
if $X(\QQ)=\emptyset$.
Indeed, some results in this section will only be proven
under the assumption $X(\QQ)=\emptyset$,
in which case this assumption will always be clearly stated.

The next lemma is a geometric reformulation of a key identity
\cite[\S3, equation~(3.3)]{BDE12}
from Davenport's bad-form analysis \cite[proof of Lemma~3]{Dav63}.
This formulation has been observed before by
Rydin Myerson \cite[\S1.4]{RM19}.
For completeness, we supply a proof.

\begin{lemma}
\label{LEM:dav-bad-tangent-space}
Let $0\le r\le n$.
Put $V_{-1}:=\emptyset$.
If $\bfx\in V_r \setminus V_{r-1}$,
$\bfv\in T_{\bfx}V_r$,
and $\bfy,\bfz\in\ker H(\bfx)$,
then $$\bfy^tH(\bfv)\bfz=0.$$
\end{lemma}

\begin{proof}
If $r=n$, the claim is trivial. If $r=0$, then $V_0$ is a
linear space, so $\bfv\in V_0$ and $H(\bfv)=0$.
Thus assume $0<r<n$.
After a linear change of variables, we may suppose that
$$
H(\bfx)=
\begin{pmatrix}
M&0\\
0&0
\end{pmatrix},
\qquad \det M\ne0,
$$
where $M$ is an $r\times r$ matrix.
For $i,j>r$, let $\Delta_{ij}(\bfw)$ be the minor of $H(\bfw)$
formed from rows $1,\dots,r,i$ and columns $1,\dots,r,j$.
Since $\Delta_{ij}$ vanishes on $V_r$ and
$\bfv\in T_{\bfx}V_r$, its differential at $\bfx$ annihilates
$\bfv$. On the other hand, linearity of $H$ gives
$$
\Delta_{ij}(\bfx+\ve\bfv)
\equiv \ve\det(M)H(\bfv)_{ij}\pmod{\ve^2}.
$$
Hence $H(\bfv)_{ij}=0$ for all $i,j>r$.
Since $\bfy,\bfz\in\ker H(\bfx)$ have their first $r$
coordinates zero, this proves $\bfy^tH(\bfv)\bfz=0$.
\end{proof}

\begin{rem}
Let us briefly sketch a
more conceptual,
but less elementary,
proof of Lemma~\ref{LEM:dav-bad-tangent-space}.
By general machinery or explicit calculation,
the vector spaces $\ker H(\bfx)$
form a vector bundle over $\bfx\in V_r \setminus V_{r-1}$.
Let $\boldsymbol{\zeta}(\bfw)$ be a local section extending $\bm{z}$ for $\bfw$ in a sufficiently small neighbourhood of $\bfx$, with
$H(\bfw)\boldsymbol{\zeta}(\bfw)=\bm0$
and $\boldsymbol{\zeta}(\bfx)=\bfz$.
K\"{a}hler differentiation at $\bfx$ in the direction $\bfv$,
using the product rule, gives $H(\bfv)\bfz+H(\bfx)d\boldsymbol{\zeta}_{\bfx}(\bfv)=\bm0$.
Multiplying by $\bfy^t$ gives $\bfy^tH(\bfv)\bfz=0$,
since $\bfy^tH(\bfx)=0$.
\end{rem}

The following definition formalises a notion that we already met in the introduction. 

\begin{defin}
\label{defn : good}
We say that the cubic form $F$ is {\em good} if
for any $H \geq 1$ and any $\ve>0$ we have the upper bound   
\[
\#\{\bfz\in\mathbb{Z}^n: |\bfz|\leq H,~ \rank(H(\bfz))=r\}
\ll_{F,\varepsilon} H^{r+\varepsilon},
\]
for each $0\leq r\leq n$,
where the implied constant only depends on $F$ and $\ve$.
\end{defin}

It follows from work of Davenport \cite{Dav63} that if $X(\QQ) = \emptyset$,
then $F$ is good.
This implication is valid for any $n\ge 2$.  In particular, $\det H$ is not the zero polynomial. 
A similar idea to Davenport's leads to the following result.

\begin{lemma}
\label{LEM:my-guy-slop}
Assume $X(\QQ)=\emptyset$.
If $n\ge 2$, then $r(F) \ge 2 \lceil n/3 \rceil$.
\end{lemma}

\begin{proof}
Since $X(\QQ)=\emptyset$, the polynomial $F$ is irreducible over $\QQ$.
It follows from Zak \cite[Proposition~4.2 and Corollary~4.4]{myguyzak} (with $N'=n-1$)
that
\begin{equation}
\label{EQ:my-guy-zak}
\det H \equiv 0\bmod{F^{n-r(F)}}.
\end{equation}
For completeness, let us unravel Zak's proof of \eqref{EQ:my-guy-zak}.
If $r(F)=n$, then \eqref{EQ:my-guy-zak} is trivial,
so assume $r(F)<n$.
By definition of $r(F)$,
every $(r(F)+1)\times(r(F)+1)$ minor of $H$
is divisible by $F$.
The divisibility \eqref{EQ:my-guy-zak} then
follows from \cite[Proposition~4.3]{myguyzak},
an elementary result on matrices,
applied to the matrix $H$
over the integral domain $A = \QQ[\bfx]$,
with principal prime ideal $\mathfrak p = (F)$.

Since $\det H \ne 0$ by goodness (Definition~\ref{defn : good}),
comparing degrees in \eqref{EQ:my-guy-zak}
gives $n \ge 3(n-r(F))$ and hence $r(F)\ge 2n/3$.
If $n\equiv 0\bmod{3}$ then $r(F) \ge 2n/3 = 2\lceil n/3\rceil$,
and if $n\equiv 2\bmod{3}$ then $r(F) \ge \lceil 2n/3 \rceil = (2n+2)/3 = 2\lceil n/3\rceil$.
It remains to prove Lemma~\ref{LEM:my-guy-slop} when $n\equiv 1\bmod{3}$.
In this case, $r(F) \ge \lceil 2n/3 \rceil = (2n+1)/3$,
whereas $2\lceil n/3\rceil = (2n+4)/3$.
Assume for contradiction that $r(F) = (2n+1)/3$.
Then $n-r(F) = (n-1)/3$, so \eqref{EQ:my-guy-zak} gives
\begin{equation*}
\det H = F^{(n-1)/3} L
\end{equation*}
for some nonzero linear form $L\in \QQ[\bfx]$,
since $\deg(\det H) = n$.
Let
\begin{equation*}
M := \{\bfx\in \A^n: L(\bfx)=0\}
\cong \A^{n-1}.
\end{equation*}
Then $\dim M = n-1 \ge 1$, so $M(\QQ) \setminus \{\bm0\}$ is Zariski dense in $M$.
If $M$ is contained in $V_s$ but not in $V_{s-1}$,
then it follows that $M(\QQ) \setminus \{\bm0\}$ is contained in $V_s$ but not in $V_{s-1}$.
Here $s<n$, since $M \subset \{\bfx\in \A^n: \det H(\bfx)=0\}$.
Choose a vector $\bfx\in M(\QQ) \setminus \{\bm0\}$ with $\bfx\in V_s - V_{s-1}$.
Since $s<n$, there exists
a vector $\bfy\in \QQ^n \setminus \{\bm0\}$ with $H(\bfx)\bfy=\bm0$.
Then $H(\bfy)\bfx = H(\bfx)\bfy = \bm0$, whence $\det H(\bfy) = 0$.
But $F(\bfy)\ne 0$, since $X(\QQ)=\emptyset$.
Hence $L(\bfy)=0$, so $\bfy\in M$.

Since $M$ is a linear space contained in $V_s$,
we have $M = T_{\bfx} M \subset T_{\bfx} V_s$.
Lemma~\ref{LEM:dav-bad-tangent-space},
with $(s,\bfx,\bfy,\bfy,\bfy)$ in place of $(r,\bfx,\bfv,\bfy,\bfz)$,
gives $0 = \bfy^t H(\bfy) \bfy = 6F(\bfy)$.
This contradicts $F(\bfy)\ne 0$.
It follows that $r(F) > (2n+1)/3$,
which implies Lemma~\ref{LEM:my-guy-slop}.
\end{proof}

\begin{rem}
As stated,
\cite[Proposition~4.3]{myguyzak}
is false for general integral domains $A$ and prime ideals $\mathfrak p$.
Its proof relies on the identity $v_{\mathfrak p}((\det M)^{n-1}) = (n-1)v_{\mathfrak p}(\det M)$, where $v_{\mathfrak p}(a) := \max\{k\in \ZZ_{\ge 0} \cup \{\infty\}: a\in \mathfrak p^k\}$ for $a\in A$.
This is true if $\mathfrak p$ is principal, for instance,
but not in general.
\end{rem}

If $X(\QQ)=\emptyset$, 
we claim that  $T_{\bfx}V_r \cap \ker H(\bfx) = 0$
for all $\bfx\in (V_r \setminus V_{r-1})(\QQ)$.
Since $\bfx$ is rational, the intersection $T_{\bfx}V_r\cap\ker H(\bfx)$ is a linear space defined over $\QQ$. If it were nonzero, it would contain a nonzero rational vector $\bfy$, and 
 Lemma~\ref{LEM:dav-bad-tangent-space}
would give \(6F(\bfy)=\bfy^tH(\bfy)\bfy=0\), contradicting \(X(\QQ)=\emptyset\).
In this way, Davenport obtained a strong restriction on $T_{\bfx}V_r$
at rational points $\bfx$. The next lemma allows us to combine
Lemma~\ref{LEM:dav-bad-tangent-space} with the assumption
$X(\QQ)=\emptyset$ even when the point under consideration
is not rational.

\begin{lemma}
\label{LEM:no-slop}
Assume that $X(\QQ)=\emptyset$.
Let $A\in\mathrm{GL}_n(\Qbar)$, and assign a real
weight $w_i$ to each variable $x_i$. The weight of a cubic monomial
$x_i x_j x_k$ is defined to be $w_i+w_j+w_k$. 
If every monomial occurring with a nonzero coefficient in
$F(A\bfx)$ has nonnegative weight, then
\begin{equation}
\label{EQ:mean-weight}
w_1+\cdots+w_n\geq 0.
\end{equation}
\end{lemma}

\begin{proof}
First suppose $A\in \GL_n(\QQ)$.
Then $A\bfe_i \in \QQ^n \setminus \{\bm0\}$ for each standard basis vector $\bfe_i \in \QQ^n$.
Thus $F(A\bfe_i) \ne 0$, since $X(\QQ) = \emptyset$.
But $F(A\bfe_i)$ is precisely the coefficient of $x_i^3$ in $F(A\bfx)$,
which has weight  $3w_i \ge 0$ for all $i$, by assumption. 
Summing over $i$ leads to \eqref{EQ:mean-weight}
in this case.
It remains to reduce the general case
to the case $A\in \GL_n(\QQ)$,
by a careful Galois descent procedure.

Let $I$ be a subset of
the set of all monomials $x_ix_jx_k$,
where $1\le i,j,k\le n$.
(After reordering $i,j,k$ we could assume $i\le j\le k$,
but that would cause confusion below.)
Let $W\in \RR$.
Given $I$, let $S_I$ be the set of matrices $A\in \GL_n(\Qbar)$
for which every monomial appearing in $F(A\bfx)$
lies in the set $I$.
Given $I$ and $W$, let $T_{I,W}$ be the set of vectors $\bfv \in \RR^n$ for which the following hold:
\begin{enumerate}
\item Every monomial $x_ix_jx_k\in I$
satisfies $v_i+v_j+v_k \ge 0$.
\item We have $\sum_{1\le i\le n} v_i = W$.
\end{enumerate}
Letting $I$ range over all possible subsets of monomials,
define the (not necessarily disjoint) union
$${\textstyle \mathcal{S}_W := \bigcup_I (S_I \times T_{I,W}).}$$
Every pair $(A,\bfw)$ in the statement of Lemma~\ref{LEM:no-slop}
lies in $\mathcal{S}_W$
for some choice of $W$.
We then wish to show that $W\ge 0$.
Note however that in the definition of $S_I$,
the monomial support of $F(A\bfx)$
is only required to be a subset of $I$,
not necessarily equal to $I$.

Two general symmetries will be useful.
First, if $(A,\bfv)\in S_I \times T_{I,W}$,
then $$(AP,P^{-1}\bfv)\in S_{I\circ P} \times T_{I\circ P,W}$$
for every permutation matrix $P$,
where $I\circ P := \{m\circ P: m\in I\}$.
Indeed, if the $(i,\pi(i))$ entry of $P$ is $1$,
where $\pi$ is a permutation,
then $(P\bfx)_i = x_{\pi(i)}$ and $(P^{-1}\bfv)_{\pi(i)} = v_i$,
and a monomial $m=x_ix_jx_k$ in $F(A\bfx)$
is transformed to $m\circ P = x_{\pi(i)}x_{\pi(j)}x_{\pi(k)}$ in $F(AP\bfx)$,
so that $$(P^{-1}\bfv)_{\pi(i)} + (P^{-1}\bfv)_{\pi(j)} + (P^{-1}\bfv)_{\pi(k)}
= v_i + v_j + v_k \ge 0.$$
Second, let $\mathcal{B}(\bfv)$ be the group of matrices $U = (u_{ij}) \in \GL_n(\Qbar)$
such that $u_{ij} \ne 0 \Rightarrow v_i \le v_j$.
For example, if $v_1\le v_2\le \cdots \le v_n$,
then any invertible upper triangular matrix lies in $\mathcal{B}(\bfv)$,
and if $v_1<v_2<\cdots<v_n$ then every matrix in $\mathcal{B}(\bfv)$ is upper triangular.
We claim that if $U \in \mathcal{B}(\bfv)$
and $J$ denotes the monomial support of $F(AU\bfx)$,
then $$(AU,\bfv)\in S_J \times T_{J,W}.$$
Indeed, if $x_ix_jx_k \in I$,
then every monomial appearing
in the transformed polynomial $$(x_ix_jx_k) \circ U
= (U\bfx)_i(U\bfx)_j(U\bfx)_k$$
may be written in the form $x_{i'}x_{j'}x_{k'}$,
where $v_i\le v_{i'}$, $v_j\le v_{j'}$ and $v_k\le v_{k'}$,
by the definition of $\mathcal{B}(\bfv)$.
It follows that $v_{i'} + v_{j'} + v_{k'}
\ge v_i + v_j + v_k \ge 0$, as desired.

With these two general symmetries established,
suppose now that $(A,\bfv),(A',\bfv')\in S_I \times T_{I,W}$.
Then there exist $U\in\mathcal{B}(\bfv)$,
$U'\in\mathcal{B}(\bfv')$
and a permutation matrix $P$ such that
$AUP = A'U'$.
Indeed, first order the coordinates for $\bfv$ and $\bfv'$
separately so that their entries are non-decreasing.
In these orderings, both $\mathcal B(\bfv)$ and
$\mathcal B(\bfv')$ contain the group of invertible upper
triangular matrices.
The usual {\em Bruhat decomposition} therefore gives
\[
A^{-1}A'=UPV,
\]
with $U\in\mathcal B(\bfv)$, $V\in\mathcal B(\bfv')$
and $P$ a permutation matrix.
This remains a decomposition of the same form after undoing
the coordinate permutations.
Taking $U'=V^{-1}$ gives $AUP=A'U'$.

Let $A'' := AUP = A'U'$.
By the symmetries in the penultimate paragraph,
$(A'',P^{-1}\bfv)=(AUP,P^{-1}\bfv)\in \mathcal{S}_W$
and $(A'',\bfv')=(A'U',\bfv')\in \mathcal{S}_W$.
If $I''$ denotes the monomial support of $F(A''\bfx)$,
then $$(A'',\tfrac12(P^{-1}\bfv+\bfv'))
\in S_{I''} \times T_{I'',W} \subset \mathcal{S}_W,$$
since $T_{I'',W}$ is convex.
Moreover, by the parallelogram law and the identity $\|P^{-1}\bfv\|_2 = \|\bfv\|_2$,
\begin{equation*}
\|\tfrac12(P^{-1}\bfv+\bfv')\|_2^2
= \tfrac12(\|P^{-1}\bfv\|_2^2 + \|\bfv'\|_2^2) - \|\tfrac12(P^{-1}\bfv-\bfv')\|_2^2
\le \tfrac12(\|\bfv\|_2^2 + \|\bfv'\|_2^2),
\end{equation*}
with equality if and only if $P^{-1}\bfv = \bfv'$.

Each set $T_{I,W}$ is closed in $\RR^n$,
and thus if $T_{I,W} \ne \emptyset$ then
$T_{I,W}$ contains a vector of minimal Euclidean length.
Since there are only finitely many possible sets $I$, it follows that
there exists $(A,\bfv)\in \mathcal{S}_W$ for which $\|\bfv\|_2$ is minimal.
Suppose $(A,\bfv)\in S_I \times T_{I,W}$.
Then for any $\sigma\in \Gal(\Qbar/\QQ)$,
we have $(\sigma(A),\bfv) \in S_I \times T_{I,W}$.
Taking $(\sigma(A),\bfv)$ for $(A',\bfv')$ in the previous paragraph,
we get $P^{-1}\bfv=\bfv$,
since $\|\tfrac12(P^{-1}\bfv+\bfv)\|_2 \ge \|\bfv\|_2$ by minimality of $\|\bfv\|_2$.

It follows from the definition of $\mathcal{B}(\bfv)$ that, for every
$U\in\mathcal{B}(\bfv)$ and every $r\in\mathbb{R}$,
\begin{equation}
U\text{span}_{\Qbar}\{e_j:v_j\leq r\}=\text{span}_{\Qbar}\{e_j:v_j\leq r\}.
\label{eq : pizzaefichi}
\end{equation}
Indeed, if $v_j\leq r$, then $Ue_j=\sum_i u_{i,j}e_i$. Since $U\in\mathcal{B}(\bfv)$, we have $u_{i,j}\neq 0$ only if $v_i\leq v_j$. Hence
\[
Ue_j\in\text{span}_{\Qbar}\{e_i:v_i\leq v_j\}
\subseteq\operatorname{span}_{\Qbar}\{e_i:v_i\leq r\}.
\]
This proves the inclusion from left to right in \eqref{eq : pizzaefichi}.
The equality follows since $U^{-1}\in \mathcal{B}(\bfv)\subset\GL_{n}(\Qbar)$.
Since $AUP=\sigma(A)U'$, with $U,U'\in\mathcal{B}(\bfv)$, it follows from \eqref{eq : pizzaefichi} and the identity $P\bfv=\bfv$ that the filtration
\[
D_{\le r}:=\operatorname{span}_{\Qbar}\{Ae_j:v_j\leq r\}\subseteq\Qbar^n
\]
is $\Gal(\Qbar/\QQ)$-stable, where $r$ ranges over all real numbers.
 By Galois descent, there exists a filtration of $\QQ$-vector spaces
$E_{\le r} \subset \QQ^n$
such that $E_{\le r} \otimes \Qbar = D_{\le r}$.
Let $B\in \GL_n(\QQ)$ be a matrix whose columns $B\bfe_j$ split this filtration,
so that
$$\mathrm{span}_{\QQ}\{B\bfe_j: v_j\le r\} = E_{\le r}.$$
By the definition of $\mathcal{B}(\bfv)$,
we may thus write $B = AM$ for some $M\in \mathcal{B}(\bfv)$.
Therefore,
$$(B,\bfv)=(AM,\bfv) \in \mathcal{S}_W,$$
by the symmetry properties observed earlier.
On replacing $A$ with $B$,
it follows that we may assume $A\in \GL_n(\QQ)$,
whence $W \ge 0$,
completing the proof of Lemma~\ref{LEM:no-slop}.
\end{proof}

\begin{rem}
\label{RMK:GIT-view}
The content of Lemma~\ref{LEM:no-slop} is best viewed through the lens of geometric invariant theory (GIT), as we shall now explain. Over $\QQ$, the group $\SL_n$ acts on the space $W_n$ of cubic forms in $n$ variables via
$(g\cdot F)(\bfx)=F(g^{-1}\bfx)$. A form $F$ is called \emph{semistable} if the zero form does not belong to the Zariski closure of the
$\SL_n(\Qbar)$-orbit of $F$.  Lemma~\ref{LEM:no-slop} is then equivalent to saying that a cubic form without a nontrivial integer zero is semistable and an alternative proof of Lemma~\ref{LEM:no-slop} can be obtained using Kempf's rational version of the geometric Hilbert--Mumford criterion ~\cite[Corollary~4.3]{Kem78}. Indeed, the elementary proof above is essentially a simplified version of the proof of this result. See \cite[\S4]{BMRT} for related optimality and rationality results.

Intuitively, semistability prevents the cubic form from having very degenerate singularities and this is reflected in the next set of results, where we draw various conclusions about the geometry of the singular locus and the Hessian loci associated to our cubic form $F$. This connection between GIT stability and singularities has been used before, for example in Laza's work~\cite{Laz09} on the GIT compactification of the moduli space of cubic fourfolds. Additionally, a slightly weaker version of Lemma~\ref{LEM:no-slop-singular-linear} was obtained by Park~\cite[Lemma~7.2]{Par25}, also using semistability.  
\end{rem}

Recall the affine incidence correspondence $I$
from \eqref{Eq: Def.hessian.Incidence}.
It will be convenient to have the following concentration result for $\dim I$.
We follow the standard convention that an irreducible set must be nonempty,
and that the affine cone of the empty set is $\{\bm0\}$.

\begin{lemma}
\label{LEM:hessian-concentration}
There exist an integer $0\le r\le n$
and an irreducible or empty closed subset $Z \subset X$,
whose affine cone $\cone Z$ lies in $V_r$ but not in $V_{r-1}$,
such that $\dim I = \dim \cone Z + n-r$.
\end{lemma}

\begin{proof}
Since $\{\bm0\}\times\A^n\subset I$, we have $\dim I\ge n$.
If $\dim I=n$, then we may take $Z=\emptyset$ and $r=0$,
in the light of our convention $\cone \emptyset=\{\bm0\}$.
Thus, we henceforth assume that $\dim I>n$.

Put $\bfv=\bfy+\bfx$ and $\bfw=\bfy-\bfx$.
The identities
$$
\nabla F(\bfv)-\nabla F(\bfw)=2H(\bfy)\bfx,
\qquad
F(\bfv)-F(\bfw)=2F(\bfx)+\bfy^tH(\bfy)\bfx,
$$
together with $H(\bfy)\bfx=H(\bfx)\bfy$, identify $I$ with
$$
J:=\{(\bfv,\bfw)\in\A^{2n}:
\nabla F(\bfv)=\nabla F(\bfw)\},
$$
and $I\vert_{\cone X\times\A^n}$ with
$$
J_0:=\{(\bfv,\bfw)\in J:F(\bfv)=F(\bfw)\}.
$$
We claim that $\dim J_0=\dim J$.
For $\bfz\in\Qbar^n$, put
$$
Y_{\bfz}:=\{\bfx\in\A^n:\nabla F(\bfx)=\bfz\}.
$$
On any irreducible component $C$ of $Y_{\bfz}$,
Euler's identity gives $3F=\bfz\cdot\bfx$.
Differentiating in the function field of $C$ gives
$$
3\dd F=\bfz\cdot\dd\bfx=\dd F.
$$
Thus $\dd F=0$, so $F$ is constant on $C$, since the ground
field is algebraically closed of characteristic zero.
Now project $J$ and $J_0$ to their common gradient.
For each nonempty $Y_{\bfz}$, choose a component $C$ of
dimension $d=\dim Y_{\bfz}$.
Then
$
C\times C\subset (J_0)_{\bfz}
\subset J_{\bfz}=Y_{\bfz}\times Y_{\bfz},
$
so both fibres have dimension $2d$.
The constructible sets
$$
S_d:=\{\bfz\in\A^n:\dim Y_{\bfz}=d\},
\qquad 0\le d\le n,
$$
therefore give
$$
\dim J_0
=\max_{S_d\ne\emptyset}(\dim S_d+2d)
=\dim J.
$$
Consequently, $\dim I=\dim I\vert_{\cone X\times\A^n}$.
Stratifying the projective incidence correspondence on $X$ by rank,
and taking affine cones on a top-dimensional stratum,
completes the proof.
\end{proof}

\begin{lemma}
\label{LEM:433-incidence-bound}
Assume $X(\QQ)=\emptyset$.
We have
\begin{align}
1 + \dim \sing X \le \max\{0, (2n-3)/3\},
\label{EQ:clean-slop-sing-bound} \\
\dim I\vert_{\sing(\cone X) \times \A^n} \le \max\{n, (4n-6)/3\},
\label{EQ:clean-slop-sing-incidence-bound} \\
\dim I \le \max\{n, (4n-3)/3\}.
\label{EQ:clean-slop-incidence-bound}
\end{align}
Moreover, fix an integer $0\le r\le n$
and an irreducible closed subset $Z \subset X$,
whose affine cone $\cone Z$ lies in $V_r$ but not in $V_{r-1}$.
If $Z \subset \sing X$, then
\begin{equation}
\label{EQ:slop-sing-bound}
\dim \cone Z \le \min\{n-r,2r-3\}.
\end{equation}
If $Z \not\subset \sing X$, then
\begin{equation}
\label{EQ:slop-smooth-bound}
\dim \cone Z \le \min\{r+(n-3)/3,2r-2\}.
\end{equation}
\end{lemma}

\begin{proof}
The bound \eqref{EQ:clean-slop-sing-bound} is trivial if $\sing X = \emptyset$,
and otherwise it follows from \eqref{EQ:slop-sing-bound},
since $$\min\{n-r,2r-3\}
\le \tfrac23(n-r) + \tfrac13(2r-3) = \tfrac13(2n-3).$$
The bound \eqref{EQ:clean-slop-sing-incidence-bound} is trivial if $\dim I \le n$,
and otherwise it follows from \eqref{EQ:slop-sing-bound},
since $$\min\{n-r,2r-3\} + n-r \le \tfrac13(n-r) + \tfrac23(2r-3) + n-r
= \tfrac13(4n-6).$$
The bound \eqref{EQ:clean-slop-incidence-bound} is trivial if $\dim I \le n$,
and otherwise it follows from Lemma~\ref{LEM:hessian-concentration},
together with \eqref{EQ:clean-slop-sing-incidence-bound}
and the first bound in \eqref{EQ:slop-smooth-bound},
since $r+(n-3)/3 + n-r = (4n-3)/3$.

It remains to prove \eqref{EQ:slop-sing-bound} and \eqref{EQ:slop-smooth-bound}.
Write $\Phi(\bfy,\bfv,\bfz) := \bfy^tH(\bfv)\bfz$
for the symmetric trilinear form associated to $F$.
Choose a nonzero smooth point $\bfx$ of $\cone Z$ with $\rank H(\bfx)=r$,
and with $\nabla F(\bfx)\ne \bm0$ if $Z\not\subset\sing X$.
Set
\begin{equation*}
T:=T_{\bfx}(\cone Z),
\qquad L:=\ker H(\bfx).
\end{equation*}
Lemma~\ref{LEM:dav-bad-tangent-space} gives $\Phi(L,T,L)=0$.
Also, $\Phi(\Qbar^n,\bfx,L)=0$, because $H(\bfx)L=\bm0$.
Finally, since $F\vert_{\cone Z}=0$, we have $F(\bfx+\ve\bfv) \equiv 0\bmod{\ve^2}$
for all $\bfv\in T$, so $\bfv^t \nabla F(\bfx) = 0$.
Thus $\Phi(T,\bfx,\bfx)=0$,
in view of the identity $H(\bfx)\bfx=2\nabla F(\bfx)$.

Before proceeding, we collect some basic facts about $T$ and $L$.
Since $\cone Z$ is cut out by homogeneous equations,
Euler's identity gives $\bfx\in T$.
By smoothness of $\cone Z$ at $\bfx$, we have $\dim T = \dim \cone Z$.
Since $\rank H(\bfx)=r$, we have $\dim L = n-r$.

Suppose first that $Z\subset\sing X$.
Then $\nabla F(\bfx+\ve\bfv) \equiv 0\bmod{\ve^2}$
for all $\bfv\in T$,
so $H(\bfx)\bfv = \bm0$.
Thus $T\subset L$, so $\dim \cone Z\le n-r$.
It remains to prove that $\dim \cone Z \le 2r-3$.
Split the flag $\Qbar\bfx \subset T \subset L \subset \Qbar^n$,
writing
\[
W_{-4} := \Qbar\bfx,\qquad
T=W_{-4}\oplus W_{-1},\qquad
L=T\oplus W_0,\qquad
\Qbar^n=L\oplus W_2,
\]
so that $\Qbar^n = W_{-4} \oplus W_{-1} \oplus W_0 \oplus W_2$.
Let $A\in \GL_n(\Qbar)$ be a matrix such that
$$\{A\bfe_1,\dots,A\bfe_n\}
\subset W_{-4} \cup W_{-1} \cup W_0 \cup W_2.$$
Let $w_i := \nu$ if $A\bfe_i\in W_\nu$, where $1\le i\le n$ and $\nu\in \{-4,-1,0,2\}$.
Let $1\le i,j,k\le n$ with $w_i\le w_j\le w_k$
and suppose $\Phi(A\bfe_i,A\bfe_j,A\bfe_k) \ne 0$.
If $w_i=-4$,
then $w_j=w_k=2$, since $\Phi(\bfx,L,\Qbar^n)=0$ and $L = W_{-4} \oplus W_{-1} \oplus W_0$.
Thus $w_i+w_j+w_k = -4+2+2 = 0$ in this case.
If $w_i=-1$, then $w_k=2$, since $\Phi(T,L,L)=0$.
Thus $w_i+w_j+w_k \ge -1-1+2 = 0$ in this case.
If $w_i\ge 0$, then $w_i+w_j+w_k\ge 0$ trivially.
Lemma~\ref{LEM:no-slop} therefore gives
\[
-4-\dim W_{-1}+2\dim W_2 \ge 0.
\]
Since $\dim W_{-1} = \dim \cone Z - 1$ and $\dim W_2 = r$,
we get $\dim \cone Z \le 2r-3$.

Suppose now that $Z\not\subset\sing X$.
Since $H(\bfx)\bfx=2\nabla F(\bfx)\ne \bm0$, we have $\bfx\notin L$.
Split the flag $(T\cap L)\oplus\Qbar\bfx \subset T \subset T+L \subset \Qbar^n$,
writing
\[
W_{-2} := (T\cap L)\oplus\Qbar\bfx,\qquad
T = W_{-2}\oplus W_1,\qquad
L = (T\cap L)\oplus W_0,\qquad
\Qbar^n = (T+L) \oplus W_4,
\]
so that $T+L = T\oplus W_0$.
By trilinearity of $\Phi$,
we have $\Phi(W_{-2},L+\Qbar\bfx,T+L)=0$ since
\begin{enumerate}
\item $\Phi(T\cap L,L,T) \subset \Phi(L,L,T)=0$,
\item $\Phi(T\cap L,L,L) \subset \Phi(T,L,L)=0$,
\item $\Phi(T\cap L,\bfx,T+L) \subset \Phi(L,\bfx,\Qbar^n)=0$,
\item $\Phi(\bfx,L,T+L) \subset \Phi(\bfx,L,\Qbar^n)=0$,
\item $\Phi(\bfx,\bfx,T)=0$, and
\item $\Phi(\bfx,\bfx,L) \subset \Phi(\bfx,\Qbar^n,L)=0$.
\end{enumerate}
Since $L+\Qbar\bfx=W_{-2} \oplus W_0$ and $T+L=W_{-2} \oplus W_0 \oplus W_1$,
it follows that
\begin{equation}
\label{EQ:clever-slop}
\Phi(W_{-2},W_{-2} \oplus W_0,W_{-2} \oplus W_0 \oplus W_1)=0.
\end{equation}
Let $A\in \GL_n(\Qbar)$ be a matrix such that
$$\{A\bfe_1,\dots,A\bfe_n\}
\subset W_{-2} \cup W_0 \cup W_1 \cup W_4.$$
Let $w_i := \nu$ if $A\bfe_i\in W_\nu$, where $1\le i\le n$ and $\nu\in \{-2,0,1,4\}$.
Let $1\le i,j,k\le n$ with $w_i\le w_j\le w_k$
and suppose $\Phi(A\bfe_i,A\bfe_j,A\bfe_k) \ne 0$.
If $w_i=-2$ and $w_j\le 0$, then $w_k=4$ by \eqref{EQ:clever-slop}.
Thus $w_i+w_j+w_k \ge -2-2+4 = 0$ in this case.
If $w_i=-2$ and $w_j\ge 1$, then $w_i+w_j+w_k \ge -2+1+1 = 0$.
If $w_i\ge 0$, then $w_i+w_j+w_k\ge 0$ trivially.
Lemma~\ref{LEM:no-slop} therefore gives
\begin{equation*}
-2\dim W_{-2} + \dim W_1 + 4 \dim W_4 \ge 0.
\end{equation*}
Let $a := \dim(T\cap L)$ and $t := \dim T = \dim \cone Z$.
Then $\dim W_{-2} = a + 1$,
$\dim W_1 = t - a - 1$
and $\dim W_4 = n - \dim (T+L) = n - t - \dim L + a
= r - t + a$.
Thus
\begin{equation*}
0\le -2(a+1) + (t-a-1) + 4(r-t+a)
= -3+a-3t+4r.
\end{equation*}
Since $\bfx\in T$ but $\bfx\notin L$, we have $a\le t-1$,
so $0 \le -4-2t+4r$ and $t\le 2r-2$.
On the other hand, $a\le \dim L = n-r$,
so $0 \le -3+n-3t+3r$ and $t\le r + (n-3)/3$.
\end{proof}

\begin{cor}
\label{cor:hesrank}
Assume that $n\ge 3$
and $X(\QQ)=\emptyset$.
Let $1\leq r \leq n-1$.
Then
$$\dim V_r\leq \min \left\{\left\lfloor \frac{n}{3}\right\rfloor +r-1, n-1\right\}$$
and $V_0 = V_1\cap \cone X = \{\bm0\}$.
Moreover, $\dim(V_r\cap \cone X)\leq \min\{2r-2,n-2\}$ if $r<2\lceil n/3\rceil$.
\end{cor}
\begin{proof}
Since $F$ is good (see Definition~\ref{defn : good}),
the polynomial $\det(H(\bm{x}))$ does not vanish identically.
In particular, if $\bm{x}\in V_r$, then $\det(H(\bm{x}))=0$ must hold and hence $\dim V_r\leq n-1$.
Furthermore, $V_0$ is a linear space defined over $\QQ$,
so $V_0 = \{\bm0\}$ by the $r=0$ case of Definition~\ref{defn : good}.

Next consider the projection $\pi\colon I\to \A^n$ defined by $(\bm{x},\bm{y})\mapsto \bm{x}$. Then clearly 
$$V_r=
\{\bm{x}\in \A^n\colon \dim \pi^{-1}(\bm{x})\geq n-r\}.
$$
Thus $\dim V_r \leq \dim I - (n-r)\leq \lfloor n/3 \rfloor +r-1$
by \eqref{Eq: Def.hessian.Incidence} and \eqref{EQ:clean-slop-incidence-bound},
since $n \le (4n-3)/3$.

The claims $\dim(V_r\cap \cone X) \le 2r-2$
and $V_1\cap \cone X = \{\bm0\}$
follow from \eqref{EQ:slop-sing-bound} and \eqref{EQ:slop-smooth-bound}.
The final claim follows directly from
Lemma~\ref{LEM:my-guy-slop}.
\end{proof}

It will also be useful to have some control on linear spaces in $\sing X$,
going beyond \eqref{EQ:clean-slop-sing-bound}.

\begin{lemma}
\label{LEM:no-slop-singular-linear}
Assume $X(\QQ)=\emptyset$.
If $\nabla F$ vanishes on
a $\Qbar$-vector space $L\subset\Qbar^n$,
then $$\dim L\le n/3.$$
\end{lemma}

\begin{proof}
Let $\Phi$ be as in the proof of Lemma~\ref{LEM:433-incidence-bound}.
Then $\Phi(L,L,\Qbar^n) = 0$,
since $$2H(\bfx)\bfy = \nabla F(\bfx+\bfy) - \nabla F(\bfx-\bfy) = \bm0-\bm0 = \bm0$$
for all $\bfx,\bfy\in L$.
Write $\Qbar^n=L\oplus M$.
Choose $A\in \GL_n(\Qbar)$ with
$\{A\bfe_1,\dots,A\bfe_n\}
\subset L \cup M$.
Let $w_i := -2$ if $A\bfe_i\in L$, and $w_i := 1$ if $A\bfe_i\in M$.
Let $1\le i,j,k\le n$ with $w_i\le w_j\le w_k$
and suppose $\Phi(A\bfe_i,A\bfe_j,A\bfe_k) \ne 0$.
If $w_i = -2$, then $w_j=w_k=1$, since $\Phi(L,L,\Qbar^n) = 0$.
Thus $w_i+w_j+w_k=0$ in this case.
If $w_i=1$, then $w_i+w_j+w_k=3$.
Lemma~\ref{LEM:no-slop} therefore gives
$-2\dim L + \dim M \ge 0$.
Since $\dim M = n - \dim L$, it follows that $3\dim L \le n$.
\end{proof}

\section{Families of singular hyperplane sections}
\label{sec:dual-tangential-geometry}

Throughout this section, we adopt the notation of \S\ref{sec : geom_cubic_hyper}.
Assume $n\ge 4$ and $X(\QQ) = \emptyset$.
By Proposition~\ref{Bro-integral-cone},
the $\Qbar$-variety $X \subset \PP^{n-1}$ is integral.
Write $H_{\bfv} := \{[\bfx]\in \PP^{n-1}: \bfv\cdot \bfx=0\}$ for $\bfv\ne \bm0$.
For each integer $t\ge 0$, put
\begin{equation}
\label{EQ:general-bad-hyperplanes}
B_t:=
\{[\bfv]\in\PP^{n-1}:
\dim\sing(X\cap H_{\bfv})\ge t\},
\end{equation}
where $\sing(X\cap H_{\bfv})$ is interpreted scheme-theoretically
as in \S\ref{sec : geom_cubic_hyper}.
Consider the closed incidence correspondence
\begin{equation}
\label{EQ:universal-section-singular}
S:=\{([\bfx],[\bfv])\in X\times\PP^{n-1}:
 \bfv\cdot\bfx=0,\quad
 v_iF_j(\bfx)-v_jF_i(\bfx)=0\text{ for all }i,j\},
\end{equation}
where $F_i=\partial F/\partial x_i$.
Its fibre over $[\bfv]$ is $\sing(X\cap H_{\bfv})$.
Thus $B_t$ is closed and defined over $\QQ$,
by upper semicontinuity of fibre dimension for a proper morphism.
The scheme structures on $B_t$ and $S$
are unimportant; for instance, the reduced scheme structure would suffice.

\begin{prop}
\label{PROP:rational-exceptional-hyperplanes-general}
Suppose $t\ge\max\{0,\dim\sing X\}$,
and let $A_t$ be the Zariski closure of $B_t(\QQ)$ in $\PP^{n-1}$.
Then $A_t$ is defined over $\QQ$.
Moreover, $A_t(\QQ)=B_t(\QQ)$ and
\begin{equation}
\label{EQ:rational-terminal-bound}
1+\dim A_t\le
\max\left\{0,n-\left\lceil\frac{3(t+2)}2\right\rceil\right\}.
\end{equation}
\end{prop}

\begin{proof}
The set $A_t$ is contained in the intersection of its Galois conjugates.
Thus $A_t$ is Galois-stable, and hence defined over $\QQ$.
Moreover, $A_t \subset B_t$, since $B_t$ is closed.
Thus $A_t(\QQ) \subset B_t(\QQ)$.
Since $B_t(\QQ) \subset A_t(\QQ)$ by definition,
it follows that $A_t(\QQ)=B_t(\QQ)$.
It remains to prove \eqref{EQ:rational-terminal-bound}.

The set $A_t(\QQ)=B_t(\QQ)$
is Zariski dense in $A_t$.
Thus every nonempty open subset of $A_t$
contains a $\QQ$-point.
This, together with \eqref{EQ:clean-slop-sing-bound},
will be the main arithmetic input used below.

Let $Z$ be an irreducible component of $(A_t)_{\rm sm}$.
As in \S\ref{sec : geom_cubic_hyper}, we work geometrically over $\Qbar$.
Let $\pi\colon S \to \PP^{n-1}$ be the projection $([\bfx],[\bfv]) \mapsto [\bfv]$.
For each $\bfv\in Z$, we have $\dim \pi^{-1}(\bfv) = \dim \sing(X\cap H_{\bfv}) \ge t \ge 0$.
Since $Z$ is irreducible,
generic smoothness of the map $\pi\vert_{\pi^{-1}(Z)}$,
for the reduced scheme structures on $Z$ and $\pi^{-1}(Z)$,
produces an irreducible open subset $U$ of $\pi^{-1}(Z)$
for which $\pi\vert_U$ is smooth of relative dimension $\ge t$
and the image $Z'$ of $\pi\vert_U$ is a dense open subset of $Z$.
In particular, $\pi\vert_U$ is flat,
with some constant fibre dimension $e\ge t$.

Let $\cone Z'$ denote the affine cone of $Z'$.
Let $\bfv\in \cone Z' \setminus \{\bm0\}$
and $M_{\bfv} := (T_{\bfv} \cone Z')^\perp \subset \A^n$.
We claim that
\begin{equation}
\label{EQ:rational-contact-restriction}
\dim\{\bfx\in M_{\bfv}:\nabla F(\bfx)\in M_{\bfv}^\perp\}\ge e+1.
\end{equation}
For a linear subspace $L\subset \Qbar^n$,
we have $\frac{\mathrm{d}}{\mathrm{d}t} F(\bfx+t\bfv) \vert_{t=0}
= \bfv\cdot \nabla F(\bfx)$ for all $\bfx,\bfv\in L$.
If $F\vert_L\ne 0$,
then $F\vert_L$ is singular at $\bfx$
if and only if $\nabla F(\bfx)$ is orthogonal to $L$.
Thus if $F\vert_{M_{\bfv}} \ne 0$,
then the left-hand side of \eqref{EQ:rational-contact-restriction}
equals $\dim\sing\{\bfx\in M_{\bfv}:F(\bfx)=0\}$.

Suppose first that $U$ is not contained in
$\sing(X) \times Z'$.
Then by shrinking $U$ and $Z'$ if necessary, we may assume that
$U \subset X_{\rm sm} \times Z'$.
For $([\bfx],[\bfv])\in U$, each $\bfw\in T_{\bfv} \cone Z'$
lifts via $\pi\vert_U$ to a tangent vector $(\bfy,\bfw)$
of (the bi-affine cone of $U$ in) $\cone X_{\rm sm} \times \cone Z'$
at $(\bfx,\bfv)$,
since $\pi\vert_U$ is smooth.
By the definition of $S$,
we have $F(\bfx) = \bfx\cdot\bfv = 0$,
and $\nabla F(\bfx)=\alpha\bfv$ for some $\alpha\ne 0$.
Differentiating the first two equations in the direction $(\bfy,\bfw)$,
using the chain and product rules, gives
\begin{equation*}
0 = \bfy \cdot \nabla F(\bfx)
= \alpha \bfy \cdot \bfv
= - \alpha \bfx \cdot \bfw.
\end{equation*}
Thus $\bfx\in (T_{\bfv} \cone Z')^\perp = M_{\bfv}$.
On the other hand, $\nabla F(\bfx) = \alpha\bfv \in T_{\bfv} \cone Z'$,
since $\bfv\in \cone Z'$ and $\cone Z'$ is a cone.
Thus $\nabla F(\bfx)$ is orthogonal to $M_{\bfv}$.
The affine cone of the fibre $\pi\vert_U^{-1}(\bfv)$
therefore lies in the locus in \eqref{EQ:rational-contact-restriction},
proving that inequality.

Now suppose $U\subset\sing(X)\times Z'$.
Then $\dim U \le \dim \sing X + \dim Z'$.
On the other hand, $\dim U = e + \dim Z'
\ge t + \dim Z'
\ge \dim \sing X + \dim Z'$,
since $e \ge t$
and $t \ge \dim \sing X$.
Thus equality holds.
The closure of $U$ in $\sing(X)\times Z'$
therefore contains some irreducible component $X' \times Z'$
of $\sing(X)\times Z'$, with $\dim X' = \dim \sing X$.
Since $\bfx \cdot \bfv$ vanishes on $U$,
it vanishes on $X' \times Z'$ by continuity.
Differentiating the equation $\bfx \cdot \bfv = 0$
with respect to $\bfv\in \cone Z' \setminus \{\bm0\}$, keeping $\bfx\in \cone X'$ fixed,
shows that $\cone X' \subset (T_{\bfv} \cone Z')^\perp = M_{\bfv}$.
On the other hand, the definition of $\sing X$
implies that $\nabla F(\bfx) = \bm0$ vanishes for all $\bfx\in \cone X'$,
so that $\nabla F(\bfx)$ is trivially orthogonal to $M_{\bfv}$.
It follows that $\cone X'$
lies in the locus in \eqref{EQ:rational-contact-restriction},
proving that inequality.

Since $Z'$ is open in $A_t$, we may choose $\bfv$ to be a $\QQ$-point.
The vector space $M_{\bfv}$ is then defined over $\QQ$.
The inequality \eqref{EQ:rational-contact-restriction}
trivially implies $\dim M_{\bfv} \ge 1$, since $e\ge t\ge 0$.
Moreover, $F\vert_{M_{\bfv}} \ne 0$, since $X(\QQ) = \emptyset$.
Thus \eqref{EQ:rational-contact-restriction} in fact implies
$\dim M_{\bfv} - 1
\ge \dim\sing\{\bfx\in M_{\bfv}:F(\bfx)=0\}
\ge e+1$, whence $\dim M_{\bfv} \ge 2$.
Plugging the inequality \eqref{EQ:clean-slop-sing-bound} for $F\vert_{M_{\bfv}} \ne 0$
into \eqref{EQ:rational-contact-restriction} now gives
\begin{equation*}
\frac{2(n-1-\dim Z) - 3}{3}
= \frac{2\dim M_{\bfv} - 3}{3}
\ge e+1 \ge t+1,
\end{equation*}
whence $2(1+\dim Z) \le 2n-3(t+2)$.
This proves \eqref{EQ:rational-terminal-bound},
after taking into account the fact that $A_t$ could be empty,
in which case $1+\dim A_t = 0$.
\end{proof}

Consider the conormal variety of $X$ together with its projections:
\begin{equation}
\label{def:closed-gauss-correspondence}
\begin{tikzcd}
& \Gamma
:= \bar{\{ (x,H)\in X_{\text{sm}}\times \check{\PP}\smash{^{n-1}} : T_xX\subset H\}}
\subset X\times \check{\PP}\smash{^{n-1}}
\arrow[dl, "\pi_1"] \arrow[dr, "\pi_2"] & \\
X && X^*,
\end{tikzcd}
\end{equation}
where $X^*\subset \check\PP\smash{^{n-1}}$ denotes the dual variety of $X$, which by definition is the image of $\Gamma$ under $\pi_2$.
Equivalently, since we are working in characteristic $\ne 3$,
\begin{equation*}
\Gamma
= \bar{\{([\bfx],[\nabla F(\bfx)]): [\bfx]\in X_{\rm sm}\}}
\subset X\times\PP^{n-1},
\end{equation*}
and $X^*$ is the closure of the image of $X_{\text{sm}}$
under the Gauss map $[\bfx] \mapsto [\nabla F(\bfx)]$.
We remark that $\Gamma$ is often called the \emph{conormal variety} of $X$,
a notion that is critical to Appendix~\ref{sec:n10-singular-support}.

\begin{lemma}\label{lem:2n/3}
We have $\dim X^*\geq 2\lceil n/3\rceil-2$.
\end{lemma}
\begin{proof}
Thanks to a result of Segre \cite{SegreII}, we know that $\dim X^*=r(F)-2$. The claim now follows from Lemma~\ref{LEM:my-guy-slop}. 
\end{proof}

Let $\Gamma_{[\bfv]} \subset X$ be the fibre of $\Gamma$ over $[\bfv]\in \PP^{n-1}$.
We view $\Gamma_{[\bfv]}$ as a closed subset of $X$.
Let
\begin{equation}
W_t:=
\{[\bfv]\in\PP^{n-1}:\dim\Gamma_{[\bfv]}\ge t\}.
\label{EQ:general-large-gauss-fibres}
\end{equation}
Again $W_t$ is closed and defined over $\QQ$,
by repeating the argument for $B_t$
from after \eqref{EQ:universal-section-singular}.
In later sections, $\sing (X_{\rm sm} \cap H_{\bfv})$
will be more directly relevant than
$\Gamma_{[\bfv]}$.
However,
\begin{equation}
\label{EQ:smooth-sing-containment}
\sing (X_{\rm sm} \cap H_{\bfv})
\subset \Gamma_{[\bfv]},
\end{equation}
and there are technical advantages to working with the closed set $\Gamma_{[\bfv]}$.
Moreover, Euler's identity $\bfx \cdot \nabla F(\bfx) = 3F(\bfx)$ implies,
via the definition \eqref{EQ:universal-section-singular} of $S$,
that
\begin{equation}
\label{EQ:gauss-section-containment}
\Gamma\subset S,\qquad
\Gamma_{[\bfv]}\subset S_{[\bfv]} = \sing(X\cap H_{\bfv}).
\end{equation}

\begin{prop}
\label{PROP:geometric-exceptional-hyperplanes-general}
If $t\ge (n+1)/3$, then
\begin{equation}
1+\dim W_t\le\max\{0,n-t-2\}.
\label{EQ:geometric-gauss-bound}
\end{equation}
If $t\ge \max\{(n+1)/3,\dim\sing X\}$, then
\begin{equation}
1+\dim B_t\le\max\{0,n-t-2\}.
\label{EQ:geometric-terminal-bound}
\end{equation}
\end{prop}

\begin{proof}
Suppose first that $t\ge (n+1)/3$.
Since $X_{\rm sm}$ is irreducible,
the closure $\Gamma$ is irreducible,
with $\dim \Gamma = \dim X_{\rm sm} = n-2$.
Yet $\dim X^* \ge 2n/3-2$ by Lemma~\ref{lem:2n/3}.
Consequently the general fibre of $\pi_2\colon \Gamma\to X^*$ has dimension
at most $n/3<t$.
The set $\pi_2^{-1}(W_t)$ is therefore a proper closed
subset of $\Gamma$, of dimension at most $\dim \Gamma - 1 = n-3$.
If $W_t$ is nonempty, every fibre over it has dimension
at least $t$, so
\[
\dim W_t+t
\le \dim \pi_2^{-1}(W_t)
\le n-3.
\]
This proves \eqref{EQ:geometric-gauss-bound},
since if $W_t = \emptyset$ then $1 + \dim W_t = 0$.

Now assume $t\ge \max\{(n+1)/3,\dim\sing X\}$, and take $[\bfv]\in B_t$.
Choose an irreducible component $Z$ of $\sing(X\cap H_{\bfv})$
with $\dim Z\ge t$.
If $Z$ meets $X_{\rm sm}$, then its dense open subset $Z\cap X_{\rm sm}$ lies in
$\Gamma_{[\bfv]}$, so $[\bfv]\in W_t$.
Now suppose $Z \subset \sing X$.
Then $\dim\sing X \ge \dim Z \ge t$,
where equality must hold in view of the assumption $t \ge \dim\sing X$.
Thus $Z$ is one of the
finitely many $t$-dimensional components of $\sing X$.
Its cone $\cone Z$ cannot be linear: Lemma~\ref{LEM:no-slop-singular-linear}
would give $t+1 = \dim \cone Z \le n/3$, contrary to $t>n/3$.
Thus the linear span $L$ of $\cone Z$ has dimension at least $t+2$.
Since $Z \subset H_{\bfv}$ by definition,
this constrains $\bfv$ to lie in the orthogonal complement $L^\perp$,
where $\dim L^\perp \le n-\dim L \le n-t-2$.
Taking the finite union over all possible
irreducible components $Z \subset \sing X$
proves \eqref{EQ:geometric-terminal-bound}.
\end{proof}

\begin{cor}\label{COR:general-n-exceptional-hyperplanes}
Assume $n\ge10$ and $X(\QQ)=\emptyset$.
Then
\begin{equation}
\label{eq:n10-sing-X}
\dim\Sing(X)\leq \lfloor(2n-6)/3\rfloor \le n-6.
\end{equation}
The set $B_{n-6}(\QQ)$ is finite if $n=10$,
and is empty if $n\ge11$. Moreover,
\begin{equation}\label{eq:rational-terminal-hyperplanes}
B_{n-5}(\QQ)=\emptyset.
\end{equation}
Geometrically, we have $B_{n-4} = \emptyset$ and
\begin{equation}\label{eq:terminal-hyperplane-dimensions}
1+\dim B_{n-6}\le4,
\qquad
1+\dim B_{n-5}\le3.
\end{equation}
\end{cor}

\begin{proof}
The inequality \eqref{eq:n10-sing-X} follows from \eqref{EQ:clean-slop-sing-bound},
Indeed, for $n\ge 12$ we have $(2n-6)/3 \le n-6$,
and for $n=10,11$ we have $\lfloor(2n-6)/3\rfloor = n-6$.
Since $n-6\ge(n+1)/3$,
taking $t=n-6,n-5$ in \eqref{EQ:geometric-terminal-bound} proves
\eqref{eq:terminal-hyperplane-dimensions}.
On the other hand,
the claim $B_{n-4} = \emptyset$ follows from \eqref{EQ:clean-slop-sing-bound}
and Theorem~\ref{Thm: ZaksBoundSingLocus}.
Separately,
Proposition~\ref{PROP:rational-exceptional-hyperplanes-general}
with $t=n-6,n-5$ gives
\[
1+\dim A_{n-6}
\le \max\left\{0,\left\lfloor\frac{12-n}{2}\right\rfloor\right\},
\qquad
1+\dim A_{n-5}
\le \max\left\{0,\left\lfloor\frac{9-n}{2}\right\rfloor\right\}.
\]
This proves the assertions about $\QQ$-points.
\end{proof}

\section{Pointwise bounds for  exponential sums}
\label{sec:pointwise}

Let $F\in \ZZ[x_1,\dots,x_n]$ be a cubic form,
and let $X \subset \PP^{n-1}_\QQ$ be the hypersurface $F=0$.
In this section, we assume that $n\geq 10$ and $X(\QQ)=\emptyset$.
By Proposition~\ref{Bro-integral-cone},
$X$ is geometrically integral and is not a cone.
Also, $\dim\sing X \le n-6$ by \eqref{eq:n10-sing-X}.

Recall the definition of the complete exponential sum $S_q(\bfv)$ from \eqref{eq:defSq}.
We will collect together various pointwise upper bounds for $S_q(\bfv)$,
particularly for prime powers $q=p^r$, in view of the multiplicativity property 
\eqref{eq:q-exact-CRT}.
As usual, any implied constants in our estimates are allowed to depend on $F$ in any way.

\subsection{Prime modulus}

Let $p$ be a prime and put 
$\pi_m(p)=\#\PP^m(\FF_p)=1+p+\cdots+p^m$,
for an integer $m\geq 0$.
Let $X_p\subset \PP^{n-1}_{\FF_p}$
denote the hypersurface $F=0$ modulo $p$.  For $\bfv\in\ZZ^n$, write
$\bar{\bfv}$ for its reduction modulo $p$. In this section we shall be concerned with estimating 
\[
S_p(\bfv)
=\sum_{a\in\FF_p^\ast}\sum_{\bfx\in\FF_p^n}
e_p\bigl(aF(\bfx)+\bfv\cdot\bfx\bigr).
\]
The following result, Proposition~\ref{prop:prime-modulus-pointwise},
provides crucial pointwise upper bounds for $S_p(\bfv)$.
Later, in Appendix~\ref{sec:n10-singular-support},
we will provide a complementary set of
stratified bounds for $S_p(\bfv)$,
which will be combined with Proposition~\ref{prop:prime-modulus-pointwise}
in \S\ref{sec:MarkV-construction}.

\begin{prop}
\label{prop:prime-modulus-pointwise}
We have the following estimates:
\begin{enumerate}[(i)]
\item If $p\mid\bfv$ then
$        S_p(\bfv)=S_p(\bm 0)\ll p^{n-3/2}$.

\item If $p\nmid\bfv$ and $n\geq 10$, then
$
S_p(\bfv)\ll p^{n-2}.
$
\end{enumerate}
\end{prop}

We begin with exact identities that convert the complete exponential sum into
point counting estimates  for $X_p$ and the hyperplane section  
$Y_{\bfv}=X_p\cap V(\bar{\bfv}\cdot\bfx)   \subset\PP^{n-2}_{\FF_p}$.
 Arguing as in Hooley \cite[Eq.~(47)]{hooley-w}, 
we have 
\begin{align}
    S_p(\bfv)
    &=p^2\bigl(\#Y_{\bfv}(\FF_p)-\pi_{n-3}(p)\bigr)
      -p\bigl(\#X_p(\FF_p)-\pi_{n-2}(p)\bigr)
    \label{eq:Sp-projective-v-nonzero}
\end{align}
when $p\nmid \bm{v}$. If $p\mid \bm{v}$, on the other hand, then a simpler orthogonality argument yields
\begin{equation}
    S_p(\bm 0)
    =p(p-1)\bigl(\#X_p(\FF_p)-\pi_{n-2}(p)\bigr).
    \label{eq:Sp-projective-v-zero}
\end{equation}
The following result will allow us to deal with degenerate cases in which $Y_{\bm{v}}$ is either a cone or fails to be geometrically integral.

\begin{lemma}
\label{lem:geometry-all-hyperplane-sections}
There is a finite set of primes, depending only on $F$, such that for every
prime outside this set and every $\bar{\bfv}\neq\bm 0$, the cubic hypersurface
$Y_{\bfv}$
is geometrically integral.  Furthermore, if $Y_{\bfv}$ is a cone over $\bar{\FF}_p$ with maximal vertex of 
projective dimension $s$, then
\begin{equation}
s\leq 1.
\label{eq:prime-sum-vertex-bound}
\end{equation}
\end{lemma}

\begin{proof}
After excluding finitely many primes,
spreading out \eqref{eq:n10-sing-X} gives
\begin{equation}
\dim\Sing(X_{p}) = \dim\Sing(X) \leq n-6.
\label{eq:prime-sum-sing-Xp}
\end{equation}
Likewise, if $V_{2,p}=\{\bfx\in\A^n_{\bar{\FF}_p}: \rank H(\bfx)\leq2\}$,
then spreading out Corollary~\ref{cor:hesrank} gives
\begin{equation}
\dim (V_{2,p} \cap \{F=0\}) = \dim (V_2 \cap \{F=0\}) \leq 2\cdot 2-2 = 2.
\label{eq:prime-sum-V2p}
\end{equation}
We may also suppose that $p>3$ and that $X_p$ is geometrically integral
of dimension $n-2$
and is not a cone,
by Proposition~\ref{Bro-integral-cone}.
Indeed, it is well known that geometric integrality
and dimension spread out,
whereas the set of vertices of $X$ spreads out by Remark~\ref{field-of-maximal-vertex}.

The hyperplane section $Y_{\bfv}$ is geometrically integral
by Theorem~\ref{Thm: ZaksBoundSingLocus},
because \eqref{eq:prime-sum-sing-Xp} gives
$$\dim X_p - \dim \sing(X_p)
\ge (n-2) - (n-6) = 4.$$
Suppose now that $Y_{\bfv}$ is a cone over $\bar{\FF}_p$
with maximal vertex $\PP^s$.  Over
$\bar{\FF}_p$, choose coordinates $(\bu,\bfz,t)$ such that $t=0$ cuts
out $\bar{\bm{v}}\cdot \bfx=0$, the variables $\bu=(u_0,\ldots,u_s)$ run along the
vertex, and the minimal base is defined by a cubic form $C(\bfz)$.  Then
\[
    F(\bu,\bfz,t)=C(\bfz)+tQ(\bu,\bfz,t)
\]
for a quadratic form $Q$.  At every point $(\bu,\bm 0,0)$ of the affine cone
over the vertex, all entries of the Hessian matrix $H(\bu,\bm 0,0)$ vanish except possibly
those in the $t$-row and the $t$-column.  Thus $\rank H(\bu,\bm 0,0)\leq2$ and so the 
affine $(s+1)$-plane over the vertex is contained in $V_{2,p}$.
Since the vertex is trivially also contained in $\{F=0\}$,
it follows from  \eqref{eq:prime-sum-V2p} that 
\[
s+1\leq 2,
\]
which completes the proof of the lemma. 
\end{proof}

\begin{proof}[Proof of Proposition~\ref{prop:prime-modulus-pointwise}]
It is clearly enough to work outside a finite set of primes depending on $F$.  
To begin with it follows from 
Lemma \ref{lem:Fq} that 
\begin{equation}\label{eq:fvf}
\left|    \#X_p(\FF_p)-\pi_{n-2}(p)\right|
       \ll_n p^{(n-1+\sigma_p)/2}\leq p^{n-7/2},
\end{equation}
since $\sigma_p=\dim\Sing(X_{p})\leq n-6$, 
by \eqref{eq:prime-sum-sing-Xp}.
Inserting this into  \eqref{eq:Sp-projective-v-zero} gives part (i).

Assume now that $\bar{\bfv}\neq\bm 0$.  By 
\eqref{eq:Sp-projective-v-nonzero} and \eqref{eq:fvf}, we have
\begin{equation}
    S_p(\bfv)
      =p^2\bigl(\#Y_{\bfv}(\FF_p)-\pi_{n-3}(p)\bigr)
       +O(p^{n-5/2}).
    \label{eq:prime-sum-reduced-to-Y}
\end{equation}
It remains to count points on $Y_{\bfv}$.
If $Y_{\bfv}$ is not a cone, the second estimate in
Lemma~\ref{lem:Fq} yields
\[
\left|    \#Y_{\bfv}(\FF_p)-\pi_{n-3}(p)\right|\ll_n p^{n-4}.
\]
Suppose now that $Y_{\bfv}$ is a cone.
Its defining cubic form over $\FF_p$
is absolutely irreducible by Lemma~\ref{lem:geometry-all-hyperplane-sections}.
Its maximal vertex is an $s$-plane defined over $\FF_p$, where $s\ge 0$.
After choosing an $\FF_p$-linear complement its minimal base is a geometrically
integral, non-conical cubic hypersurface
$Z_{\bfv}\subset\PP^{n-s-3}_{\FF_p}$,
defined by a cubic form in $n-s-2$ variables.
Let
$$
b=n-s-4.
$$
Projection away from the maximal vertex shows that
\[
    \#Y_{\bfv}(\FF_p)
       =\pi_s(p)+p^{s+1}\#Z_{\bfv}(\FF_p),
\]
and $\pi_{n-3}(p)=\pi_s(p)+p^{s+1}\pi_b(p)$.
If $b\geq2$, then Lemma~\ref{lem:Fq} gives
\[
   \left| \#Z_{\bfv}(\FF_p)-\pi_b(p)\right|\ll_n p^{b-1},
\]
from which it follows that 
\begin{equation}
\left|    \#Y_{\bfv}(\FF_p)-\pi_{n-3}(p)\right|
      \ll_n p^{s+1+b-1}=p^{n-4}.
    \label{eq:prime-sum-Y-good}
\end{equation}
Yet by \eqref{eq:prime-sum-vertex-bound} and Remark~\ref{field-of-maximal-vertex},
we have $s\le 1$
and $b=n-s-4\ge 2$ (in fact $b\ge 5$)
whenever $n\geq 10$.
Combining \eqref{eq:prime-sum-Y-good} with
\eqref{eq:prime-sum-reduced-to-Y} therefore proves part~(ii), uniformly for every
$\bar{\bfv}\neq\bm 0$.
\end{proof}

\subsection{Prime squares}
\label{sec:prime-square-strat}

For integers $t\ge 0$,
recall the definitions of $B_t$ and $W_t$
from \eqref{EQ:general-bad-hyperplanes} and \eqref{EQ:general-large-gauss-fibres},
respectively.
By \eqref{EQ:gauss-section-containment}, we have $W_t \subset B_t$.
Corollary~\ref{COR:general-n-exceptional-hyperplanes}
gives valuable information about $B_{n-6}(\QQ)$ and $B_{n-5}(\QQ)$.

The incidence correspondences $S$ and $\Gamma$
have natural analogues $\mathcal S_R$ and $\Gamma_R$
over any base ring $R$ in place of $\Qbar$.
It is clear from the definition of $\mathcal S$ in \eqref{EQ:universal-section-singular}
that $\mathcal S_R = \mathcal S_\ZZ \otimes R$ for any ring $R$.
Using the definition \eqref{def:closed-gauss-correspondence} of $\Gamma$
as a scheme-theoretic closure, it can also be shown that
$\Gamma_R = \Gamma_\ZZ \otimes R$ for
$R\in \{\Qbar,\QQ,\FF_p\}$ for all but finitely many primes $p$.
Indeed, flat base change for scheme-theoretic closures
\cite[\href{https://stacks.math.columbia.edu/tag/081I}{Tag 081I}]{TSta}
proves this for $R=\QQ$ and $R=\Qbar$.
Moreover, scheme-theoretic density
spreads out \cite[\href{https://stacks.math.columbia.edu/tag/0573}{Tag 0573}]{TSta},
proving the claim for $R=\FF_p$.

Let $\mathcal B_{t,R}$ (resp.~$\mathcal W_{t,R}$)
denote the set of fibres of $\mathcal S_R$ (resp.~$\Gamma_R$)
of dimension $\ge t$.
Then $\mathcal B_{t,R}$ and $\mathcal W_{t,R}$ are closed sets,
by upper semicontinuity for proper morphisms.
Let $C \otimes R$ denote the preimage of a set $C \subset \PP^{n-1}_\ZZ$
along the map $\PP^{n-1}_R \to \PP^{n-1}_\ZZ$.
Since fibre dimensions are preserved by base change
\cite[\href{https://stacks.math.columbia.edu/tag/05F8}{Tag 05F8}]{TSta},
we have $\mathcal B_{t,R} = \mathcal B_{t,\ZZ} \otimes R$ for any ring $R$,
and $\mathcal W_{t,R} = \mathcal W_{t,\ZZ} \otimes R$ for
$R\in \{\Qbar,\QQ,\FF_p\}$ for all but finitely many primes $p$.

Let $\mathcal B_{t,p}$, $\mathcal W_{t,p}$,
$\mathcal S_p$, and $\Gamma_p$
denote $\mathcal B_{t,\FF_p}$, $\mathcal W_{t,\FF_p}$,
$\mathcal S_{\FF_p}$, and $\Gamma_{\FF_p}$,
respectively.
\begin{lemma}
\label{lem:prime-square-stratification}
Let $p$ be any prime and let $t\ge 0$.
If $p\nmid\bfv$ and $[\bar{\bfv}]\notin \mathcal W_{t,p}(\FF_p)$,
then
\begin{equation}
\label{eq:p2-strat-bound}
S_{p^2}(\bfv)\ll_F p^{n+t+1}.
\end{equation}
Moreover, for every prime $p$ and every $\bfv\in\ZZ^n$,
\begin{equation}
\label{eq:p2-cutoff}
S_{p^2}(\bfv)\ll_F p^{2n-3}.
\end{equation}
\end{lemma}

\begin{proof}
It is enough to prove the estimates outside a fixed finite set of primes.
We may therefore assume that
$p>3$, that $X_p$ is geometrically integral, and that $\dim\Sing(X_p)\leq n-6$.

Writing $a=b+pc$ and $\bfx=\bfy+p\bfz$,
and then summing over $c\in\FF_p$ and $\bfz\in\FF_p^n$, gives
\[
S_{p^2}(\bfv)
= \sum_{a\in (\ZZ/p^2\ZZ)^\ast}
\sum_{\bfx\in (\ZZ/p^2\ZZ)^n}e_{p^2}\left(aF(\x)+ \bfv\cdot \x\right)
= p^{n+1}
\sum_{b\in\FF_p^\ast}
\sum_{\substack{\bfy\in\FF_p^n\\
F(\bfy)=0\\
b\nabla F(\bfy)=-\bar{\bfv}}}
e_{p^2}\bigl(bF(\bfy)+\bfv\cdot\bfy\bigr),
\]
since $aF(\x)+ \bfv\cdot \x$ may be rewritten as
$$(b+pc)F(\bfy+p\bfz) + \bfv\cdot (\bfy+p\bfz)
\equiv  bF(\bfy) + \bfv\cdot \bfy
+ p\left(b\nabla F(\bfy)\cdot \bfz
+ cF(\bfy) + \bfv\cdot \bfz\right)  \bmod{p^2}.$$
The final sum over $b$ and $\bfy$ is well-defined:
in view of the conditions $F(\bfy)=0$ and $b\nabla F(\bfy)=-\bar{\bfv}$,
each summand $e_{p^2}\bigl(bF(\bfy)+\bfv\cdot\bfy\bigr)$ 
only depends on $b,\bfy\bmod{p}$ and $\bfv\bmod{p^2}$.

By the triangle inequality and the Lang--Weil estimate,
\begin{equation*}
|S_{p^2}(\bfv)|
\le p^{n+1}
\sum_{b\in\FF_p^\ast}
\sum_{\substack{\bfy\in\FF_p^n\\
F(\bfy)=0\\
b\nabla F(\bfy)=-\bar{\bfv}}} 1
= p^{n+1}
(p-1)\,
\#V(\FF_p)
\ll_n p^{n+2+\dim V},
\end{equation*}
where $V = \sing(X_{p,\rm sm} \cap \{\bar{\bfv}\cdot\bfx=0\})$ if $p\nmid \bfv$,
and $V = \sing(\cone X_p)$ if $p\mid \bfv$.
If $p\nmid\bfv$ and $[\bar{\bfv}]\notin \mathcal W_{t,p}(\FF_p)$,
then $\dim V \le \dim \Gamma_{p,[\bar{\bfv}]} \le t-1$,
by the analogue of \eqref{EQ:smooth-sing-containment} over $\FF_p$.
Thus \eqref{eq:p2-strat-bound} follows.
Moreover, Theorem~\ref{Thm: ZaksBoundSingLocus} implies
$\dim V \le \dim \sing(X_p \cap \{\bar{\bfv}\cdot\bfx=0\})
\le \dim \sing(X_p) + 1
\le n-5$ if $p\nmid \bfv$.
If $p\mid\bfv$,
then $\dim V = 1 + \dim \sing(X_p) \le n-5$.
Either way, \eqref{eq:p2-cutoff} follows.
\end{proof}

\subsection{Higher prime power modulus}

Using the goodness of $F$,
it is possible to obtain a pointwise bound
$S_{p^r}(\bfv)\ll p^{r(1+5n/6+\ve)}$
by following the proof of \cite[Eq.~(7.1)]{HB07},
or Bernert's version in \cite[\S3]{Ber25}.
The van der Corput differencing process is insensitive to any linear or quadratic terms, and $S_{p^r}(\bfv)$ has only
a cubic term $F(\bfx)$
and a linear term $\bfv\cdot \bfx$.
However, for $\bfv\ne \bm0$,
it turns out that it will be more useful for us to bound
$S_{p^r}(\bfv)$ on average over $\bfv$ when $r\ge 3$.
Nonetheless, for $\bfv=\bm0$ and for parts of the averaging procedure over $\bfv\ne \bm0$,
it will be convenient for us to draw on the recent work of
Bernert \cite{Ber25}.

\begin{thm}
\label{Bernert-n10-sing-series}
If $n\ge 10$ and $F$ is good,
then $\mathfrak S := \sum_{q\ge 1} q^{-n} S_q(\bm0)$
converges absolutely to a positive real number,
and satisfies
\[
\mathfrak S=\prod_p\chi_p,
\]
where
\[
\chi_p
:=\sum_{k\ge0}p^{-kn}S_{p^k}(\bm0)
=\lim_{j\to\infty}p^{-j(n-1)}
\#\{\bfx\bmod p^j:F(\bfx)\equiv0\bmod p^j\}
>0.
\]
\end{thm}

\begin{proof}
This follows from \cite[\S2, Theorem~1]{Ber25}.
\end{proof}

The main point in the proof of Theorem~\ref{Bernert-n10-sing-series} is the absolute convergence
\begin{equation}
\label{Bernert-n10-absolute}
\sum_{q\ge 1} q^{-n} |S_q(\bm0)| < \infty.
\end{equation}
We remark that Bernert's result is stronger,
in that he takes the absolute value \emph{inside} the $a$-sum in $S_q(\bm0)$,
rather than \emph{outside} the $a$-sum as in \eqref{Bernert-n10-absolute}.
Also,
as Bernert remarks, Theorem~\ref{Bernert-n10-sing-series} was already proven for $n\ge 11$
by Heath-Brown \cite{HB07}.

\subsection{Hensel lifting}

It will be useful for us to develop a flexible version of Hensel lifting.
For integers $A,T\ge 1$
and vectors $\bfy\in \ZZ^n$, put
\begin{equation}
\label{eq:GAT-max-sum}
\mathcal G_{A,T}(\bfy)
:= \max_{\substack{\alpha\in (\ZZ/T\ZZ)^\ast\\
\bfd\in (\ZZ/T\ZZ)^n}}
\left|
\sum_{\bfz\in (\ZZ/T\ZZ)^n}
e_T(\bfd\cdot\bfz+\alpha Q_{\bfy}(\bfz)+\alpha A F(\bfz))
\right|,
\end{equation}
where $Q_{\bfy}(\bfz):=\frac12\bfz^tH(\bfy)\bfz$
is an integral quadratic form defined in terms of $F$ and $\bfy$.
We proceed by proving the following three results. 

\begin{lemma}
\label{LEM:max-G-multiplicative}
The quantity $\mathcal G_{A,T}(\bfy)$ is multiplicative in $T$.
That is,
\begin{equation}
\label{EQ:max-G-multiplicative}
\mathcal G_{A,T_1T_2}(\bfy)
= \mathcal G_{A,T_1}(\bfy)
\mathcal G_{A,T_2}(\bfy)
\end{equation}
for all $A,T_1,T_2\ge 1$ with $\gcd(T_1,T_2)=1$.
\end{lemma}

\begin{lemma}
\label{LEM:max-G-A-invariant}
As a function of $\bfy$,
the quantity $\mathcal G_{A,T}(\bfy)$
only depends on $\bfy \bmod{\gcd(A,T)}$.
\end{lemma}

\begin{lemma}
\label{lem:q-flexible-lifting}
Let $A,T\ge 1$ and $q = A^2T$.
Then
\begin{equation}
\label{eq:q-flexible-pointwise}
|S_q(\bfv)|
\le A^{n+1} T
\sum_{a\in (\ZZ/A\ZZ)^\ast}
\sum_{\substack{\bfy\bmod A\\F(\bfy)\equiv0\bmod A
\\ a\nabla F(\bfy)+\bfv\equiv \bm0\bmod A}} \mathcal G_{A,T}(\bfy).
\end{equation}
\end{lemma}

Lemma~\ref{LEM:max-G-A-invariant} ensures that each summand on the right-hand side of
\eqref{eq:q-flexible-pointwise} is well-defined.

\begin{proof}
[Proof of Lemma~\ref{LEM:max-G-multiplicative}]
Let
\begin{equation}
\label{z-sum-in-max-G}
\mathcal{S}_{A,T}(\bfd,\alpha,\bfy)
:= \sum_{\bfz\bmod T} e_T(\bfd\cdot\bfz+\alpha Q_{\bfy}(\bfz)+\alpha A F(\bfz))
\end{equation}
be the $\bfz$-sum in \eqref{eq:GAT-max-sum}.
If $\gcd(T_1,T_2)=1$,
the Chinese remainder theorem implies
\begin{equation*}
\mathcal{S}_{A,T_1T_2}(\bfd,\alpha,\bfy)
= \mathcal{S}_{A,T_1}(\bar{T_2}\bfd,\bar{T_2}\alpha,\bfy)
\mathcal{S}_{A,T_2}(\bar{T_1}\bfd,\bar{T_1}\alpha,\bfy).
\end{equation*}
By the Chinese remainder theorem, the pairs
$(\bar{T_2}\bfd,\bar{T_2}\alpha)\bmod T_1$
and
$(\bar{T_1}\bfd,\bar{T_1}\alpha)\bmod T_2$
range independently over their respective parameter sets.
Thus the two local maxima can be attained simultaneously,
and taking the maximum of the absolute value gives the
claimed equality.
\end{proof}

\begin{proof}[Proof of Lemma~\ref{LEM:max-G-A-invariant}]
The definition gives $T$-periodicity in $\bfy$.
To prove $A$-periodicity, put
$$
P_{\bfy}(\bfz):=Q_{\bfy}(\bfz)+AF(\bfz).
$$
For $\bfh\in\ZZ^n$, Taylor expansion gives
$$
P_{\bfy+A\bfh}(\bfz)
=P_{\bfy}(\bfz+\bfh)
-\boldsymbol{\ell}\cdot\bfz-P_{\bfy}(\bfh),
$$
where
$\boldsymbol{\ell}:=H(\bfy)\bfh+A\nabla F(\bfh)$.
Translating the summation variable in
\eqref{z-sum-in-max-G} therefore gives
$$
|\mathcal S_{A,T}(\bfd,\alpha,\bfy+A\bfh)|
=
|\mathcal S_{A,T}(\bfd-\alpha\boldsymbol{\ell},
\alpha,\bfy)|.
$$
For each $\alpha$, the map
$\bfd\mapsto\bfd-\alpha\boldsymbol{\ell}$
is a bijection modulo $T$.
Taking maxima proves
$\mathcal G_{A,T}(\bfy+A\bfh)=\mathcal G_{A,T}(\bfy)$.
Together with $T$-periodicity, this gives the desired
$\gcd(A,T)$-periodicity.
\end{proof}

\begin{proof}
[Proof of Lemma~\ref{lem:q-flexible-lifting}]
Let \(M:=AT\).
Then \(q=A^2T=AM\),
$\rad(q)=\rad(AT)=\rad(M)$,
and $M^2=A^2T^2=qT\equiv 0\bmod{q}$.
For $a,\bfx$ as in \eqref{eq:defSq},
write
\[
 \bfx=\bfx_0+M\bfh_0,\qquad
 a=a_0+Me,
\]
with $0\le \bfx_0<M$ and $0\le a_0<M$
running over a complete set of representatives modulo $M$,
and similarly with $0\le \bfh_0<A$ and $0\le e<A$.
Since \(\rad(M)=\rad(q)\),
the condition $\gcd(a,q)=1$ is equivalent to $\gcd(a_0,M)=1$.
Taylor expansion of $F(\bfx_0+M\bfh_0)$ gives
\begin{equation*}
a F(\bfx)+\bfv\cdot\bfx
\equiv (a_0+Me) F(\bfx_0) + a_0 M\nabla F(\bfx_0)\cdot\bfh_0
+ \bfv\cdot(\bfx_0+M\bfh_0) \bmod{q},
\end{equation*}
since $M^2 \equiv 0\bmod{q}$.
Orthogonality in \(e,\bfh_0\bmod{A}\) gives
\begin{equation}
\label{eq:q-first-lift}
S_q(\bfv)=A^{n+1}
\sum_{\substack{0\le a_0<M \\ \gcd(a_0,M)=1}}
\sum_{\substack{0\le \bfx_0<M\\
F(\bfx_0)\equiv 0\bmod{A}\\
a_0\nabla F(\bfx_0)+\bfv\equiv \bm0\bmod{A}}}
e_q(a_0 F(\bfx_0)+\bfv\cdot\bfx_0).
\end{equation}
For all $\bfx_0$ in the support of \eqref{eq:q-first-lift},
write \(\bfx_0=\bfy+A\bfz\), with
$0\le \bfy<A$ and $0\le \bfz<T$.
Then $$\bfd := \frac{a_0\nabla F(\bfy)+\bfv}{A} \in \ZZ^n,$$
since $a_0\nabla F(\bfx_0)+\bfv\equiv \bm0\bmod{A}$.
Taylor expansion of $F(\bfy+A\bfz)$ gives
\begin{equation*}
\begin{split}
\frac{e_q(a_0 F(\bfx_0)+\bfv\cdot\bfx_0)}{e_q(a_0 F(\bfy) + \bfv\cdot\bfy)}
&= e_q(a_0 A\nabla F(\bfy)\cdot\bfz
+ a_0 A^2Q_{\bfy}(\bfz) + a_0 A^3F(\bfz)
+ A\bfv\cdot \bfz) \\
&= e_T(\bfd\cdot\bfz+a_0 Q_{\bfy}(\bfz)+a_0 A F(\bfz))
\end{split}
\end{equation*}
in view of the identity $e_q(A^2x) = e_{A^2T}(A^2x) = e_T(x)$.
By the triangle inequality on \eqref{eq:q-first-lift},
\begin{equation*}
|S_q(\bfv)|
\le A^{n+1}
\sum_{\substack{0\le a_0<M \\ \gcd(a_0,M)=1}}
\sum_{\substack{0\le \bfy<A\\
F(\bfy)\equiv 0\bmod{A}\\
a_0\nabla F(\bfy)+\bfv\equiv \bm0\bmod{A}}}
|\mathcal{S}_{A,T}(\bfd,a_0,\bfy)|,
\end{equation*}
where $\mathcal{S}_{A,T}$ was defined in \eqref{z-sum-in-max-G}.
By definition, $|\mathcal{S}_{A,T}|
\le \mathcal G_{A,T}(\bfy)$.
Since $\mathcal G_{A,T}(\bfy)$ is independent of $a_0$,
and the condition $a_0\nabla F(\bfy)+\bfv\equiv \bm0\bmod{A}$
depends only on $a_0\bmod{A}$,
the desired inequality \eqref{eq:q-flexible-pointwise} follows
from the fact that each fibre of the map $(\ZZ/M\ZZ)^\ast \to (\ZZ/A\ZZ)^\ast$
has cardinality $\varphi(M)/\varphi(A)
= \varphi(AT)/\varphi(A) \le T$.
\end{proof}

\subsection{Gauss sums}

Given $\bm{a}\in \ZZ^n$ and a positive integer $d$, the quantity
\begin{equation}\label{eq:N_d}
N_d(\bm{a})=\#\left\{ \bm{y}\in (\ZZ/d\ZZ)^n: H(\bm{a})\bm{y}\equiv \bm{0} \bmod{d}\right\}
\end{equation}
will play a pivotal role in estimating \eqref{eq:GAT-max-sum},
where $H(\bm{a})$ is the Hessian matrix associated to $F$. 

\begin{lemma}
\label{Bernert-Smith-bound}
Let $r(\bfa)$ be the rank of $H(\bfa)$ over $\QQ$.
Then $N_d(\bfa) \ll_F d^{n-r(\bfa)} |\bfa|^{r(\bfa)}$.
\end{lemma}

\begin{proof}
This is \cite[Eq.~(3.1)]{Ber25},
applied to the cubic form $C := 6F$.
It does not require goodness.
For $C=6F$,
Bernert's notation $M,\bfh,q$ corresponds to our notation $H,\bfa,d$,
and his bilinear forms $B_i(\bfx,\bfy)$
are simply the entries of the vector $M(\bfx)\bfy = H(\bfx)\bfy$.
\end{proof}

\begin{lemma}
\label{LEM:square-and-diff}
If $A,T\ge 1$ are integers with $T\mid A$, then
\begin{equation*}
\mathcal G_{A,T}(\bfa)
\le T^{n/2} N_T(\bfa)^{1/2}.
\end{equation*}
\end{lemma}

\begin{proof}
Since $T\mid A$,
the definition of $\mathcal G_{A,T}$ in \eqref{eq:GAT-max-sum} simplifies to
\[
\mathcal G_{A,T}(\bfa)
= \max_{\substack{\alpha\in (\ZZ/T\ZZ)^\ast\\
\bfd\in (\ZZ/T\ZZ)^n}}
\left|
\sum_{\bfz\in (\ZZ/T\ZZ)^n}
e_T(\bfd\cdot\bfz+\alpha Q_{\bfa}(\bfz))
\right|,
\]
where $Q_{\bfa}(\bfz):=\frac12\bfz^tH(\bfa)\bfz$.
On squaring and differencing the $\bfz$-sum, exactly as in the close of the proof of
\cite[Lemma~11]{41}, we are led to the desired conclusion.
\end{proof}

\section{Squarefull exponential sums on average}
\label{sec:squarefull-average}

Throughout this section, we assume $n\ge10$ and $X(\QQ)=\emptyset$;
in particular, $F$ is good.

\subsection{Squarefull exponential sums: averaging over $\bm{v}$} 

Any squarefull integer $r$ admits a factorisation $r=c^2d$ for square-free $d$ such that $d\mid c$. Let $V\geq 1$. In this section we consider the averages
$$
\sum_{\bm{v}\in \mathcal{V}} |S_r(\bm{v})|,
$$
for an arbitrary set
$\mathcal V\subset\ZZ^n\cap(\bfv_0+[-V,V]^n)$,
uniformly in $\bfv_0\in\RR^n$.
Let 
\begin{equation}\label{Defi:d1}
d_1=\prod_{\substack{p\mid d \\ v_p(c)=1}}p, \qquad d_2=\frac{d}{d_1}.
\end{equation}
In particular, we have 
$d_1d_2^2\mid c$. Thus we can 
write $c=d_1c_2$, where $\gcd(d_1,c_2)=1$ and $d_2^2\mid c_2$. 
We begin by proving the following estimate.

\begin{lemma}\label{lem:c/d-bound}
Let $c,d$ be positive integers such that $d$ is square-free and $d\mid c$.
Let $\bm{k}\in \ZZ^n$ such that $d\mid F(\bm{k})$. 
Then
$$
\#\left\{ \bm{x}\in (\ZZ/c\ZZ)^n: 
\begin{array}{l}
\bm{x}\equiv \bm{k}\bmod{d}\\
F(\bm{x})\equiv 0\bmod{c}
\end{array}
\right\}\ll \left(\frac{c}{d}\right)^{n-1+\ve} 
\gcd(d_2,\nabla F(\bm{k}))\gcd(d_2,\bm{k}),
$$
for any $\ve>0$, where $d_2=d/d_1$ and $d_1$ is defined in \eqref{Defi:d1}.
\end{lemma}

\begin{proof}
Let $N(c,d)$ be the quantity to be counted in the lemma.
By the Chinese remainder theorem, it suffices to treat the case where
$c=p^\ell$ and $d=p^j$, where $j\in\{0,1\}$ and $\ell\ge j$.
If $j=0$, additive orthogonality gives
\begin{equation*}
N(p^\ell,1)
= p^{-\ell} \sum_{1\le a\le p^\ell} \sum_{\bfx\bmod{p^\ell}} e_{p^\ell}(aF(\bfx))
= p^{-\ell} \sum_{h=0}^\ell p^{(\ell-h)n}S_{p^h}(\bm0)
= p^{(n-1)\ell} \sum_{h=0}^\ell p^{-hn}S_{p^h}(\bm0)
\end{equation*}
where $h = \ell - v_p(a)$.
Thus $N(p^\ell,1)
\le p^{(n-1)\ell} \sum_{q\ge 1} q^{-n} |S_q(\bm0)|
\ll_F p^{(n-1)\ell}$
by \eqref{Bernert-n10-absolute}.

Suppose $j=1$.
For $\ell=1$ we have $N(p,p)=1$.
Now suppose $\ell\ge 2$.
For any fixed prime $p$ we have
$N(p^\ell,p) \le N(p^\ell,1)
\ll_F p^{(n-1)\ell} \ll_p p^{(n-1)(\ell-1)}$.
At every other prime we have $p>3$ and
\begin{equation}\label{eq:rank-two-on-cubic}
F(\bm u)=0,\quad \bm u\in \bar{\FF}_p^n \setminus \{\bm0\}
\quad\Longrightarrow\quad \operatorname{rank}H(\bm u)\ge2,
\end{equation}
by spreading out the claim $V_1 \cap \cone X = \{\bm0\}$ of Corollary~\ref{cor:hesrank}.

If $p\mid\bm k$, homogeneity gives
$N(p^2,p) = p^n = p^{(\ell-1)(n-1)+1}$ for $\ell=2$, and
\[
N(p^\ell,p)=p^{2n}N(p^{\ell-3},1)
\ll_F p^{2+2(n-1)} p^{(\ell-3)(n-1)}
= p^{(\ell-1)(n-1)+2}
\]
for $\ell\ge3$.
If $p\nmid \bm k$ and $p\nmid\nabla F(\bm k)$, Hensel lifting gives
$N(p^\ell,p) = p^{(\ell-1)(n-1)}$.

It remains to consider the case where $p\nmid\bm k$ and $p\mid\nabla F(\bm k)$.
Write $\bm x=\bm k+p\bm y$.
If $p^2\nmid F(\bm k)$,
then $N(p^\ell,p)=0$ for all $\ell\ge 2$.
Now suppose $p^2\mid F(\bm k)$.
Then
\[
p^{-2}F(\bm k+p\bm y)
=p^{-2}F(\bm k)+p^{-1}\nabla F(\bm k)\cdot\bm y
+\tfrac12\bm y^tH(\bm k)\bm y+pF(\bm y)
\in \ZZ_p[\bfy].
\]
Here $\rank H(\bm k \bmod{p}) \ge 2$ by \eqref{eq:rank-two-on-cubic}.
Applying Lemma~\ref{LEM:rank-two-root-lifting}
with $s=\ell-2\ge 0$ therefore gives
\[
N(p^\ell,p) = p^n \#\{\bfy\bmod{p^s}: p^{-2}F(\bm k+p\bm y) \equiv 0\bmod{p^s}\}
\le p^n (\ell-1)p^{(\ell-2)(n-1)}.
\]
This is satisfactory, since $\ell-1 \ll p^{(\ell-1)\ve}$
and $p^n p^{(\ell-2)(n-1)} = p^{(\ell-1)(n-1) + 1}$.
\end{proof}

\begin{lemma}
\label{LEM:rank-two-root-lifting}
Suppose $p$ is odd.
Let $f_i\in \ZZ_p[\bfy]$ be homogeneous of degree $i$, for $0\le i\le 3$.
Suppose $\rank(f_2\bmod{p}) \ge 2$.
Let $g:=f_0+f_1+f_2+pf_3$.
Then, for every $s\ge0$,
\begin{equation}\label{eq:rank-two-root-lifting}
\#\{\bm y\bmod p^s:g(\bm y)\equiv0\bmod p^s\}
\le (s+1)p^{s(n-1)}.
\end{equation}
\end{lemma}

\begin{proof}
Since $p\ne 2$, we may diagonalise the quadratic form $f_2$ over $\ZZ_p$.
After permuting coordinates if necessary, we can
assume the coefficients are listed in increasing order of valuation.
Then $f_2 = \sum_{1\le i\le n} c_iy_i^2$
with $n\ge 2$ and $v_p(c_1)=v_p(c_2)=0$, say.
Fix $y_3,\ldots,y_n$ modulo $p^s$. After collecting homogeneous
terms in $y_1,y_2$, the resulting binary polynomial again has
the form $h_0+h_1+h_2+ph_3$, with each $h_i$ homogeneous of
degree $i$, and
$h_2\equiv c_1y_1^2+c_2y_2^2\bmod p$.
Thus its quadratic part has rank two modulo $p$. The binary
case of the lemma gives at most $(s+1)p^s$ solutions for each
of the $p^{s(n-2)}$ choices of the remaining variables,
proving the required bound.

For $n=2$, we induct on $s\ge 0$.
If $s\le 1$ the bound \eqref{eq:rank-two-root-lifting} is well known, so assume $s\ge 2$.
Suppose again that $f_2 = c_1y_1^2+c_2y_2^2$.
By completing the square, we may assume that $p\mid f_1$.
If $p\nmid f_0$, then $g\equiv 0\bmod{p}$ is a smooth affine conic,
so $\#\{\bfy\in \FF_p^2: g=0\} \le p+1 \le (s+1)p$,
and \eqref{eq:rank-two-root-lifting} holds by Hensel's lemma.
If $p\mid f_0$, then $g\equiv 0\bmod{p}$ is smooth away from the origin.
The contribution to \eqref{eq:rank-two-root-lifting} from $\bfy\equiv \bm0\bmod{p}$
is zero unless $p^2\mid f_0$, in which case it
is at most $p^2 (s-1)p^{s-2}$, by the $s-2$ case of \eqref{eq:rank-two-root-lifting}
for the polynomial
$$p^{-2} g(p\bfz) = p^{-2}f_0 + p^{-1}f_1(\bfz) + f_2(\bfz) + p^2f_3(\bfz).$$
The contribution to \eqref{eq:rank-two-root-lifting} from $\bfy\not\equiv \bm0\bmod{p}$
is at most $p^{s-1} (1+1)p$, by Hensel's lemma and
the $s=1$ case of \eqref{eq:rank-two-root-lifting}.
Adding the two contributions gives \eqref{eq:rank-two-root-lifting}.
\end{proof}

We will apply Lemma~\ref{lem:c/d-bound} twice
in this paper,
each time together with the same auxiliary bound,
Lemma~\ref{LEM:weighted-rank-locus} below,
whose statement requires some notation.
For \(0\leq \varrho\leq n\), let
\begin{equation}
\delta_n(\varrho) :=
\begin{cases}
0&0\le \varrho\le 1,\\
\min\{2\varrho-2,\lfloor n/3\rfloor+\varrho-1,n-2\}&2\leq \varrho<2\lceil n/3\rceil,\\
n-1&\varrho\geq2\lceil n/3\rceil.
\end{cases}
\label{eq:q-delta}
\end{equation}
Then for all primes $p$,
\begin{equation}
\label{eq:q-first-rank}
\#\{\bfx\in\FF_p^n:F(\bfx)=0,\ \rank H(\bfx)=\varrho\}
\ll_F p^{\delta_n(\varrho)},
\end{equation}
by the Lang--Weil estimate
and Corollary~\ref{cor:hesrank}.
Let
\begin{equation}
\gamma_n(\varrho)
:= \bm{1}_{\delta_n(\varrho)\le n-5}
+ \bm{1}_{\varrho=0}.
\label{eq:q-gamma}
\end{equation}

\begin{lemma}
\label{LEM:weighted-rank-locus}
If $0\le \varrho\le n$, then
$$\sum_{\substack{\bm u\in\FF_p^n\\F(\bm u)=0\\
\rank H(\bm u) = \varrho}}
\gcd(p,\nabla F(\bm u))\gcd(p,\bm u)
\ll_F p^{\delta_n(\varrho) + \gamma_n(\varrho)}.$$
\end{lemma}

\begin{proof}
The case $\varrho=0$ is immediate from \eqref{eq:q-first-rank},
since $\delta_n(0)=0 \le n-5$ and $\gamma_n(0)=2$.
Now suppose $\varrho\ge 1$.
Then $H(\bm u)$ is nonzero, so $\bm u\not\equiv \bm0 \bmod{p}$ and $\gcd(p,\bm u) = 1$.
The case $\gamma_n(\varrho)\ge 1$ is then immediate from \eqref{eq:q-first-rank}.
Finally, suppose $\gamma_n(\varrho)=0$.
Then $\delta_n(\varrho)\ge n-4$ by \eqref{eq:q-gamma}.
In view of \eqref{eq:n10-sing-X}, the Lang--Weil estimate gives
$$\#\{\bm u\bmod{p}: \nabla F(\bm u) \equiv 0\bmod{p}\} \ll_F p^{n-5}
\le p^{\delta_n(\varrho)-1}.$$
On the other hand, there are $\ll p^{\delta_n(\varrho)}$ smooth residues $\bm u$ at hand, by \eqref{eq:q-first-rank}.
After multiplying by $\gcd(p,\nabla F(\bm u))$,
we find that both the singular and smooth contributions are satisfactory.
\end{proof}
	
Let us now record an initial  upper bound for $S_r(\bm{v})$, in which we recall the factorisation $r=c^2d$ of any squarefull $r$, with square-free $d$  such that $d\mid c$.

\begin{lemma}\label{lem:starT}
Let $r=c^2d$ be squarefull. Then 
$$
|S_r(\bm{v})|\leq 
c^{n+1}d^{n/2+1}
\sum_{a\in (\ZZ/c\ZZ)^*} \sum_{
\substack{\bm{a}\in (\ZZ/c\ZZ)^n\\ 
c\mid F(\bm{a})\\
c\mid (a\nabla F(\bm{a})+\bm{v})}} N_d(\bm{a})^{1/2},
$$
where $N_d(\bm{a})$ is given by \eqref{eq:N_d}.
\end{lemma}

\begin{proof}
Apply Lemmas~\ref{lem:q-flexible-lifting} and~\ref{LEM:square-and-diff}
with $A := c$ and $T := d$, using the fact that $d\mid c$.
\end{proof}

\begin{lemma}\label{lem:frog}
We have
\begin{equation}\label{eq:weighted-prime-smith-mass}
\sum_{\substack{\bm u\in\FF_p^n\\F(\bm u)=0}}
N_p(\bm u)^{1/2}
\gcd(p,\nabla F(\bm u))\gcd(p,\bm u)
\ll_F p^{3n/2 - \lceil n/3\rceil - 1}.
\end{equation}
\end{lemma}

\begin{proof}
Lemma~\ref{LEM:weighted-rank-locus} bounds the left-hand side by
$\sum_{0\le \varrho\le n} p^{\delta_n(\varrho) + \gamma_n(\varrho)} p^{(n-\varrho)/2}$.
For $\varrho\ge 2\lceil n/3\rceil$,
we have $\delta_n(\varrho) = n-1$ and $\gamma_n(\varrho) = 0$, so
\begin{equation*}
\delta_n(\varrho)+\gamma_n(\varrho)
+\frac{n-\varrho}{2}
= n-1 + \frac{n-\varrho}{2}
\le \frac{3n}{2}-\lceil n/3\rceil-1.
\end{equation*}
For $\varrho=0$, we have $\delta_n(\varrho)=0$ and $\gamma_n(\varrho)=2$,
so $\delta_n(\varrho)+\gamma_n(\varrho) + (n-\varrho)/2
= 2 + n/2 \le 3n/2 - \lceil n/3\rceil - 1$,
since $n\ge 10$.
For $1\le\varrho<2\lceil n/3\rceil$, definitions
\eqref{eq:q-delta} and \eqref{eq:q-gamma} give
\[
\delta_n(\varrho)-\varrho\le \lfloor n/3\rfloor - 1,
\qquad
\delta_n(\varrho)+2\gamma_n(\varrho)
= \delta_n(\varrho) + 2\cdot \bm{1}_{\delta_n(\varrho)\le n-5} \le n-2.
\]
Consequently,
\[
\begin{aligned}
\delta_n(\varrho)+\gamma_n(\varrho)+\frac{n-\varrho}{2}
&=\frac n2+
\frac{\delta_n(\varrho)-\varrho}{2}
+\frac{\delta_n(\varrho)+2\gamma_n(\varrho)}{2}\\
&\leq n+\frac{\lfloor n/3\rfloor-3}{2}\\
&\leq\frac{3n}{2}-\lceil n/3\rceil-1,
\end{aligned}
\]
where the last inequality follows from
$\lfloor n/3\rfloor+2\lceil n/3\rceil\leq n+1$.
\end{proof}

In what follows it will be convenient to set
\begin{equation}\label{eq:Nc-V}
E_c(\mathcal{V})=\max_{\bm{k}\in (\ZZ/c\ZZ)^n} 
\#\left\{\bfv\in \mathcal{V}: \bm{v}\equiv \bm{k} \bmod{c}\right\},
\end{equation}
for any integer $c>0$, 
for which we have the trivial bound $E_c(\mathcal{V})\ll (1+V/c)^n.$ 
Let 
\begin{equation}\label{Eq: Defi.Lc,d,V}
L_j(c,d,\mathcal{V}) = 		\sum_{\substack{\bfv\in \mathcal{V}}}
c^{n+1}d^{n/2+1}
\sum_{a\in (\ZZ/c\ZZ)^*} \sum_{
\substack{\bm{a}\in (\ZZ/c\ZZ)^n\\ 
c\mid F(\bm{a})\\
c\mid (a\nabla F(\bm{a})+\bm{v})}} N_d(\bm{a})^{j/2},
\end{equation}
for integer $j\geq 0$. It follows from Lemma \ref{lem:starT} that
\begin{equation}\label{eq:from-lemstar}
\sum_{\substack{\bfv\in \mathcal{V}}} |S_{r}(\bm{v})|\leq L_1(c,d,\mathcal{V}).
\end{equation}
We are now ready to prove
the following result.

\begin{lemma}\label{lem:cube-sum1}
Let $d\mid c$ with $d$ square-free.
We have
\[
L_1(c,d,\mathcal V)
\ll_{F,\varepsilon}
c^{2n+1+\varepsilon}d^{n-\lceil n/3\rceil+1}E_c(\mathcal V).
\]
\end{lemma}

\begin{proof}
Interchanging the sums in \eqref{Eq: Defi.Lc,d,V} gives
\[
L_1(c,d,\mathcal V)
\le c^{n+2}d^{n/2+1}E_c(\mathcal V)
\sum_{\substack{\bm a\bmod c\\c\mid F(\bm a)}}N_d(\bm a)^{1/2}.
\]
Splitting the last sum into residue classes modulo $d$ and applying
Lemma~\ref{lem:c/d-bound} shows that it is at most
\[
\ll_{F,\varepsilon}(c/d)^{n-1+\varepsilon}
\sum_{\substack{\bm u\bmod d\\d\mid F(\bm u)}}
N_d(\bm u)^{1/2}
\gcd(d_2,\nabla F(\bm u))\gcd(d_2,\bm u).
\]
By the Chinese remainder theorem and Lemma~\ref{lem:frog}, the last
sum is $\ll_{F,\varepsilon}d^{3n/2 - \lceil n/3\rceil - 1+\varepsilon}$.
Combining these estimates proves the assertion.
\end{proof}

We now provide an alternative estimate based on Poisson summation.
Put
\begin{equation}\label{eq:poisson-weight-exponents}
 a_n=\frac n2-\lceil n/3\rceil-1,\qquad b_n=\frac{\lfloor n/3\rfloor-1}{2}.
\end{equation}
For $n\ge10$ we have $0\le a_n\le b_n$.
The factors $d_1,d_2$ retain the meaning of \eqref{Defi:d1}.

\begin{lemma}\label{lem:onion}
Let $r=c^2d$, where $d\mid c$ is square-free, and let
$\mathcal V\subset\ZZ^n\cap(\bfv_0+[-V,V]^n)$, with $V\ge1$.
Uniformly in $\bfv_0\in\RR^n$, we have
\[
 \sum_{\bfv\in\mathcal V}|S_r(\bfv)|
 \le L_1(c,d,\mathcal V)
 \ll_{F,\ve}
 r^{n/2+1+\ve}d_1^{a_n}d_2^{b_n}(V+r^{1/3})^n.
\]
\end{lemma}

\begin{proof}
The first inequality is \eqref{eq:from-lemstar}.
On replacing the box radius by $V+r^{1/3}$,
the second inequality reduces to the claim
\[
 L_1(c,d,\mathcal V)
 \ll_{F,\ve}r^{n/2+1+\ve}d_1^{a_n}d_2^{b_n}V^n
 \qquad(V\ge r^{1/3}).
\]
We henceforth assume $V\ge r^{1/3}$.
If $c\le V$, then $E_c(\mathcal V)\ll(V/c)^n$, and
Lemma~\ref{lem:cube-sum1} gives
\begin{align*}
L_1(c,d,\mathcal V)
&\ll_{F,\ve}c^{n+1+\ve}d^{n-\lceil n/3\rceil+1}V^n\\
&= r^{n/2+1}c^\ve\frac dc\,d^{a_n}V^n
\le r^{n/2+1+\ve}d_1^{a_n}d_2^{b_n}V^n,
\end{align*}
since $d/c\le 1$, $d=d_1d_2$, and $a_n\le b_n$.
We may therefore assume $c>V$.

Two auxiliary results will be useful.
First, for integers $m\ge1$ and real $R\ge1$,
Lemma~\ref{Bernert-Smith-bound} and goodness
(Definition~\ref{defn : good}) give
\begin{equation}\label{eq:good-kernel-average}
\sum_{|\bfh|\le R}N_m(\bfh)^{1/2}
\ll_{F,\ve}m^{n/2}R^\ve \sum_{j=0}^n(R^3/m)^{j/2}
\ll_n R^\ve m^{n/2}(1+R^3/m)^{n/2}.
\end{equation}
Second, 
on recalling the definition \eqref{eq:N_d}
of $N_p(\bu)$, it follows from combining 
the incidence bound \eqref{EQ:clean-slop-incidence-bound}
 with the Lang--Weil estimate, that
\begin{equation}\label{eq:laval_dim}
 \sum_{\bu\in\FF_p^n}N_p(\bu)
 \ll_F p^{n+\lfloor n/3\rfloor-1}.
\end{equation}

In the nonnegative sum \eqref{Eq: Defi.Lc,d,V}
defining $L_1(c,d,\mathcal V)$, we may weaken $c\mid F(\bfa)$
to $d_1\mid F(\bfa)$, and majorise the box $\bfv_0+[-V,V]^n$ by the
smooth weight function $e^n\exp(-\|\bfv-\bfv_0\|_2^2/V^2)$.
For $a\in(\ZZ/c\ZZ)^\ast$, put
\[
 \sigma(\bfh;a)=
 \sum_{\substack{\bfa\bmod c\\d_1\mid F(\bfa)}}
 N_d(\bfa)^{1/2}e_c(a\bfh\cdot\nabla F(\bfa)).
\]
Poisson summation in the class
$\bfv\equiv-a\nabla F(\bfa)\bmod c$ gives
\begin{equation}\label{eq:onion-poisson}
\begin{aligned}
 L_1(c,d,\mathcal V)
 &\ll_n c^{n+1}d^{n/2+1}(V/c)^n
\sum_{a \in (\ZZ/c\ZZ)^*}\sum_{\bfh\in\ZZ^n}
 |\sigma(\bfh;a)|
 \exp(-\pi^2\|\bfh\|_2^2V^2/c^2).
\end{aligned}
\end{equation}
Here the Fourier factor $e_c(\bfh\cdot\bfv_0)$ has absolute
value one, so all estimates remain uniform in $\bfv_0$.

Recall that $c=d_1c_2$, $\gcd(d_1,c_2)=1$, and $d_2^2\mid c_2$.
Write $\bar{c_2}$ for the inverse of $c_2$ modulo $d_1$,
and $\bar{d_1}$ for the inverse of $d_1$ modulo $c_2$.
The Chinese remainder theorem gives
\begin{align*}
 \sigma(\bfh;a)
 =\left(
 \sum_{\substack{\bu\bmod d_1\\d_1\mid F(\bu)}}
 N_{d_1}(\bu)^{1/2}
 e_{d_1}(a\bar{c_2}\bfh\cdot\nabla F(\bu))
 \right)
\left(
 \sum_{\bu\bmod c_2}N_{d_2}(\bu)^{1/2}
 e_{c_2}(a\bar{d_1}\bfh\cdot\nabla F(\bu))
 \right).
\end{align*}
By Lemma~\ref{lem:frog}, with its gcd factors omitted, the first
factor is
\[
 \ll_{F,\ve}d_1^{3n/2-\lceil n/3\rceil-1+\ve}
 =d_1^{n+a_n+\ve}.
\]
For the second factor, as in \cite[Eq.~(5.14)]{41} Cauchy's inequality over $\bu\bmod d_2$,
followed by quadratic differencing, gives
\begin{align*}
\left|
 \sum_{\bu\bmod c_2}N_{d_2}(\bu)^{1/2}
 e_{c_2}(a\bar{d_1}\bfh\cdot\nabla F(\bu))
 \right|^2\le
 \left(\sum_{\bfb\bmod d_2}N_{d_2}(\bfb)\right)
 c_2^nN_{c_2/d_2}(\bfh).
\end{align*}
Since $d_2$ is square-free, combining the Chinese remainder theorem with \eqref{eq:laval_dim} gives
\[
 \sum_{\bfb\bmod d_2}N_{d_2}(\bfb)
 \ll_{F,\ve}d_2^{n+\lfloor n/3\rfloor-1+\ve}.
\]
Combining the two factors, and using $c/d=(d_1c_2)/(d_1d_2)=c_2/d_2$, yields
\begin{equation}\label{eq:onion-sigma-refined}
\begin{aligned}
 |\sigma(\bfh;a)|
 \ll_{F,\ve}
 d_1^{n+a_n+\ve}c_2^{n/2}
 d_2^{n/2+b_n+\ve}N_{c/d}(\bfh)^{1/2}
 =
 (cd)^{n/2}d^\ve d_1^{a_n}d_2^{b_n}
 N_{c/d}(\bfh)^{1/2}.
\end{aligned}
\end{equation}
In \eqref{eq:onion-poisson}, the contribution from
$|\bfh|\gg (c/V)\log(2c)$ is negligible by the trivial
bound for $\sigma(\bfh;a)$.
In particular, substituting \eqref{eq:onion-sigma-refined} into \eqref{eq:onion-poisson} and applying \eqref{eq:good-kernel-average} with $m=c/d$ therefore gives
\begin{align*}
 L_1(c,d,\mathcal V)
 &\ll_{F,\ve}
 c^{n/2+2+\ve}d^{n+1}d_1^{a_n}d_2^{b_n}V^n
 \sum_{|\bfh|\ll(c/V)\log(2c)}N_{c/d}(\bfh)^{1/2}\\
 &\ll_{F,\ve}
 r^{n/2+1+\ve}d_1^{a_n}d_2^{b_n}
 V^n(1+r/V^3)^{n/2}\\
 &\ll_{F,\ve}
 r^{n/2+1+\ve}d_1^{a_n}d_2^{b_n}V^n.
\end{align*}
In the last step we used $V\ge r^{1/3}$.
This proves the result.
\end{proof}

\subsection{Further averaging over cubefull  moduli} 

Let 
$\mathcal V\subset\ZZ^n\cap(\bfv_0+[-V,V]^n)$ be any subset, 
with  $R,V\ge1$.
We proceed to prove the following result. 

\begin{lemma}\label{Le: Cubefull.finalbound}
For every $\ve>0$, uniformly in $\bfv_0\in\RR^n$, we have
\begin{equation}\label{eq:cubefull-average-sharp}
 \sum_{\substack{r\asymp R\\r\ \mathrm{cubefull}}}
 \sum_{\bfv\in\mathcal V}|S_r(\bfv)|
 \ll_{F,\ve}
 R^{n/2+1+\theta_n+\ve}(V+R^{1/3})^n,
\end{equation}
where 
\begin{equation}\label{eq:powerfull-averaging-exponents}
 \theta_n=\max\left\{\frac25,\frac{n-2\lceil n/3\rceil}{6}\right\}.
\end{equation}
In particular, for $n=10$ the power of $R$ in this bound is
$32/5+\ve$.
\end{lemma}

\begin{proof}
For a squarefull integer $r=c^2d$, put
$w_n(r)=d_1^{a_n}d_2^{b_n}$, with $d_1,d_2$ as in \eqref{Defi:d1}, and where $a_n,b_n$ are given by \eqref{eq:poisson-weight-exponents}.
Then $w_n$ is multiplicative on squarefull integers, and
\[
 w_n(p^\nu)=
 \begin{cases}
  1&\nu\ge2\text{ even},\\
  p^{a_n}&\nu=3,\\
  p^{b_n}&\nu\ge5\text{ odd}.
 \end{cases}
\]
For real $s$ sufficiently large, it follows that
\begin{align*}
 \sum_{r\ \mathrm{cubefull}}\frac{w_n(r)}{r^s}
 &=\prod_p\left(
  1+p^{a_n-3s}
   +\frac{p^{-4s}+p^{b_n-5s}}{1-p^{-2s}}
 \right).
\end{align*}
This product converges for
\[
 s>\max\left\{\frac14,\frac{a_n+1}{3},\frac{b_n+1}{5}\right\}.
\]
Now
\[
 \frac{a_n+1}{3}=\frac{n-2\lceil n/3\rceil}{6},\qquad
 \frac{b_n+1}{5}=\frac{\lfloor n/3\rfloor+1}{10}.
\]
For $n\ge 11$, it is then easy to check that
$(a_n+1)/3 \ge \max\{1/2, (b_n+1)/5\}$.
For $n=10$, we have $(a_n+1)/3 = 1/3$ and $(b_n+1)/5 = 2/5$.
Thus for any $n\ge 10$, the convergence threshold is exactly
 $\theta_n$ from
\eqref{eq:powerfull-averaging-exponents}.
Rankin's trick gives, for every $\eta>0$,
\[
 \sum_{\substack{r\asymp R\\r\ \mathrm{cubefull}}}w_n(r)
 \ll_{n,\eta}R^{\theta_n+\eta}.
\]
Finally, 
combining these estimates with Lemma \ref{lem:onion}
completes the proof.
\end{proof}

\section{Lifting Smithwork}
\label{sec:smith-work}

As hinted at in \S\ref{sec:intro},
the Poisson summation based estimate of \S\ref{sec:squarefull-average}
is strong for some collections of moduli $q$ and vectors $\bfv$.
However, for other $q$ and $\bfv$,
it does not fit well with the stratification estimates
that we will provide in \S\ref{sec:MarkV-construction}.
To resolve this issue, it will be useful to
modify the choice of the auxiliary lifting modulus $A$
in the factorisation $q = A^2T$ of Lemma~\ref{lem:q-flexible-lifting}.
Very roughly speaking,
in \S\ref{sec:squarefull-average} we took $A\approx q^{1/2}$ for a squarefull modulus $q$,
whereas in the present section we will
eventually take $A\approx q^{1/3}$ for a cubefull modulus $q$.
The reduction in the size of $A$
will slightly weaken our estimates on average over $\bfv$,
while also making the estimates
more applicable to relatively thin sets of vectors $\bfv$.
The restriction to cubefull moduli will also force us to  treat
arbitrary cube-free moduli separately later, in \S\ref{sec:MarkV-construction}.
The precise manner in which
we combine cube-free and cubefull estimates
is somewhat complicated,
and is done in \S\ref{sec:n10-mixed-moduli}.

The main result of this section is Lemma~\ref{lem:q-smith-cubefull},
concerning cubefull averages for $n=10$.
In fact, Lemma~\ref{lem:q-smith-cubefull}
is restricted to cubefull moduli
whose prime factors are all sufficiently large,
with the complementary case being deferred to \S\ref{sec:fixed-primes}.
However, we begin more generally,
assuming only that $n\ge 10$ and $X(\QQ)=\emptyset$,
and that $q$ is an arbitrary positive integer.
Our results are based on the fact that
the bound for $S_q(\bfv)$ given in
\eqref{eq:q-flexible-pointwise} depends only on $\bfv\bmod{A}$.
Thus \eqref{eq:q-flexible-pointwise} controls $S_q(\bfv)$
uniformly along a given residue class $\bfv\bmod{A}$.

\begin{lemma}
\label{lem:q-flexible-average}
Let $q = A^2T$, where $A,T\ge 1$.
Then
\begin{equation}
\sum_{\bfa\bmod{A}}
\max_{\substack{\bfv\bmod{q} \\ \bfv\equiv \bfa\bmod{A}}} |S_q(\bfv)|
\leq
A^{n+2}T
\sum_{\substack{\bfy\bmod A\\F(\bfy)\equiv0\bmod A}}
\mathcal G_{A,T}(\bfy),
\label{eq:q-flexible-average}
\end{equation}
where $\mathcal G_{A,T}(\bfy)$ is defined in \eqref{eq:GAT-max-sum}.
\end{lemma}

\begin{proof}
Plug \eqref{eq:q-flexible-pointwise} into
the left-hand side of \eqref{eq:q-flexible-average}.
Switching the order of summation of
the variables $\bfa,\bfy\bmod{A}$ gives \eqref{eq:q-flexible-average},
since $\sum_{a\in (\ZZ/A\ZZ)^\ast} A^{n+1} T \le A^{n+2} T$.
\end{proof}

Henceforth suppose that
\begin{equation}
\label{EQ:radTA}
\rad(T)\mid \rad(A).
\end{equation}
For \(a\geq1\),
\(t\geq0\),
and $0\le i<a$, let
\begin{equation}
\lambda_i(a,t) :=
\begin{cases}
\frac12\bm{1}_{i<t}&t\leq a,\\[1mm]
\frac12&t>a \text{ and } i<2a-t,\\
1&t>a \text{ and } i\geq 2a-t.
\end{cases}
\label{eq:q-penalty}
\end{equation}

\begin{lemma}
\label{lem:q-terminal}
Suppose $v_p(A) = a \ge 1$ and
$T = p^t$, where $t\ge 0$.
Let the Smith normal form of \(H(\bfy)\bmod p^a\) be
\[
\diag(p^{j_1},\ldots,p^{j_n}),\qquad
0\leq j_\nu\leq a,
\]
and put
\[
c_i:=\#\{\nu:j_\nu\leq i\}\qquad(0\leq i<a).
\]
If either $t\le a$ or $p\ne 2$, then
\begin{equation}
\mathcal G_{A,T}(\bfy)
\leq
p^{nt-\sum_{i=0}^{a-1}\lambda_i(a,t)c_i}.
\label{eq:q-terminal-profile}
\end{equation}
\end{lemma}

\begin{proof}
By \eqref{eq:GAT-max-sum} and \eqref{z-sum-in-max-G},
it suffices to prove the bound in \eqref{eq:q-terminal-profile}
for $|\mathcal S_{A,T}(\bfd,\alpha,\bfy)|$, uniformly in
$\bfd\in(\ZZ/T\ZZ)^n$ and
$\alpha\in(\ZZ/T\ZZ)^\ast$.
Recall the definition of $N_d(\bfy)$ from \eqref{eq:N_d}.
If $t\le a$, then by Lemma~\ref{LEM:square-and-diff} we have
\begin{equation*}
|\mathcal{S}_{A,T}(\bfd,\alpha,\bfy)|^2
\le T^n N_T(\bfy)
= p^{nt} p^{\sum_{\nu=1}^{n} \min(j_\nu,t)}.
\end{equation*}
Here
\begin{equation*}
\begin{split}
\sum_{\nu=1}^{n} \min(j_\nu,t)
&= \sum_{\nu=1}^{n} \#\{0\le i<t: i<j_\nu\} \\
&= \sum_{i=0}^{t-1} \#\{\nu:j_\nu>i\}
= nt - \sum_{i=0}^{t-1}c_i
= nt - 2\sum_{i=0}^{a-1}\lambda_i(a,t)c_i,
\end{split}
\end{equation*}
by the $t\le a$ case of \eqref{eq:q-penalty}.
Thus \eqref{eq:q-terminal-profile} holds for $t\le a$.

Now suppose $p\ne 2$ and $t>a$.
The cubic term in \eqref{z-sum-in-max-G}
depends only on $\bfz\bmod{p^{t-a}}$.
Writing $\bfz=\bfv+p^{t-a}\bfx$ gives
\begin{equation*}
|\mathcal{S}_{A,T}(\bfd,\alpha,\bfy)|
\le \sum_{\bfv\bmod p^{t-a}}
\left|\sum_{\bfx\bmod p^a}
e_{p^a}((\bfd+\alpha H(\bfy)\bfv)\cdot\bfx+\alpha p^{t-a}Q_{\bfy}(\bfx))\right|.
\end{equation*}
Translating $\bfx$ by $\bfh\in \ker(p^{t-a}H(\bfy)\bmod{p^a})$
and averaging over $\bfh$,
noting that $p^{t-a}Q_{\bfy}(\bfh)
\equiv \bar{2}p^{t-a}H(\bfy)\bfh\cdot \bfh
\equiv 0 \bmod{p^a}$
since $p$ is odd, shows that the $\bfx$-sum vanishes
unless $$\bfd+\alpha H(\bfy)\bfv
\in \ker(p^{t-a}H(\bfy)\bmod{p^a})^\perp
= \im(p^{t-a}H(\bfy)\bmod{p^a}),$$
where we recall that $\ker(M)^\perp = \im(M)$
holds for every symmetric matrix $M\bmod{p^a}$.
After writing $H(\bfy)$ in Smith normal form,
we find that the number of admissible vectors $\bfv\bmod{p^{t-a}}$ is either zero
or $\prod_{\nu=1}^{n} p^{t-a - \min(t-a,a-j_\nu)}
= \prod_{\nu=1}^{n} p^{\max(0,t-2a+j_\nu)}$.
Lemma~\ref{LEM:square-and-diff} bounds
the remaining sum over $\bfx\bmod{p^a}$
for each admissible $\bfv$, giving in total
\begin{equation*}
|\mathcal{S}_{A,T}(\bfd,\alpha,\bfy)|
\le p^{\sum_{\nu=1}^{n} \max(0,t-2a+j_\nu)}
(p^{na} N_{p^a}(p^{t-a}\bfy))^{1/2}.
\end{equation*}
Put $h=\max\{2a-t,0\}$, so that $0\leq h<a$. By the
definition of $c_i$, we have
\[
\sum_{\nu=1}^n\max(0,t-2a+j_\nu)
=n(t-a)-\sum_{i=h}^{a-1}c_i,
\]
and
\[
N_{p^a}(p^{t-a}\bfy)
=p^{\sum_{\nu=1}^n\min(t-a+j_\nu,a)}
=p^{na-\sum_{i=0}^{h-1}c_i}.
\]
Substituting into the preceding bound gives
\[
|\mathcal S_{A,T}(\bfd,\alpha,\bfy)|
\leq
p^{nt-\frac12\sum_{i=0}^{h-1}c_i-\sum_{i=h}^{a-1}c_i}.
\]
For $t>a$, the weights $\lambda_i(a,t)$ are $1/2$ for
$i<h$ and $1$ for $i\geq h$. Maximising over $\bfd$ and
$\alpha$ therefore gives the required bound.
\end{proof}

For \(0\leq\varrho\leq n\), let $\delta_n(\varrho)$ be as in \eqref{eq:q-delta}
and let $\gamma_n(\varrho)$ be as in \eqref{eq:q-gamma}.

\begin{lemma}
\label{lem:q-easy-hessian-lift-bound}
For $a\ge 1$ and \(0\leq \varrho\leq n\),
\[
\#\{\bfx\bmod p^a:
F(\bfx)\equiv0\bmod{p^a},\;
\rank H(\bm x\bmod p)=\varrho\}
\ll_{\varepsilon,F}
p^{(a-1)(n-1+\ve)+\delta_n(\varrho)+\gamma_n(\varrho)\bm{1}_{a\ge 2}}.
\]
\end{lemma}

\begin{proof}
If $a=1$, this is trivial by \eqref{eq:q-first-rank}.
If $a\ge 2$, we apply Lemma~\ref{lem:c/d-bound} with \(c=p^a\) and \(d=p\).
Lemma~\ref{LEM:weighted-rank-locus} then completes the proof.
\end{proof}

For $j\ge 0$, let
\begin{equation}
\tau_n(j) :=
\begin{cases}
0&j=0,\\
\min\{n-1,\lfloor n/3\rfloor+j-1\}&1\leq j\leq n-1,\\
n&j\geq n.
\end{cases}
\label{eq:q-tau}
\end{equation}

\begin{lemma}
\label{q-translated-rank}
For every prime \(p\)
and matrix \(\mathcal T\in\Mat_{n\times n}(\FF_p)\),
\begin{equation*}
\#\{\bfy\in\FF_p^n:
\rank(\mathcal T+H(\bfy))\leq j\}
\ll_F p^{\tau_n(j)}.
\end{equation*}
\end{lemma}

\begin{proof}
The case $j\ge n$ is trivial, since the rank condition is
automatic and $\tau_n(j)=n$. Hence suppose $0\le j<n$.
Let $p$ be large in terms of $F$,
so that Corollary~\ref{cor:hesrank} spreads out to the prime $p$.
Consider the projective variety
$W_{\mathcal T,j}
:= \{[x_0:\bfy]\in\PP^n:
\rank(x_0\mathcal T+H(\bfy))\leq j\}$,
and 
identify $\A^n$ with the chart $x_0\ne0$ in $\PP^n$
via $\bfy\mapsto[1:\bfy]$. Then
$W_{\mathcal T,j}\cap\A^n
=\{\bfy\in\A^n:\rank(\mathcal T+H(\bfy))\leq j\}$, and we 
have  $$\dim(W_{\mathcal T,j} \cap \A^n)
\le \dim(W_{\mathcal T,j})
\le 1+\dim(W_{\mathcal T,j} \cap \{x_0=0\})
= \dim(V_j),$$
where $V_j$ is defined 
in \eqref{EQ:Hessian-Vr-locus}.
But $\dim(V_j) \le \tau_n(j)$, by Corollary~\ref{cor:hesrank}.
The Lang--Weil estimate completes the proof, uniformly in $\mathcal T$ and $p$,
since the locus $\rank(\mathcal T+H(\bfy))\leq j$ is cut out by
minors of uniformly bounded number and degree.
The finitely many exceptional primes are absorbed into the implied constant.
\end{proof}

For \(a\geq1\) and any sequence $\bfc=(c_0,\ldots,c_{a-1})\in \ZZ^a$
with $0\leq c_0\leq\cdots\leq c_{a-1}\leq n$,
set
\begin{equation}
\label{eq:q-profile-cost}
\mathcal C_{n,a}(\bfc)
:= \min\left\{
(a-1)(n-1)+\gamma_n(c_0)\bm{1}_{a\ge 2},\;
\sum_{i=1}^{a-1}\tau_n(c_{i-1}+c_i)
\right\}.
\end{equation}

\begin{lemma}\label{lem:q-profile}
Let $a\ge1$ and $\bfc=(c_0,\ldots,c_{a-1})\in\ZZ^a$
with $0\le c_0\le\cdots\le c_{a-1}\le n$.
Let $0\le j_1\le\cdots\le j_n\le a$ be the unique integers
satisfying
\[
\#\{\nu:j_\nu\le i\}=c_i
\qquad(0\le i<a).
\]
For a prime $p$, let $\mathcal P_p(\bfc)$ be the set of
$\bfx\bmod p^a$ such that $F(\bfx)\equiv0\bmod p^a$
and $H(\bfx)\bmod p^a$ has Smith normal form
$\diag(p^{j_1},\ldots,p^{j_n})$.
Then, for every $\ve>0$ and all sufficiently large primes $p$
in terms of $F$ and $\ve$, we have
\[
\#\mathcal P_p(\bfc)
\ll_{\ve,F}
p^{\delta_n(c_0)+\mathcal C_{n,a}(\bfc)+\ve a}.
\]
\end{lemma}

\begin{proof}
Every $\bfx\in\mathcal P_p(\bfc)$ satisfies
$\rank H(\bfx\bmod p)=c_0$.
Discarding the remaining Smith conditions and applying
Lemma~\ref{lem:q-easy-hessian-lift-bound} with $\varrho=c_0$
gives the bound corresponding to the first entry in
\eqref{eq:q-profile-cost}.
For the second entry, we will prove by induction on $a\ge 1$ that
\begin{equation*}
\#\mathcal P_p(\bfc)
\le e^{O_F(a)}
p^{\delta_n(c_0) + \sum_{i=1}^{a-1}\tau_n(c_{i-1}+c_i)}
\end{equation*}
for all $p$.
This will then imply Lemma~\ref{lem:q-profile}
by taking $p^\ve \ge e^{O_F(1)}$.

The base case $a=1$ is simply \eqref{eq:q-first-rank}.
Now suppose $a\ge 2$.
Let $\bfy\bmod{p^{a-1}}$ be a vector in $\mathcal P_p(c_0,\dots,c_{a-2})$.
Fix a lift $\bfy\bmod{p^a}$,
invertible matrices $A,B\bmod{p^a}$,
and a matrix $M\bmod{p}$,
for which $$A(H(\bfy)-p^{a-1}M)B
\equiv \diag(p^{j_1},\dots,p^{j_{c_{a-2}}},0,\dots,0)
\bmod{p^a}.$$
This is possible by the Smith normal form of $H(\bfy)\bmod{p^{a-1}}$.
Our task is now to give an upper bound for the number of vectors
$\bfx \equiv \bfy+p^{a-1}\bfz \bmod{p^a}$, with $\bfz\bmod{p}$,
for which the matrix
\begin{equation}
\label{matrix-pivot-hensel}
AH(\bfx)B \equiv \diag(p^{j_1},\dots,p^{j_{c_{a-2}}},0,\dots,0)
+ p^{a-1}A(M+H(\bfz))B \bmod{p^a}
\end{equation}
has Smith normal form $\diag(p^{j_1},\ldots,p^{j_{c_{a-1}}},0,\dots,0) \bmod{p^a}$.
By elimination in the first $c_{a-2}$ rows and columns,
using the fact that $p^{a-1-j_i} p^{a-1} \equiv 0\bmod{p^a}$
for all $1\le i\le c_{a-2}$, we find that
the right-hand side of \eqref{matrix-pivot-hensel}
has the same Smith normal form as
the matrix $$\diag(p^{j_1},\dots,p^{j_{c_{a-2}}},p^{a-1}\pi_{n-c_{a-2}}),$$
where $\pi_{n-c_{a-2}}$ is the bottom-right
$(n-c_{a-2}) \times (n-c_{a-2})$ block of $A(M+H(\bfz))B$.
Thus $$\rank(\pi_{n-c_{a-2}}\bmod{p}) = c_{a-1}-c_{a-2}.$$
Adding a row or a column to a matrix
over $\FF_p$ increases its rank by at most one, so
\begin{equation*}
\begin{split}
\rank(M+H(\bfz)\bmod{p})
&= \rank(A(M+H(\bfz))B\bmod{p}) \\
&\le 2c_{a-2} + \rank(\pi_{n-c_{a-2}}\bmod{p})
= c_{a-1}+c_{a-2}.
\end{split}
\end{equation*}
Lemma~\ref{q-translated-rank} leaves
\(O_F(p^{\tau_n(c_{a-1}+c_{a-2})})\) possibilities for \(\bfz\).
This completes the induction.
\end{proof}

For any sequence
\(\boldsymbol\lambda=(\lambda_0,\ldots,\lambda_{a-1})\in \RR_{\ge 0}^a\), put
\begin{equation}
\label{eq:q-master-functional}
\mathfrak M_{n,a}(\boldsymbol\lambda)
:= \max_{\substack{\bfc=(c_0,\ldots,c_{a-1})\in \ZZ^a \\
0\leq c_0\leq\cdots\leq c_{a-1}\leq n}}
\left\{\delta_n(c_0) + \mathcal C_{n,a}(\bfc)
-\sum_{i=0}^{a-1}\lambda_i c_i\right\}.
\end{equation}
For \(a\geq1\) and \(t\geq0\),
recall $\lambda_i(a,t)$ from \eqref{eq:q-penalty} and let
\begin{equation}
\label{eq:q-local-D}
\Phi_n(a,t)
:= \mathfrak M_{n,a}\bigl((\lambda_i(a,t))_{0\leq i<a}\bigr),
\qquad
D_n(a,t) := (n+2)a+(n+1)t+\Phi_n(a,t).
\end{equation}

\begin{prop}
\label{prop:q-smith-bound}
Let $q\ge 1$ be squarefull.
Suppose every prime $p\mid q$
is large enough in terms of $F$ and $\ve$.
Choose \(1\leq a_p\leq\lfloor v_p(q)/2\rfloor\)
and let $t_p := v_p(q)-2a_p$.
If \(A := \prod_{p\mid q} p^{a_p}\)
and $T := \prod_{p\mid q} p^{t_p}$,
then the condition \eqref{EQ:radTA} holds and
\begin{equation}
\sum_{\bfa\bmod{A}}
\max_{\substack{\bfv\bmod{q} \\ \bfv\equiv \bfa\bmod{A}}} |S_q(\bfv)|
\ll_{\ve,F}
q^\ve \prod_{p\mid q}p^{D_n(a_p,t_p)}.
\label{eq:q-global-smith-bound}
\end{equation}
\end{prop}

\begin{proof}
Plugging Lemma~\ref{lem:q-terminal}
into Lemma~\ref{lem:q-flexible-average}
using the $T$-multiplicativity \eqref{EQ:max-G-multiplicative}
gives
\[
\sum_{\bfa\bmod{A}}
\max_{\substack{\bfv\bmod{q} \\ \bfv\equiv \bfa\bmod{A}}} |S_q(\bfv)|
\leq
A^{n+2}T
\prod_{p\mid q} \sum_{\substack{\bfy\bmod p^{a_p}\\F(\bfy)\equiv0\bmod p^{a_p}}}
p^{nt_p-\sum_{i=0}^{a_p-1}\lambda_i(a_p,t_p)c_i},
\]
where $\bfc \in \ZZ^{a_p}$ is such that $\bfy\in \mathcal P_p(\bfc)$
in the notation of Lemma~\ref{lem:q-profile}.
The number of possible vectors $\bfc$ is $O((a_p+1)^n)$,
since each $\bfc$ corresponds to a unique Smith normal form modulo $p^{a_p}$.
Therefore, by Lemma~\ref{lem:q-profile} and the divisor bound on $q$,
\[
\sum_{\bfa\bmod{A}}
\max_{\substack{\bfv\bmod{q} \\ \bfv\equiv \bfa\bmod{A}}} |S_q(\bfv)|
\ll_\ve A^{n+2}T q^\ve T^n
\prod_{p\mid q} p^{\Phi_n(a_p,t_p)}.
\]
The definition of $D_n(a,t)$ in \eqref{eq:q-local-D} accounts for
the powers of $A$ and $T$.
\end{proof}

\begin{lemma}
\label{lem:q-smith-cubefull}
Assume $n=10$.
Let \(r\ge 1\) be cubefull.
Suppose every prime $p\mid r$
is large enough in terms of $F$ and $\ve$.
Let
\begin{equation}
\label{eq:q-buw}
b := \prod_{v_p(r)>0} p^{\lfloor v_p(r)/3 \rfloor},
\qquad z_i := \prod_{\substack{v_p(r)>0 \\ v_p(r) \equiv i\bmod{3}}} p,
\end{equation}
for $0\le i\le 2$,
so that $r=b^3z_1z_2^2$ and $z_0z_1z_2\mid b$.
Define
\[
A := bz_2,\quad T := bz_1.
\]
Then \(r=A^2T\), the condition \eqref{EQ:radTA} holds, and
\begin{equation}
\sum_{\bfa\bmod{A}}
\max_{\substack{\bfv\bmod{r} \\ \bfv\equiv \bfa\bmod{A}}} |S_r(\bfv)|
\ll_{F,\ve}
r^{10+\ve}z_0^{-2}z_1^{-4}z_2^{-1}.
\label{eq:q-route-zero}
\end{equation}
\end{lemma}

\begin{proof}
It suffices to prove
\[
\Phi_{10}(e,e)\le 7e-2,\qquad
\Phi_{10}(e,e+1)\le 7e-5,\qquad
\Phi_{10}(e+1,e)\le 7e+7
\]
for every integer $e\ge1$.
Indeed, if $p\mid r$
and $v_p(r)=3e+y$ with $0\le y\le 2$,
then the pair
$(v_p(A),v_p(T))$ is respectively
$(e,e)$, $(e,e+1)$ or $(e+1,e)$ according as
$y=0,1$ or $2$.
Writing $a=v_p(A)$ and $t=v_p(T)$,
it then follows from the $n=10$ case of \eqref{eq:q-local-D} that
\begin{equation*}
D_{10}(a,t) \le
\begin{cases}
30e-2 = v_p(r^{10}z_0^{-2})&y=0,\\
30e+6 = v_p(r^{10}z_1^{-4})&y=1,\\
30e+19 = v_p(r^{10}z_2^{-1})&y=2.
\end{cases}
\end{equation*}
Together with Proposition~\ref{prop:q-smith-bound},
this implies \eqref{eq:q-route-zero}.

Write $\delta=\delta_{10}$, $\gamma=\gamma_{10}$ and
$\tau=\tau_{10}$,
for the functions defined in \eqref{eq:q-delta}, \eqref{eq:q-gamma} and
\eqref{eq:q-tau}, respectively.
Their values at integers $0\le x\le 10$ are given in the following table:
\[
\begin{array}{c|*{11}{c}}
x         & 0 & 1 & 2 & 3 & 4 & 5 & 6 & 7 & 8 & 9 & 10 \\ \hline
\delta(x) & 0 & 0 & 2 & 4 & 6 & 7 & 8 & 8 & 9 & 9 & 9  \\
\gamma(x) & 2 & 1 & 1 & 1 & 0 & 0 & 0 & 0 & 0 & 0 & 0  \\
\tau(x)   & 0 & 3 & 4 & 5 & 6 & 7 & 8 & 9 & 9 & 9 & 10.
\end{array}
\]
Moreover, $\tau(x)=10$ for every integer $x\ge 10$.

For an integer sequence $0\le c_0\le\cdots\le c_{a-1}\le10$, put
\[
x=c_0,\qquad
P=\sum_{i=0}^{a-1}\lambda_i(a,t)c_i
= \frac12\sum_{i=0}^{e-1}c_i
+ \frac12c_{e-1}\, \bm{1}_{y=1},
\]
by \eqref{eq:q-penalty},
where the extra term $\frac12c_{e-1}$ appears only when $(a,t)=(e,e+1)$.
Define
\[
U=\delta(x)+9(a-1)+\gamma(x)-P,\qquad
V=\delta(x)+\sum_{i=1}^{a-1}\tau(c_{i-1}+c_i)-P.
\]
The definitions \eqref{eq:q-local-D},
\eqref{eq:q-master-functional}
and \eqref{eq:q-profile-cost}
give
\[
\Phi_{10}(a,t)\le\max_{\bfc}\min\{U,V\}.
\]
Here we have replaced
$\gamma(x)\bm{1}_{a\ge2}$ in \eqref{eq:q-profile-cost}
by the larger quantity
$\gamma(x)$.
Indeed, the argument below will handle all cases $a\ge 1$ uniformly.

First suppose $x\ge4$. Then $\gamma(x)=0$, and
from the table of values of $\delta$ we find that
\[
\delta(x)-x/2\le5,\qquad
\delta(x)-x\le2,\qquad
9-x/2\le7.
\]
Since every $c_i\ge x$, we have $P \ge (e+\bm{1}_{y=1})x/2$
and these inequalities imply
\[
\begin{array}{c|c|c}
(a,t)&\text{upper bound for }U&\text{linear combination used}\\ \hline
(e,e)&5+7(e-1)=7e-2&(\delta(x)-x/2) + (e-1)(9-x/2)\\
(e,e+1)&2+7(e-1)=7e-5&(\delta(x)-x) + (e-1)(9-x/2)\\
(e+1,e)&14+7(e-1)=7e+7&(\delta(x)-x/2) + 9 + (e-1)(9-x/2),
\end{array}
\]
where the third column records the decomposition of $U$
used to obtain the second column.

Now suppose $x\le3$. Put
\[
R=\delta(x)-x/2+\gamma(x)/2.
\]
Checking the values of $\delta$ and $\gamma$ gives
$R\le3$ and $R-x/2\le2$.
Moreover, $\tau$ satisfies
\[
\tau(c_{i-1}+c_i)\le \tau(2c_i) \le \min\{c_i+5, 10\},
\]
because $c_{i-1}\le c_i$ and $c_i\ge 0$.
From our formulas for $U,V,P$ we obtain
\begin{equation*}
(U+V)/2
= R + 9(a-1)/2 + \sum_{i=1}^{a-1} \tau(c_{i-1}+c_i)/2
- \sum_{i=1}^{e-1} c_i/2 - \bm{1}_{y=1} c_{e-1}/2.
\end{equation*}
Now plug in $\tau(c_{i-1}+c_i)\le c_i+5$ for $i\le e-1$,
and $\tau(c_{i-1}+c_i)\le 5+5$ for $i=e$ if $a=e+1$.
Since $9(a-1)/2 + 5(a-1)/2 = 7(a-1)$
and $\bm{1}_{y=1} c_{e-1}/2 \ge \bm{1}_{y=1} x/2$,
we thus obtain
\[
\begin{array}{c|c}
(a,t)&\text{upper bound for }(U+V)/2\\ \hline
(e,e)&R+7(e-1)\le7e-4\\
(e,e+1)&R-x/2+7(e-1)\le7e-5\\
(e+1,e)&R+7e+\frac52\le7e+\frac{11}2 \le 7e+7.
\end{array}
\]

Since $\min\{U,V\}\le U$ and
$\min\{U,V\}\le(U+V)/2$, the two cases establish
all three required bounds for $\Phi_{10}$.
\end{proof}

\addtocontents{toc}{\protect\columnbreak}
\section{Localisation at small primes}
\label{sec:fixed-primes}

In the last few results of \S\ref{sec:smith-work},
we restricted attention to moduli $q\ge 1$ supported on sufficiently large primes $p$,
in order to accommodate the factor of $e^{O_F(\Omega(q))}$
that arose from the Smith-based Hensel lifting analysis
in the proof of Lemma~\ref{lem:q-profile}.
Here, $\Omega(q)$ denotes the total number of prime factors of $q$ (with multiplicity).
The factor $e^{O_F(\Omega(q))}$ is perhaps unexpected,
and certainly less commonly seen in Diophantine contexts compared to
$e^{O_F(\omega(q))}$,
where $\omega(q)$ denotes the number of distinct prime factors of $q$
(without multiplicity).
Nonetheless, the proof we have given naturally leads to the factor $e^{O_F(\Omega(q))}$ and so we are forced to confront small primes $p$ separately,
or else the factor $e^{O_F(\Omega(q))}$ would contribute a fixed power of $q$
in our estimates
whenever $q$ is a large power of a fixed prime.

The goal of the present section is precisely to provide satisfactory
estimates at any fixed prime $p=p_0$.
It is easiest to do this after first localising the counting function in the circle method
to convenient $p$-adic congruence classes and real boxes,
as we record below in the formula \eqref{eq:inel-restricted-count}.
In the case of a real box this is done by choosing a suitable real weight function,
a standard procedure that we will recall in \S\ref{sec:intro-circle}.
In the $p$-adic case, the analogue of choosing a weight function is
simply to restrict to a union of congruence classes.
Throughout this section, we assume $n\ge 10$ and $X(\QQ) = \emptyset$.
Then $X$ is non-conical by Proposition~\ref{Bro-integral-cone}.

\begin{lemma}\label{lem:inel-high-rank-point}
Let $p$ be a prime.
Then there is a primitive vector $\boldsymbol{\xi}_p\in\Z_p^n$ such that
\[
 F(\boldsymbol{\xi}_p)=0,\qquad
 \nabla F(\boldsymbol{\xi}_p)\neq\bm 0,\qquad
 \rank_{\QQ_p}H(\boldsymbol{\xi}_p) = r(F) \ge 2 \lceil n/3 \rceil.
\]
\end{lemma}

\begin{proof}
It is known that every non-conical cubic form in $\ge 10$ variables has a smooth $p$-adic zero for all $p\le \infty$; see for instance Theorem~2 of
\cite{Ple71}.
By definition of $r(F)$, the locus on $X_{\QQ_p}$ where
$\rank H = r(F)$ is a nonempty Zariski-open subset.  Since $X$ is
geometrically integral, this open subset meets the Zariski-dense
set $X_{\mathrm{sm}}(\QQ_p)$.
Finally, $r(F)\ge 2\lceil n/3\rceil$ by Lemma~\ref{LEM:my-guy-slop}.
\end{proof}

Since $H(\boldsymbol{\xi}_p)$ is symmetric of rank $r(F)$ over the
characteristic-zero field $\QQ_p$,
we may assume after permuting coordinates that
its leading $r(F)\times r(F)$ principal
minor is nonsingular.  Write
\[
\Delta_p(\bfx)
:=\det\left(\frac{\partial^2F}
{\partial x_i\partial x_j}(\bfx)\right)_{1\leq i,j\leq r(F)}.
\]
Choose $M_p\geq1$ sufficiently large that, on the ball
\begin{equation}\label{eq:inel-ball}
B_p:=\boldsymbol{\xi}_p+p^{M_p}\Z_p^n,
\end{equation}
all points are primitive, one fixed component of $\nabla F$ is nonzero, and
\begin{equation}\label{eq:inel-minor-valuation}
v_p(\Delta_p(\bfx))=v_p(\Delta_p(\boldsymbol{\xi}_p))
\qquad(\bfx\in B_p).
\end{equation}

To preserve the multiplicativity of $S_q(\bfv)$ later on
in Lemma~\ref{lem:restricted-multiplicative},
consider the unit orbit
\begin{equation}\label{eq:inel-Omega}
\Omega_p := \Z_p^\ast B_p
=\bigcup_{\lambda\in\Z_p^\ast}\lambda B_p.
\end{equation}
By homogeneity, the smoothness and Hessian-rank properties on $B_p$ extend to $\Omega_p$.

For $k\geq1$, let $\Omega_p[k]$ denote the image of $\Omega_p$ in
$(\Z/p^k\Z)^n$.
For $\bfv\in\Z^n$, put
\begin{equation}\label{eq:inel-localised-sum}
S_{p^k}^{\Omega_p}(\bfv)
:= \sum_{\alpha\in(\Z/p^k\Z)^\ast}
\sum_{\bfx\in\Omega_p[k]}
e_{p^k}\bigl(\alpha F(\bfx)+\bfv\cdot\bfx\bigr).
\end{equation}
For $\lfloor k/2\rfloor\geq M_p$, the proof of Lemma~\ref{lem:starT}, with $c = p^{\lfloor k/2 \rfloor}$, shows that
\begin{equation*}
S_{p^k}^{\Omega_p}(\bfv)
\ll_p c^{n+1} \sum_{a\in (\ZZ/c\ZZ)^\ast}
\#\{\bfx\bmod{c}: \bfx\in \Omega_p,\; a\nabla F(\bfx) \equiv -\bfv \bmod{c}\}.
\end{equation*}
By the $p$-adic implicit function theorem,
there are only $O_p(1)$ choices for $(x_1,\dots,x_\rho)\bmod{c}$
once the residues $a,x_{\rho+1},\dots,x_n\bmod{c}$ are fixed,
where $\rho=r(F)$.
Thus 
\begin{equation}
\label{fixed-p-adic-stationary-phase}
S_{p^k}^{\Omega_p}(\bfv)
\ll_p c^{n+1} c^{1+n-r(F)}
\ll_p p^{k(n+1-r(F)/2)}.
\end{equation}
The finitely many remaining values of $k$ are absorbed into the implied constant
in \eqref{fixed-p-adic-stationary-phase}.

We proceed to sketch how to restrict the circle method to appropriate residue classes
at a finite set of primes.
Let $\mathcal P_{\mathrm{small}}$ be a finite set of primes, choose $\Omega_p$ as in \eqref{eq:inel-Omega}, and put
\begin{equation}
\label{large-fixed-W}
W := \prod_{p\in\mathcal P_{\mathrm{small}}}p^{M_p}.
\end{equation}
The Chinese remainder theorem defines a subset
$\Omega[H] = \prod_{p\mid H} \Omega_p[v_p(H)] \subset(\Z/H\Z)^n$
for any divisor $H\mid W$.
In \S\ref{sec:intro-circle} we will use the circle method to study the counting function
\begin{equation}\label{eq:inel-restricted-count}
N(P)
:= \sum_{\substack{\bfx\in\Z^n\\F(\bfx)=0\\
\bfx\bmod W\in\Omega[W]}}
w(\bfx/P),
\end{equation}
for a suitable fixed compactly supported weight function $w\in C^\infty_c(\RR^n)$.
Proving a lower bound for
\eqref{eq:inel-restricted-count} is enough for the existence of an
integral zero.
To this end, observe that the restricted local density at each $p\mid W$ is positive.
Since every $\Omega_p$ contains a neighbourhood of the smooth zero
$\boldsymbol{\xi}_p$, the $p$-adic implicit function theorem gives
\begin{equation}
\label{eq:localised-pos-density}
\liminf_{j\to\infty}p^{-j(n-1)}
\#\left\{\bfx\bmod p^j:
\bfx\in\Omega_p[j],\;F(\bfx)\equiv0\bmod{p^j}\right\}>0.
\end{equation}

For any integer $H\mid W$, let
\begin{equation}\label{eq:inel-global-complete-sum}
S_{q,H}(\bfv)
:=\sum_{a\in(\Z/q\Z)^\ast}
\sum_{\substack{\bfx\bmod\lcm(q,H)\\
\bfx\bmod H\in\Omega[H]}}
e_q\bigl(aF(\bfx)\bigr)e_{\lcm(q,H)}(\bfv\cdot\bfx).
\end{equation}
The following lemma records the necessary multiplicativity properties of $S_{q,W}(\bfv)$.
\begin{lemma}
\label{lem:restricted-multiplicative}
Let $q_0,q_1\ge 1$.
If $\gcd(Wq_0,q_1)=1$, then $$S_{q_0q_1,W}(\bfv)
= S_{q_0,W}(\bfv) S_{q_1}(\bfv).$$
Now suppose $p\mid q_0 \Rightarrow p\mid W$.
Then $$S_{q_0,W}(\bfv) = \prod_{p\mid W} S_{p^{v_p(q_0)},p^{M_p}}(\bfv).$$
Moreover, for all $p\mid W$ with $v_p(q_0) \ge M_p$,
we have $$S_{p^{v_p(q_0)},p^{M_p}}(\bfv) = S_{p^{v_p(q_0)}}^{\Omega_p}(\bfv).$$
Finally, for all $p\mid W$ with $v_p(q_0) < M_p$,
we have $$S_{p^{v_p(q_0)},p^{M_p}}(\bfv) \ll_p 1.$$
\end{lemma}

\begin{proof}
The first two displays follow from the Chinese remainder theorem,
as in the proof of \eqref{eq:q-exact-CRT},
together with the scalar $\ZZ_p^\ast$-invariance of the set $\Omega_p$
defined in \eqref{eq:inel-Omega}.
The third display follows directly on comparing
the definitions \eqref{eq:inel-localised-sum} and \eqref{eq:inel-global-complete-sum}.
The fourth display follows from the trivial bound on $S_{p^{v_p(q_0)},p^{M_p}}(\bfv)$,
because $p,v_p(q_0),M_p\ll_p 1$.
\end{proof}

Define the localised singular series
\begin{equation}\label{eq:localised-singular-series}
\mathfrak S_\Omega := \sum_{q\ge1} \lcm(q,W)^{-n} S_{q,W}(\bm0).
\end{equation}

\begin{lemma}
\label{LEM:localised-singular-series-convergence}
For $n\geq 10$, the series $\mathfrak S_\Omega$ converges absolutely to a positive real number.
\end{lemma}

\begin{proof}
In view of Theorem~\ref{Bernert-n10-sing-series}
and Lemma~\ref{lem:restricted-multiplicative},
it suffices to show that for each $p\mid W$
the Euler factor $\sum_{k\ge 0} \lcm(p^k,p^{M_p})^{-n} S_{p^k,p^{M_p}}(\bm0)$
converges absolutely to a positive real number.
The absolute convergence itself is clear from Lemma~\ref{lem:restricted-multiplicative}
and \eqref{fixed-p-adic-stationary-phase}, since $r(F) \ge 2n/3 \ge 3$.
It remains to exhibit positivity.
First observe that if $j\ge M_p$ then
\begin{equation*}
\#\left\{\bfx\bmod p^j:
\bfx\in\Omega_p[j],\;F(\bfx)\equiv0\bmod{p^j}\right\}
= \sum_{\substack{\bfx\bmod{p^j} \\ \bfx\bmod p^{M_p}\in\Omega_p[M_p]}}
p^{-j} \sum_{1\le a\le p^j} e_{p^j}(aF(\bfx)).
\end{equation*}
Let $0\le k\le j$.
The total contribution to the right-hand side
from $a$ with $v_p(a) = j-k$ is
\begin{equation*}
\sum_{\substack{\bfx\bmod{p^j} \\ \bfx\bmod p^{M_p}\in\Omega_p[M_p]}}
p^{-j} \sum_{\substack{1\le b\le p^k \\ \gcd(b,p^k)=1}} e_{p^k}(bF(\bfx))
= p^{-j} p^{n(j-\max(k,M_p))} S_{p^k,p^{M_p}}(\bm0),
\end{equation*}
by counting lifts modulo $p^j$ of residues modulo
$p^{\max(k,M_p)}$.
It follows that
\begin{equation*}
\sum_{0\le k\le j} \lcm(p^k,p^{M_p})^{-n} S_{p^k,p^{M_p}}(\bm0)
= p^{-j(n-1)}
\#\left\{\bfx\bmod p^j:
\bfx\in\Omega_p[j],\;F(\bfx)\equiv0\bmod{p^j}\right\}.
\end{equation*}
Taking $j\to \infty$ completes the proof,
in view of the lower bound \eqref{eq:localised-pos-density}.
\end{proof}

The following variant of Proposition~\ref{prop:q-smith-bound} will also be useful
in the circle method.
\begin{lemma}
\label{lem:q-fixed-shifted box}
Fix a prime $p$, and let
\(\Omega_p\) be the neighbourhood chosen in \eqref{eq:inel-Omega}.
Let $k \ge 0$ and $a := \lceil k/3 \rceil$.
Then
\begin{equation}
\sum_{\bfa\bmod{p^a}}
\max_{\substack{\bfv\bmod{p^{\max(k,M_p)}} \\ \bfv\equiv \bfa\bmod{p^a}}}
|S_{p^k,p^{M_p}}(\bfv)|
\ll_{F,p}
p^{k(n + (4-r(F))/6)}.
\label{eq:q-fixed-shifted box}
\end{equation}
\end{lemma}

\begin{proof}
The finitely many cases $k<3M_p$ are absorbed into the implied constant.
Thus suppose $k\ge 3M_p$, and let $t := k-2a$.
On \(\Omega_p\), a fixed
\(r(F)\times r(F)\) Hessian minor has bounded valuation.
Therefore, if $\bfy\in \Omega_p$, then $H(\bfy)$ has at least $r(F)$ Smith invariant factors $p^j$ with $j\ll_p 1$.
Since \(0\leq t\leq a\), it follows from Lemma~\ref{lem:q-terminal}
that $\mathcal G_{p^a,p^t}(\bfy) \ll_{F,p} p^{nt-r(F) t/2}$.
Moreover, the \(p\)-adic implicit function theorem gives
\[
\#\{\bfy\bmod p^a:
\bfy\in\Omega_p,\ F(\bfy)\equiv0\bmod p^a\}
\ll_{F,p} p^{(n-1)a}.
\]
The proof of Lemma~\ref{lem:q-flexible-average}, with the restriction
to \(\Omega_p\) retained, now gives
\[
\sum_{\bfa\bmod{p^a}}
\max_{\substack{\bfv\bmod{p^k} \\ \bfv\equiv \bfa\bmod{p^a}}}
|S_{p^k}^{\Omega_p}(\bfv)|
\ll_{F,p} p^{(n+2)a} p^t
p^{(n-1)a}  p^{nt-r(F) t/2}.
\]
Here $p^a,p^t \asymp_p p^{k/3}$, so \eqref{eq:q-fixed-shifted box} follows.
\end{proof}

\section{Second moments over linear spaces}
\label{sec:n10-second-moments}

Throughout this section suppose that $n\ge10$ and $X(\QQ)=\emptyset$.
In particular, $F$ is absolutely irreducible
by Proposition~\ref{Bro-integral-cone}.
For a prime power $q=p^a$,
a nontrivial additive character $\psi_q$ of $\FF_q$,
and a vector $\bfv\in \FF_q^n$,
put
\begin{equation}
T_q(\bfv)=T(\psi_q,F;\bfv):=\sum_{\substack{\bfx\in\FF_q^n\\F(\bfx)=0}}
\psi_q(\bfx\cdot\bfv)
-q^{n-1}\bm1_{\bfv=\bm0}.
\label{n10:T}
\end{equation}
When $q=p$ and $\psi_p=e_p$, orthogonality in the outer sum of \eqref{eq:defSq} gives
\begin{equation}\label{eq:pST}
S_p(\bfv)=pT_p(\bfv),    
\end{equation}
even when $\bfv=\bm0$.
We will be interested in the second moment of $T_q(\bfv)$ over certain linear spaces.

Given a linear space $L$ over $\QQ$,
we may define $L(\FF_{p^a})$ by choosing an integral model of $L$.
For any two fixed choices of integral models of $L$,
the sets $L(\FF_{p^a})$ then coincide for all sufficiently large primes $p$.
For brevity, we thus omit any further discussion of integral models of $L$.

\begin{lemma}
\label{n10:second-moment}
Let $L\subset\QQ^n$ be a linear space of dimension $n-2$.
For every sufficiently large prime $p$, every $q=p^a$, and every
nontrivial $\psi_q$, one has
\begin{equation}
 \sum_{\bfv\in L(\FF_q)}|T_q(\bfv)|^2\ll_{F,L}q^{2n-3}.
 \label{n10:moment-two}
\end{equation}
\end{lemma}
\begin{proof}
Let $X_1,\dots,X_n$ and $V_1,\dots,V_n$ be the generic $\bfx$ and $\bfv$ variables, respectively.
After a change of variables defined over $\QQ$, with the corresponding dual change in the $\bfx$-variables, we may assume that $L=\{V_1=V_2=0\}$.
Moreover, observe that for every $\bfv=(0,0,\bfv')\in L(\FF_q)$ we can rewrite
\[
\begin{aligned}
T_q(\bfv)
&=\sum_{\substack{\bfx\in\FF_q^n\\F(\bfx)=0}}
  \psi_q(\bfx\cdot\bfv)-q^{n-1}\bm{1}_{\bfv=\bm0}\\
&=\sum_{\bfx'\in\FF_q^{n-2}}\psi_q(\bfv'\cdot\bfx')
  \sum_{\substack{(x_1,x_2)\in\FF_q^2\\F(x_1,x_2,\bfx')=0}}1
  -q^{n-1}\bm{1}_{\bfv=\bm0}\\
&=\sum_{\bfx'\in\FF_q^{n-2}}
  \bigl(\nu_F(\bfx')-q\bigr)\psi_q(\bfv'\cdot\bfx'),
\end{aligned}
\]
where
\[
\nu_F(\bfx')=\#\{(x_1,x_2)\in\FF_q^2:F(x_1,x_2,\bfx')=0\},
\]
and $\bfx'=(x_3,\ldots,x_n)$. Hence
\begin{equation*}
\begin{split}
\sum_{\substack{\bfv\in\FF_q^n\\v_1=v_2=0}}|T_q(\bfv)|^2
&=\sum_{\substack{\bfv\in\FF_q^n\\v_1=v_2=0}}
\sum_{(\bfx',\bfy')\in\FF_q^{2n-4}}
\bigl(\nu_F(\bfx')-q\bigr)\bigl(\nu_F(\bfy')-q\bigr)
\psi_q\bigl(\bfv'\cdot(\bfx'-\bfy')\bigr) \\
&= q^{n-2}\sum_{\bfx'\in\FF_q^{n-2}} \bigl(\nu_F(\bfx')-q\bigr)^2,
\end{split}
\end{equation*}
by averaging over $\bfv$.

Since $X(\QQ)=\emptyset$, we have $F(X_1,X_2,0,\dots,0) \ne 0$,
so $\deg F(X_1,X_2,\bfx') = 3$ for all $\bfx'\in\FF_q^{n-2}$
if $p$ is sufficiently large.
By Noether's lemma on absolute irreducibility, as in \cite[Ch.~V, Thm.~2A]{Sch76}, there exists a polynomial
$G_{F,L}(\mathbf{X}')\in\mathbb{Z}[X_3,\ldots,X_n]$ such that
\[
G_{F,L}(\bfx')\neq0
\quad\Longrightarrow\quad
F(X_1,X_2,\bfx')\text{ is irreducible over }\bar{\FF}_q,
\]
for every sufficiently large prime $p$ and every $q=p^a$, where $G_{F,L}(\bfx')$ is evaluated in $\FF_q$.
Moreover, $G_{F,L}$ can be chosen nonzero because the
generic specialization is absolutely irreducible,
as we show in Lemma~\ref{generic-irred} below.

For sufficiently large $p$, the reduction of $G_{F,L}$ is nonzero, so there are $O_{F,L}(q^{n-3})$ vectors $\bfx'\in\FF_q^{n-2}$ at which $G_{F,L}$ vanishes.
For each such $\bfx'$, the polynomial $F(X_1,X_2,\bfx')$
has leading cubic part $F(X_1,X_2,0,\dots,0) \ne 0$,
provided that $p$ is large enough.
Thus $\nu_F(\bfx')\ll q$ trivially.
For the remaining vectors $\bfx'$, the Lang--Weil bound gives $|\nu_F(\bfx')-q|\ll q^{1/2}$.
Hence
\[
\begin{aligned}
\sum_{\bfv\in L(\FF_q)}|T_q(\bfv)|^2
&\ll_{F,L}
q^{n-2} \sum_{\substack{\bfx'\in\FF_q^{n-2}\\G_{F,L}(\bfx')=0}} q^2
+ q^{n-2} \sum_{\substack{\bfx'\in\FF_q^{n-2}\\G_{F,L}(\bfx')\neq0}} q \\
&\ll_{F,L}q^{2n-3},
\end{aligned}
\]
as claimed.
\end{proof}

\begin{lemma}
\label{generic-irred}
Let $K=\Qbar(T_3,\ldots,T_n)$.
Then $F(X_1,X_2,T_3,\dots,T_n)$ is irreducible in $\bar{K}[X_1,X_2]$.
\end{lemma}

\begin{proof}
Since $F(X_1,X_2,T_3,\dots,T_n)$ is irreducible in $\Qbar[T_3,\dots,T_n][X_1,X_2]$,
it is irreducible in $K[X_1,X_2]$ by Gauss's lemma.
If $F(X_1,X_2,T_3,\dots,T_n)$ were reducible in $\bar{K}[X_1,X_2]$,
then by Galois theory it would have three $\Gal(\bar{K}/K)$-conjugate linear factors
$F_i = X_1+b_iX_2+c_i\in \bar{K}[X_1,X_2]$, say,
possibly after switching the variables $X_1$ and $X_2$.
Equating top-degree terms in $X_1,X_2$ gives
$F(X_1,X_2,0,\dots,0)
\in \bar{K}^\ast \prod_{i=1}^{3} (X_1+b_iX_2)$.
Factoring $F(X_1,X_2,0,\dots,0)$ over $\Qbar$
gives $b_i\in \Qbar\subset K$.
Thus $\Gal(\bar{K}/K)$ fixes each $b_i$, whence $b_1=b_2=b_3$.
However, $X(\QQ) = \emptyset$, so the binary cubic form
$F(X_1,X_2,0,\dots,0) \in \QQ[X_1,X_2]$
is irreducible, and hence square-free.
This is a contradiction.
The statement of the lemma follows.
\end{proof}

We next record precisely how the moment will be used.
Fix a prime $\ell\ne p$ and a field isomorphism
$\iota:\Qbar_\ell\to \CC$.
Let $i:\cone X=\{F=0\}\hookrightarrow\A^n$, and let
$\pi_1,\pi_2$ be the projections of
$\A^n_{\bfx}\times\A^n_{\bfv}$.
For a nontrivial additive character $\psi$ of $\FF_p$, put
\begin{equation}
K_{p,\psi}:=R\pi_{2!}
\bigl(\pi_1^*i_*\Qbar_\ell
\otimes\mathcal L_\psi(\bfx\cdot\bfv)\bigr),
\label{n10:fourier-complex}
\end{equation}
where we use $\iota$ to identify the values of $\psi$ with elements of $\Qbar_\ell$ when
defining the Artin--Schreier sheaf $\mathcal L_\psi$.
Here $\mathcal L_\psi(\bfx\cdot\bfv)$ denotes the pullback
of $\mathcal L_\psi$ along the map $(\bfx,\bfv)\mapsto\bfx\cdot\bfv$.

For every $q=p^a$, $a\ge1$, and every
$\bfv\in\FF_q^n\setminus\{\bm0\}$, its trace is $T_q(\bfv)$,
where $\psi_q=\psi\circ\Tr_{\FF_q/\FF_p}$.
More precisely, by the construction of $K_{p,\psi}$ and the
Grothendieck--Lefschetz trace formula,
\begin{equation}
 T_q(\bfv)
 =t_{K_{p,\psi}}(\bfv;q)
 =\sum_i(-1)^i\iota\left(
   \operatorname{tr}\left(
     \operatorname{Fr}_{q,\bfv}
     \mid
     \mathcal H^i(K_{p,\psi})_{\bar{\bfv}}
   \right)\right),
   \label{eq : GLT}
\end{equation}
where $\operatorname{Fr}_{q,\bfv}$ denotes geometric Frobenius.
The sheaf inside $R\pi_{2!}$ in
\eqref{n10:fourier-complex} is punctually pure of weight $0$.
Consequently, by the compact-support weight theorem
\cite[Theorem~3.3.1]{Del}, each cohomology sheaf
$\mathcal H^i(K_{p,\psi})$ is constructible and mixed
with integer weights $\leq i$.

It will be useful to have the following \emph{uniform} constructibility result as $p$ varies.

\begin{lemma}\label{n10:ss-adapted}
The variety $\A^n \setminus \{\bm0\}$ has a finite partition into smooth
homogeneous locally closed sets $C_j$ defined over $\QQ$, such that
$\mathcal H^i(K_{p,\psi})$ is lisse on $C_j$
for all sufficiently large $p$
and every nontrivial $\psi$.
Moreover,
$\sum_i \rank \mathcal H^i(K_{p,\psi})\vert_{C_j} \ll_F 1$, uniformly over $p$ and $\psi$.
\end{lemma}

\begin{proof}
By \cite[Theorem~2.1]{FK01}, applied to the morphism
$\cone X\times\A^n \to \A^n$
in place of $X\to T$,
we obtain a finite partition $\A^n\setminus \{\bm0\} = \bigsqcup_j Y_j$
into locally closed sets $Y_j$
such that $\mathcal H^i(K_{p,\psi})$ is lisse on $Y_j$
for all sufficiently large $p$,
and such that $\sum_i \rank \mathcal H^i(K_{p,\psi})\vert_{Y_j} \ll_F 1$.
The sets $Y_j$
are not necessarily smooth and homogeneous.
We modify the partition $\{Y_j\}$ as follows.

Let $V := \A^n \setminus \{\bm0\}$.
Since $F$ is homogeneous,
the sheaves $\mathcal H^i(K_{p,\psi})$
factor through the projectivisation $V \to \PP^{n-1}$.
If $\dim Y_j = n$, then $Y_j$ contains a smooth dense open set $U_j \subset V$.
Then $\GG_m U_j\subset V$,
the union of all scalar multiples of $U_j$,
has the property that
$\mathcal H^i(K_{p,\psi})$ is lisse on $\GG_m U_j$
for all sufficiently large $p$.
This follows from gluing,
since the translates $\lambda U_j$ for $\lambda\in \GG_m$
form an open cover of $\GG_m U_j$.
Moreover, $\GG_m U_j$ is open in $V$, being a union of open sets $\lambda U_j$.

Choose $j_0$ for which $\dim Y_{j_0} = n$.
We then create a new partition of $V$
consisting of the set $C_0 = \GG_m U_{j_0}$
together with the sets $Y'_j = Y_j \setminus C_0$ for all $j$,
and let $V' := V \setminus C_0$.
Then the sets $Y'_j$ partition $V'$.
Moreover, $V'$ is homogeneous and closed in $V$.
If $\dim Y'_j$ is maximal over all $j$,
then $Y'_j$ contains a smooth open set $U'_j \subset V'$,
where $U'_j$ is dense in an irreducible component of $V'$.
Any irreducible component of $V'$ is homogeneous, so it follows that
$\mathcal H^i(K_{p,\psi})$ is lisse on $\GG_m U'_j \subset V'$
for all sufficiently large $p$.

Next, choose $j_1$ for which $\dim Y'_{j_1}$ is maximal and again let us again create a new partition of $V'$
consisting of the set $C_1 = \GG_m U'_{j_1}$
together with the sets $Y''_j = Y'_j \setminus C_1$ for all $j$,
and let $V'' := V' \setminus C_1$.
The set $V''$ is homogeneous and closed in $V'$.
Continuing in this manner,
decreasing the pair
$$(\dim V^{(r)},\text{number of top-dimensional irreducible components of }V^{(r)})$$
at each step $r\ge 1$,
we eventually obtain a partition
$V = \bigsqcup_{j'} C_{j'}$ into smooth homogeneous locally closed sets $C_{j'}$,
such that $\mathcal H^i(K_{p,\psi})$ is lisse on $C_{j'}$
for all sufficiently large $p$.
\end{proof}

\begin{rem}
Since $F$ is homogeneous,
there exists an alternative sheaf-theoretic interpretation of $T_q(\bfv)$
that involves constant coefficients instead of Artin--Schreier coefficients. This perspective is taken in Appendix~\ref{sec:n10-singular-support}. 
\end{rem}

\begin{prop}
\label{n10:moment-trace}
    Let $V\subset\mathbb{A}^{n}$ be a geometrically irreducible
    affine cone of codimension $r+1$ and degree $d$, defined
    over $\QQ$, and consider the restriction
    $K_{p,\psi}|_{V_{\FF_p}}$.
    Let $U\subset V\setminus\{\bm0\}$ be a nonempty smooth
    open set, defined over $\QQ$, such that, for all but
    finitely many primes $p$, the sheaves
    $\mathcal H^i(K_{p,\psi})|_{U_{\FF_p}}$ are lisse and
    the sum of their ranks is bounded independently of $p$.
    Then either
    \[
        T_p(\bfv)\ll_{F,V,U}p^{(n+r-2)/2}
        \qquad
        \bigl(\bfv\in U(\FF_p)\bigr)
    \]
    for all sufficiently large primes $p$, or $d\ge r+1$.
\end{prop}

\begin{proof}
Fix a sufficiently large prime $p$.
From \eqref{eq : GLT}, we know that the restriction of
$T_{p^a}$ to $U(\FF_{p^a})$ is an alternating sum of
trace functions of lisse sheaves that are mixed with
integer weights.
Hence, an application of
\cite[Theorem~3.5 and Remark~3.6]{Xu20}
with $w=n+r-2$ gives that either
\[
|T_{p^a}(\bfv)|\le Cp^{a(n+r-2)/2}
\qquad
\bigl(a\ge1,\ \bfv\in U(\FF_{p^a})\bigr),
\]
where $C$ is a uniform upper bound for the sum of the
ranks in the statement, or
\[
\limsup_{a\to\infty}
\frac{1}{p^{a(2n-2)}}
\sum_{\bfv\in U(\FF_{p^a})}
|T_{p^a}(\bfv)|^2
\ge1,
\]
since
\[
\dim U+w+1
=(n-r-1)+(n+r-2)+1
=2n-2.
\]
In the first case, taking $a=1$ gives the desired bound.
Let us therefore assume that the second alternative holds and,
by contradiction, that $d\le r$.
Consider the projective variety associated to $V$, say
$Y\subset\mathbb{P}^{n-1}$.
Since $Y$ has codimension $r+1$ in $\mathbb{P}^{n-1}$ and
degree $d\le r$, by applying
\cite[Lemma~3 and Corollary]{Swi73} twice, it follows that
$Y$ is contained in a linear subspace
$P\subset\mathbb{P}^{n-1}$ of codimension two.
Moreover, since $Y$ is defined over $\QQ$, its linear span
is also defined over $\QQ$, and we may choose $P$ to be
defined over $\QQ$.
Let $L\subset\mathbb{A}^{n}$ be the linear subspace
associated to $P$.
Then $V\subset L$ and $\dim L=n-2$.

By Lemma~\ref{n10:second-moment}, for every $a\ge1$ we have
\[
\sum_{\substack{
\bfv\in L(\FF_{p^a})\\
\bfv\ne\bm0}}
|T_{p^a}(\bfv)|^2
\ll_{F,L}p^{a(2n-3)}.
\]
On the other hand, since
$U\subset L\setminus\{\bm0\}$, positivity gives
\[
\begin{split}
1
&\le
\limsup_{a\to\infty}
\frac{1}{p^{a(2n-2)}}
\sum_{\bfv\in U(\FF_{p^a})}
|T_{p^a}(\bfv)|^2 \\
&\le
\limsup_{a\to\infty}
\frac{1}{p^{a(2n-2)}}
\sum_{\substack{
\bfv\in L(\FF_{p^a})\\
\bfv\ne\bm0}}
|T_{p^a}(\bfv)|^2
\\
&\ll \lim_{a\rightarrow\infty} p^{-a}
\\& =0.
\end{split}
\]
This gives a contradiction.
\end{proof}

\section{Hybrid cube-free stratifications}
\label{sec:MarkV-construction}

Throughout this section $n\ge 10$ and $X(\QQ)=\emptyset$.

\subsection{A fixed partition of the rational frequencies}

We first obtain a stratification using Theorem~\ref{n10:ss}.
Recall from \eqref{eq:pST} that $S_p(\bfv)=pT_p(\bfv)$.
Combining Theorem~\ref{n10:ss} with
Corollary~\ref{COR:general-n-exceptional-hyperplanes}
and Proposition~\ref{prop:prime-modulus-pointwise},
we obtain a stratification
\[
\A^{n}=Z_0
\supset Z_{1}\supset\cdots\supset Z_{n-3}
\]
of homogeneous closed sets $Z_j$
(affine cones of projective varieties),
for $0\le j\le n-3$,
with the following properties:
\begin{itemize}
\item[$i)$]
For $0\le j\le n-4$ we have $\cone W_j \subset Z_{j+1} \subset \cone B_j$,
in the notation of \eqref{EQ:general-bad-hyperplanes}
and \eqref{EQ:general-large-gauss-fibres}.

\item[$ii)$]
For every $j=0,\ldots,n-3$, the variety $Z_j$ has codimension
at least $j$ in $\A^n$.
Moreover, $\dim Z_{n-5} \le 4$,
$\dim Z_{n-4} \le 3$
and $Z_{n-3}=\{\bm0\}$.

\item[$iii)$]
Fix integral models of $Z_j$.
If $0\le j\le n-6$, then
for every $\bfv\in (Z_j\setminus Z_{j+1})(\FF_p)$, we have
\begin{equation}
\label{eq:FK-benchmark}
S_p(\bfv)\ll_F p^{\frac{n+1+j}{2}}.
\end{equation}
If $j=n-5$,
then \eqref{eq:FK-benchmark} holds for all $\bfv\in (Z_j\setminus Z_{j+2})(\FF_p)$.
If $j=n-4$,
then \eqref{eq:FK-benchmark} holds for all $\bfv\in Z_j(\FF_p)$, including $\bfv=\bm0$.

\item[$iv)$]
For every $j=1,\ldots,n-4$,
if $R_j$ is the Zariski closure of $Z_j(\QQ)$ in $Z_j$, then
\begin{equation}
\label{ZjRj6}
Z_j(\QQ)=R_j(\QQ)
\end{equation}
and $\dim R_j \le n-j-1$.
Moreover, $\dim R_{n-5} \le 1$
and $R_{n-4} = \{\bm0\}$.
\end{itemize}
In particular, combining the inclusion $\cone W_j \subset Z_{j+1}$
with Lemma~\ref{lem:prime-square-stratification},
we find that if $0\le j\le n-5$ then 
for every $\bfv\in \ZZ^n$ with $\bfv\bmod{p}\in (Z_j\setminus Z_{j+1})(\FF_p)$, we have
\begin{equation}
\label{eq:square-benchmark}
S_{p^2}(\bfv)\ll_F p^{n+1+j}.
\end{equation}
Moreover, if $j=n-4$,
then \eqref{eq:square-benchmark} holds for all $\bfv\in \ZZ^n$, including $\bfv\in p\ZZ^n$.

We now modify the stratification for square-free moduli
so that it will be more convenient to work with.
It will be more convenient for us to rearrange the vectors
$\bfv\in\QQ^n$ according to the weight of the exponential sum
on each component of $Z_j$.
For example, suppose that $L\subset Z_j$ is a linear space of
dimension $n-j-1$, defined over $\QQ$, with
$L\not\subset Z_{j+1}$, where $1\le j\le n-6$.
By Lemma~\ref{n10:ss-adapted},
there exists a smooth
dense open subset
\[
U\subset L\setminus Z_{j+1},
\]
defined over $\QQ$, on which the cohomology sheaves of
$K_{p,\psi}$ are lisse for all sufficiently large primes $p$,
with uniformly bounded total rank.
Proposition \ref{n10:moment-trace} then gives
\[
S_p(\bfv)\ll_{F,L,U}p^{\frac{n+j}{2}}
\qquad (\bfv\in U(\FF_p)).
\]
Hence we can assign $U(\QQ)$ to level $j-1$ in the rearranged
partition, while leaving the lower-dimensional complement
$L\setminus U$ in $Z_j$ to be treated separately.

With this example in mind, for every $j=1,\ldots,n-6$, consider
\begin{equation}
\label{n10:promotion-loci}
L_j:=
\bigcup_{\substack{
Y\text{ a geometric irreducible component of }R_j\\
Y\text{ defined over }\QQ,\quad Y\not\subset Z_{j+1}\\
\dim Y=n-j-1,\quad \deg Y\leq j
}}
Y.
\end{equation}
By Lemma~\ref{n10:ss-adapted}, we can 
choose a smooth homogeneous dense open subset
\begin{equation}
\label{n10:promotion-opens}
U_j\subset
L_j\setminus Z_{j+1},
\end{equation}
defined over $\QQ$, such that the restrictions to $U_j$ of the
cohomology sheaves of $K_{p,\psi}$ are lisse for all sufficiently
large primes $p$ and every nontrivial $\psi$, with uniformly
bounded total rank. If $L_j=\emptyset$, take
$U_j=\emptyset$.

Recall the convention $Z_0=\A^n$ from above.
With the additional conventions
\[
U_0=U_{n-5}:=\emptyset,
\]
we define, for every $0\leq j\leq n-6$,
\begin{equation}
\label{n10:rational-partition}
P_j:=
\left(
\bigl((Z_j\setminus Z_{j+1})\setminus U_j\bigr)
\cup U_{j+1}
\right)(\QQ)\setminus\{\boldsymbol{0}\},
\end{equation}
and put
\[
P_{n-5}:=Z_{n-5}(\QQ)\setminus\{\bm0\},\quad P_{n-4}=\{\boldsymbol{0}\}.
\]
The sets $(Z_j\setminus Z_{j+1})(\QQ)$ for $0\le j\le n-6$,
together with the sets $P_{n-5}$ and $P_{n-4}$, certainly partition $\QQ^n$.
Therefore, $\{P_j\}_{j=0}^{n-4}$ is also a partition of $\QQ^n$.
Indeed, each $U_j(\QQ)$, for $1\le j\le n-6$, is removed from the set
$(Z_j\setminus Z_{j+1})(\QQ)$ and assigned to the preceding set $P_{j-1}$.
We record now some properties of this partition which will be used extensively later.
For any set $Y \subset \QQ^n$, let
\begin{equation*}
N(Y,L,\bfv_0,m,\bfb) :=
\#\{\bfv\in Y\cap\ZZ^n:
|\bfv-\bfv_0|\le L,\;
\bfv\equiv\bfb\bmod m\}.
\end{equation*}

\begin{prop}
\label{PROP:Dante-partition-Pj}
Let $\{P_{j}\}_{j}$ be the partition of $\QQ^{n}$ defined above.
Then the following hold.
\begin{itemize}
\item[$i)$] For every $0\leq j\leq n-5$ and every $\bfv\in P_{j}\cap \ZZ^n$,
there exists a positive integer
$$\Delta(\bfv)\ll_{F}(2+|\bfv|)^{O_{F}(1)}$$
such that
\[
p\nmid \Delta(\bfv)
\quad\Longrightarrow\quad
S_{p}(\bfv)\ll_{F} p^{\frac{n+j+1}{2}}.
\]
\item[$ii)$] Set
\[
T=1+\frac{L}{m}.
\]
Then, for every $m,L\geq 1$, $\bfv_0\in \RR^n$, and every class $\bfb\bmod m$, one has
\begin{equation}\label{n10:progression-counts}
N(P_j,L,\bfv_0,m,\bfb)
\ll_{F,\ve}
\begin{cases}
T^{n}          &\text{if $j=0$},\\
T^{n-j-2+\frac{1}{j+1}+\ve}     &\text{if $1\leq j\leq 2$},\\
T^{n-j-2+\ve}  &\text{if $3\leq j\leq n-6$},\\
T         &\text{if $j=n-5$}.
\end{cases}
\end{equation}
\end{itemize}
\end{prop}

\begin{proof}
We start by establishing part $(i)$. Assume that
$\bfv\in P_j\cap\ZZ^n$, initially with
$0\le j\le n-6$. Fix an integer $N\ge1$, depending only on
$F$, such that all the chosen models and local bounds are
valid for $p\nmid N$.
By combining the bound \eqref{eq:FK-benchmark} with
Proposition \ref{n10:moment-trace}, we obtain
\[
S_p(\bfv)\ll_F p^{\frac{n+j+1}{2}}
\]
whenever $p\nmid N$, the reduction of $\bfv$ is nonzero,
and
\begin{equation}
\label{EQ:modified-FK-condition}
\bfv\bmod p\notin
(Z_{j+1}\setminus U_{j+1})(\FF_p).
\end{equation}

By the definition of $P_j$, either $\bfv\notin Z_{j+1}$
or $\bfv\in U_{j+1}$. In the first case, choose $g$ from
a fixed finite family of integral defining equations for
$Z_{j+1}$ such that $g(\bfv)\ne0$. Then, for $p\nmid N$,
\[
p\nmid g(\bfv)
\quad\Longrightarrow\quad
\bfv\bmod p\notin Z_{j+1}(\FF_p).
\]
In the second case, $U_{j+1}$ is open in $L_{j+1}$, so
$L_{j+1}\setminus U_{j+1}$ is closed. Choose $g$ from
a fixed finite family of integral defining equations for
this complement such that $g(\bfv)\ne0$. Since
$\bfv\in L_{j+1}$, its reduction belongs to
$L_{j+1}(\FF_p)$ for $p\nmid N$. Hence
\[
p\nmid Ng(\bfv)
\quad\Longrightarrow\quad
\bfv\bmod p\in U_{j+1}(\FF_p).
\]
In both cases define
$\Delta(\bfv):=N\,|g(\bfv)|$.
Then $p\nmid\Delta(\bfv)$ ensures that the reduction of
$\bfv$ is nonzero and lies outside
$(Z_{j+1}\setminus U_{j+1})(\FF_p)$, so the required
bound follows. Since $g$ is chosen from a fixed finite
family of polynomials depending only on $F$, we have
\[
1\le\Delta(\bfv)\ll_F(2+|\bfv|)^{O_F(1)}.
\]

Finally, if $j=n-5$, take
$\Delta(\bfv):=N|\bfv|$.
As noted after \eqref{eq:FK-benchmark},
Proposition~\ref{prop:prime-modulus-pointwise} then gives
\[
S_p(\bfv)\ll_F p^{n-2}
=p^{\frac{n+j+1}{2}}
\qquad(p\nmid\Delta(\bfv)).
\]
This completes the proof of part $(i)$.

Let us turn our attention to part $(ii)$.
When $j=0$, the
result is trivial since $P_0\subset\QQ^n$.
Let us deal with the cases $j=1,2$. Since
$Z_j(\QQ)=R_j(\QQ)$ by \eqref{ZjRj6}, the definition
of $P_j$ gives
\[
\begin{aligned}
N(P_j,L,\bfv_0,m,\bfb)
&\leq N\bigl(R_j\setminus(Z_{j+1}\cup U_j),L,\bfv_0,m,\bfb\bigr)
      +N(Z_{j+1},L,\bfv_0,m,\bfb)\\
&=N\bigl(R_j\setminus(Z_{j+1}\cup U_j),L,\bfv_0,m,\bfb\bigr)
      +N(R_{j+1},L,\bfv_0,m,\bfb).
\end{aligned}
\]
Since $\codim R_{j+1}\geq j+2$,
Lemma~\ref{LEM:basically-zippel}
implies that
\[
N(R_{j+1},L,\bfv_0,m,\bfb)\ll_F T^{n-j-2},
\]
which is acceptable. Let $Y$ be a geometric irreducible component
of the Zariski closure
\[
\bar{R_j\setminus(Z_{j+1}\cup U_j)}
\subset \A^n
\]
of $R_j\setminus(Z_{j+1}\cup U_j)$ in $\A^n$.
Based on the definition of $L_j$ and the density of $U_j$ in $L_j$
(see \eqref{n10:promotion-loci} and \eqref{n10:promotion-opens}),
we distinguish three cases.
\begin{itemize}
\item[$i)$]
The component $Y$ is defined over $\QQ$ and has
$\deg Y\geq j+1$. Hence, by the uniform
Bombieri--Pila bound recorded in \cite[Theorem~A]{Pila}, we have
\[
N(Y,L,\bfv_0,m,\bfb)
\ll_{F,\varepsilon}
T^{n-j-2+\frac{1}{j+1}+\ve}.
\]

\item[$ii)$]
The component $Y$ is not defined over $\QQ$. In this case,
there exists a subvariety $W\subset Y$, defined over
$\QQ$, of codimension at least $1$ in $Y$, such that
$Y(\QQ)\subset W(\QQ)$. Hence,
by Lemma~\ref{LEM:basically-zippel},
we have
\[
N(Y,L,\bfv_0,m,\bfb)
\ll_F T^{n-j-2}.
\]

\item[$iii)$]
The component $Y$ has $\deg Y\leq j$ and is defined over $\QQ$.
Then necessarily $\dim Y\leq n-j-2$,
or else $Y$ would be contained in $L_j$ by \eqref{n10:promotion-loci}.
Hence,
by Lemma~\ref{LEM:basically-zippel},
we have
\[
N(Y,L,\bfv_0,m,\bfb)
\ll_F T^{n-j-2}.
\]
\end{itemize}
Summing over the finitely many components, the bound follows.
For $3\leq j\leq n-6$, the argument is almost the same, except
for the treatment of the geometrically irreducible components
of
\[
\bar{R_j\setminus(Z_{j+1}\cup U_j)}
\]
that are defined over $\QQ$ and have degree at least
$j+1\geq4$. If $Y$ is such a component, then we have
\[
N(Y,L,\bfv_0,m,\bfb)
\ll_{F,\ve} T^{n-j-2+\ve},
\]
by \cite[Theorem~1.2]{Vermeulen}, provided that $Y$ is not
cylindrical over a curve.

We claim that $Y$ is not cylindrical over a curve. Since
$R_j$, $Z_{j+1}$ and $U_j$ are homogeneous, $Y$ is an affine
cone associated to a projective variety. Put $r:=\dim Y$.
If $Y$ were cylindrical over a curve, there would exist a
surjective $\QQ$-linear map
\[
\pi:\A^n\longrightarrow\A^{n-r+1}
\]
and an affine curve $C$ such that $Y=\pi^{-1}(C)$.
Since $\pi$ is linear and $Y$ is homogeneous, $C$ is also
homogeneous. Moreover, $C$ is geometrically irreducible,
since $Y$ is geometrically irreducible. Hence $C$ must be
a line. It follows that $Y=\pi^{-1}(C)$ is a linear space,
contradicting $\deg Y\geq4$.
The case $j=n-5$ follows from the fact that $R_{n-5}$ has dimension $\leq 1$ and $R_{n-4}=\{\mathbf 0\}$.
\end{proof}

For prime squares, recall the bounds in \eqref{eq:square-benchmark}.
Hence we obtain a partition $\{Q_j\}_{j=0}^{n-4}$ of $\QQ^n$,
where
\[
Q_j=(Z_j\setminus Z_{j+1})(\QQ)
\qquad (0\leq j\leq n-5),
\qquad
Q_{n-4}=\{\bm0\},
\]
where we recall the convention that $Z_0=\A^n$. Here we use the fact that $Z_{n-4}(\QQ)=\{\bm0\}$.

As for the partition associated with prime moduli, we record
the following proposition.
\begin{prop}
\label{PROP:partition-Qj}
Let $\{Q_j\}_{j=0}^{n-4}$ be the partition of $\QQ^n$
defined above. Then the following hold.
\begin{itemize}
\item[$i)$]
For every $0\leq j\leq n-4$ and every
$\bfv\in Q_j\cap\ZZ^n$, there exists a positive integer
\[
\Theta(\bfv)\ll_F(2+|\bfv|)^{O_F(1)}
\]
such that
\[
p\nmid\Theta(\bfv)
\quad\Longrightarrow\quad
S_{p^2}(\bfv)\ll_F p^{n+j+1}.
\]

\item[$ii)$]
Set 
\[
T=1+\frac{L}{m}.
\]
Then, for every $m,L\geq 1$, 
$\bfv_0\in\RR^n$, and every class $\bfb\bmod m$, one has
\begin{equation}\label{n10:Q-progression-counts_square}
N(Q_j,L,\bfv_0,m,\bfb)
\ll_F
\begin{cases}
T^n
&\text{if $j=0$},\\
T^{n-j-1}
&\text{if $1\leq j\leq n-6$},\\
T
&\text{if $j=n-5$},\\
1
&\text{if $j=n-4$}.
\end{cases}
\end{equation}
In particular, for $1\le j\le n-4$ we have $N(Q_j,L,\bfv_0,m,\bfb)
\ll_F T^{n-j-1}$.
\end{itemize}
\end{prop}

\begin{proof}
The arguments are analogous to those  of Proposition \ref{PROP:Dante-partition-Pj},
except that the proof of \eqref{n10:Q-progression-counts_square} only requires
Lemma~\ref{LEM:basically-zippel}, rather than 
any progress towards the dimension growth conjecture.
\end{proof}

\subsection{The cube-free averages}\label{sec:silver}

In this section we are going to combine the two partitions of $\QQ^n$ we have constructed.
Since $\{P_j\}_j$ and $\{Q_k\}_k$
are two partitions of $\QQ^n$,
the set $\{P_j\cap Q_k\}_{j,k=0}^{n-4}$
is also a partition of $\QQ^n$.
It will be useful to record the following fact:
observe that by definition of our partitions,
together with the definition of $U_j$ in \eqref{n10:promotion-opens},
we have
\begin{equation}
\label{EQ:PjQk-formula}
P_j\cap Q_k=
\begin{cases}
\bigl((Z_j\setminus Z_{j+1})\setminus U_j\bigr)(\QQ)
\setminus\{\boldsymbol{0}\}
& 0\le j\le n-6 \text{ and } k=j,\\[2pt]
U_{j+1}(\QQ)
& 0\le j\le n-6 \text{ and } k=j+1,\\[2pt]
Z_{n-5}(\QQ)\setminus\{\boldsymbol{0}\}
& j=k=n-5,\\[2pt]
\{\boldsymbol{0}\}
& j=k=n-4,\\[2pt]
\emptyset
& \text{otherwise}.
\end{cases}
\end{equation}

\begin{defin}
Let $N_F$ be the product of finitely many primes depending only on $F$.
Let $\bfv\in \ZZ^n$.
For $p\nmid N_F$, let $\bar{\bfv}$ denote the reduction
of $\bfv$ modulo $p$, and put
\[
\begin{aligned}
e(\bfv;p)
&=\min\left(
 \left\{0\le h\le n-5:
 \bar{\bfv}\notin
 (Z_{h+1,p}\setminus U_{h+1,p})(\FF_p)\right\}
 \cup\{n-4\}\right),\\
f(\bfv;p)
&=\min\left(
 \left\{0\le h\le n-5:
 \bar{\bfv}\notin Z_{h+1,p}(\FF_p)\right\}
 \cup\{n-4\}\right),
\end{aligned}
\]
with the convention $U_{h,p}=\emptyset$ for $h\ge n-5$.
For $p\mid N_F$, put $e(\bfv;p)=f(\bfv;p)=0$.
\end{defin}

The definitions of $e(\bfv;p)$ and $f(\bfv;p)$
are motivated by
\eqref{EQ:modified-FK-condition} and \eqref{eq:square-benchmark},
respectively.
For cube-free moduli $d$, we now define further key invariants
based on the factorisation type of $d$.

\begin{defin}

Every positive cube-free integer $d$ can be written uniquely as
$d=d_1d_2^2$, where $d_1,d_2$ are coprime positive square-free
integers.
Suppose now that $\bfv\in P_j\cap Q_k\cap \ZZ^n$, so that necessarily
\begin{equation*}
0\le j\le k\le n-4.
\end{equation*}
Then let
\begin{equation}\label{n10:K}
K(d,\bfv)=
\prod_{\substack{p\mid d_1\\e(\bfv;p)> j+2}}
p^{\frac{e(\bfv;p)-j}{2}-1}
\prod_{\substack{p\mid d_2\\f(\bfv;p)> k+2}}
p^{f(\bfv;p)-k-2}.
\end{equation}
Moreover, for all integers $h\ge j+1$ and $\ell\ge k+1$ define the formal infinite products
\[
q_{\bfv,h}=\prod_{\substack{p\\e(\bfv;p)\ge h}}p,
\qquad
s_{\bfv,\ell}=\prod_{\substack{p\\f(\bfv;p)\ge \ell}}p,
\]
and let
\begin{equation}\label{n10:radicals}
R_{h,\ell}(d,\bfv)
= \gcd(d_1,q_{\bfv,h})
\gcd(d_2,s_{\bfv,\ell}).
\end{equation}
\end{defin}

We start with the following lemma
based on the invariants $K(d,\bfv)$ and $R_{j+2,k+2}(d,\bfv)$.
\begin{lemma}\label{n10:Lemma-weight-gen}
Let $0\le j,k\le n-5$ and
$\bfv\in(P_j\cap Q_k)\cap\ZZ^n$.
If $D_1,D_2\ge 1/2$,
then for any two square-free integers $d_{1}\sim D_1$ and $d_2\sim D_2$
with $\gcd(d_{1},d_{2})=1$,
and for every $\ve>0$, we have
\begin{equation}
\label{EQ:K-relaxed-hybrid-pointwise}
S_{d_{1}d_{2}^{2}}(\bfv)
\ll_{F,\ve} D^\ve D_1^{\frac{n+j+3}{2}}D_2^{n+k+3} K(d_1d_2^2,\bfv),
\end{equation}
where $D = D_1D_2^2$.
In fact, if $K_j\ge 1/2$ then
\begin{equation}\label{n10:weighted-fixed-point}
\sideset{}{^*}\sum_{d_1\sim D_1}
\sideset{}{^*}\sum_{\substack{
d_2\sim D_2\\\gcd(d_1,d_2)=1
\\ K(d_1d_2^2,\bfv)\sim K_j}}
|S_{d_1d_2^2}(\bfv)|
\ll_{F,\ve}
[D(2+|\bfv|)]^\ve
D_1^{\frac{n+j+3}{2}}D_2^{n+k+3}
K_j,
\end{equation}
where every starred sum is restricted to square-free integers.
Moreover, if $j+1\le h\le n-3$ and $k+1\le\ell\le n-3$, then
\begin{equation}
\label{n10:residual-fixed-point}
\sideset{}{^*}\sum_{d_1\sim D_1}
\sideset{}{^*}\sum_{\substack{
d_2\sim D_2\\\gcd(d_1,d_2)=1
\\ R_{h,\ell}(d_1d_2^2,\bfv) = 1}}
|S_{d_1d_2^2}(\bfv)|
\ll_{F,\ve}
[D(2+|\bfv|)]^\ve
D_1^{\frac{n+\max\{j+3,h\}}{2}}
D_2^{n + \max\{k+2,\ell\}}.
\end{equation}
\end{lemma}

\begin{proof}
In the notation of Propositions~\ref{PROP:Dante-partition-Pj}
and~\ref{PROP:partition-Qj}, we have
\[
\begin{aligned}
p\nmid\Delta(\bfv)
&\quad\Longrightarrow\quad
S_p(\bfv)\ll_F p^{\frac{n+j+1}{2}},\\
p\nmid\Theta(\bfv)
&\quad\Longrightarrow\quad
S_{p^2}(\bfv)\ll_F p^{n+k+1}.
\end{aligned}
\]
For $q=\gcd(d_1,q_{\bfv,j+1},\Delta(\bfv))$
and $s=\gcd(d_2,s_{\bfv,k+1},\Theta(\bfv))$,
multiplicativity
and the local bounds \eqref{EQ:modified-FK-condition} and \eqref{eq:square-benchmark}
give
\begin{equation}
\label{EQ:general-S-expansion}
S_{d_1d_2^2}(\bfv)
\ll_{F,\ve}
D^\ve d_1^{\frac{n+j+1}{2}}d_2^{n+k+1}
\prod_{p\mid q}p^{(e(\bfv;p)-j)/2}
\prod_{p\mid s}p^{f(\bfv;p)-k}.
\end{equation}
Here the product of the local constants is absorbed into
$D^\ve$. By the definition of $K$, we have
\[
\prod_{p\mid q}p^{(e(\bfv;p)-j)/2}
\prod_{p\mid s}p^{f(\bfv;p)-k}
\le
K(qs^2,\bfv)\,qs^2.
\]
To conclude the proof of \eqref{EQ:K-relaxed-hybrid-pointwise},
we observe that $qs^2 \le d_1d_2^2$ and
$K(qs^2,\bfv) \le K(d_1d_2^2,\bfv)$,
since the exponents in the definition \eqref{n10:K} are positive.

For fixed $q,s$, there are
$O(D_1D_2/(qs))$ choices for $d_1,d_2$ satisfying
$q\mid d_1$ and $s\mid d_2$.
Consequently, the left-hand side of \eqref{n10:weighted-fixed-point} is
\[
\ll_{F,\ve}
D^\ve D_1^{\frac{n+j+1}{2}}D_2^{n+k+1}
\sum_{\substack{q\mid \gcd(q_{\bfv,j+1},\Delta(\bfv))\\q\le2D_1}}
\sum_{\substack{
s\mid \gcd(s_{\bfv,k+1},\Theta(\bfv))\\s\le2D_2\\\gcd(q,s)=1}}
K_j qs^2 \frac{D_1D_2}{qs}.
\]
Using $s\le 2D_2$ and the divisor bound on the positive integers
$\Delta(\bfv),\Theta(\bfv) \ll (2+|\bfv|)^{O(1)}$,
the claim \eqref{n10:weighted-fixed-point} follows.

Finally, if $R_{h,\ell}(d_1d_2^2,\bfv) = 1$,
then $e(\bfv;p)\le h-1$ and $f(\bfv;p)\le \ell-1$
in the inequality \eqref{EQ:general-S-expansion},
so
\begin{equation*}
S_{d_1d_2^2}(\bfv)
\ll_{F,\ve}
D^\ve d_1^{\frac{n+j+1}{2}}d_2^{n+k+1}
q^{(h-1-j)/2} s^{\ell-1-k}.
\end{equation*}
Proceeding as we did for \eqref{n10:weighted-fixed-point},
we find that the left-hand side of \eqref{n10:residual-fixed-point} is
\[
\ll_{F,\ve}
D^\ve D_1^{\frac{n+j+1}{2}}D_2^{n+k+1}
\sum_{\substack{q\mid \gcd(q_{\bfv,j+1},\Delta(\bfv))\\q\le2D_1}}
\sum_{\substack{
s\mid \gcd(s_{\bfv,k+1},\Theta(\bfv))\\s\le2D_2\\\gcd(q,s)=1}}
q^{(h-1-j)/2} s^{\ell-1-k} \frac{D_1D_2}{qs}.
\]
The divisor bound
yields
\eqref{n10:residual-fixed-point} on optimising over $q$ and $s$,
since $\frac{n+j+3}{2} + \max\{0,(h-3-j)/2\}
= \frac{n+\max\{j+3,h\}}{2}$
and $n+k+2 + \max\{0,\ell-2-k\}
= n + \max\{k+2,\ell\}$.
\end{proof}

Let $D\ge 1/2$, $L\ge1$, $m\in\ZZ_{\ge1}$, $\bfv_0\in\RR^n$, and $\bfb\in\ZZ^n$.
For the statements of the next three lemmas, we write
\begin{equation}\label{n10:average-box}
\begin{split}
T&=1+L/m,\qquad H=2+|\bfv_0|+L+m,\\
\mathcal V&=\{\bfv\in\ZZ^n:|\bfv-\bfv_0|\le L,
\ \bfv\equiv\bfb\bmod m\}.
\end{split}
\end{equation}
All constants below will be uniform in the parameters of
\eqref{n10:average-box}.

\begin{lemma}
\label{LEM:n5divisorbound}
Let notation be as in \eqref{n10:average-box}.
For every $\varepsilon>0$,
\[
\sum_{\substack{
d\sim D\\d\ \mathrm{cube\text{-}free}\\
\gcd(d,m)=1}}
\sum_{\substack{
\bfv\in(P_{n-5}\cap Q_{n-5})\cap\mathcal V}}
|S_d(\bfv)|\ll_{F,\ve} (DH)^\ve D^{n-1}T.
\]
\end{lemma}

\begin{proof}
Let $j=k=n-5$ and $h=\ell=n-3$.
If $\bfv\in P_j\cap Q_k\cap \ZZ^n$,
then by \eqref{n10:radicals} we have $R_{h,\ell}(d,\bfv)=1$ for every $d$.
Hence, by combining \eqref{n10:residual-fixed-point}
with the bound \eqref{n10:progression-counts} for $P_{n-5}$,
the lemma follows. 
\end{proof}

Henceforth we introduce the convenient indicator notation
\[
\delta_S(j)=
\begin{cases}
1&\text{ if $j\in S$},\\
0&\text{ otherwise},
\end{cases}
\]
for any finite list of integers $S$.
We also let $\delta_0(j,k)=\bm1_{j=k=0}$.
In this notation we have 
\begin{equation}\label{eq:silver}
\dim Z_{h}\leq n-h-\delta_{n-5,n-4}(h),
\end{equation}
for any integer $0\leq h\leq n-4$.

By \eqref{EQ:PjQk-formula},
we know that if $P_j\cap Q_k$ contains a nonzero vector $\bfv$, then
$0\le j\le k\le n-5$ and $k\in \{j,j+1\}$.
Having studied the case $j=n-5$ in Lemma~\ref{LEM:n5divisorbound},
we turn our attention to the remaining cases $0\le j\le n-6$.
It will be convenient to make a further case division based on the size of
the integer $R_{j+2,k+2}(d,\bfv)$.
We first treat the situation where $R_{j+2,k+2}(d,\bfv)\le T$.

\begin{lemma}\label{n10:weighted-conductor}
Let notation be as in \eqref{n10:average-box}.
Suppose $0\le j\le n-6$ and $k\in \{j,j+1\}$.
For every $K_j\ge 1/2$ and $\ve>0$,
\begin{equation}\label{n10:weighted-mean}
\sum_{\substack{
d\sim D\\d\ \mathrm{cube\text{-}free}\\
\gcd(d,m)=1}}
\sum_{\substack{
\bfv\in(P_j\cap Q_k)\cap\mathcal V\\
R_{j+2,k+2}(d,\bfv)\le T\\
K(d,\bfv) \sim K_j}}
|S_d(\bfv)|
\ll_{F,\ve}
(DH)^\ve T^{\theta_n(j,k)}D^{\frac{n+j+3}{2}}
K_j^{-\alpha_n\delta_0(j,k)},
\end{equation}
where $\theta_n(j,k) = n-j-1+\delta_0(j,k)-\delta_{n-7,n-6}(j)$
and $\alpha_n = 1/(n-8)$.
\end{lemma}

\begin{proof}
Let
\[
 r=R_{j+2,k+2}(d,\bfv),\qquad c=\gcd(d,r^2),\qquad b=d/c.
\]
Then $\rad(c)=r\le T$, $\gcd(b,c)=1$, $R_{j+2,k+2}(b,\bfv)=1$,
and $K(b,\bfv)=1$, so $K(d,\bfv)=K(c,\bfv)$.
This follows from the definitions \eqref{n10:radicals} and \eqref{n10:K}.
For fixed $c,\bfv$,
we have
\begin{equation}\label{eq:fix-cv}
    \sum_{\substack{\mathrm{cube\text{-}free}\ b\sim D/c \\ R_{j+2,k+2}(b,\bfv)=1}}
|S_b(\bfv)| \ll_{F,\ve} (DH)^\ve (D/c)^{(n+j+3)/2},
\end{equation}
by \eqref{n10:residual-fixed-point}, since $k\le j+1$.
Consequently the left-hand side $\mathcal L$ of \eqref{n10:weighted-mean} satisfies
\begin{equation}\label{n10:small-radical-allocation}
\mathcal L\ll_{F,\ve} (DH)^\ve D^{(n+j+3)/2}
\sum_{\substack{r\le T\\r\ \mathrm{square\text{-}free}\\\gcd(r,m)=1}}
\sum_{\substack{\rad(c)=r\\c\ \mathrm{cube\text{-}free}\\c\le2D}}
\sum_{\substack{\bfv\in(P_j\cap Q_k)\cap\mathcal V\\
R_{j+2,k+2}(c,\bfv)=r\\K(c,\bfv)\sim K_j}}
\frac{|S_c(\bfv)|}{c^{(n+j+3)/2}}.
\end{equation}

Now write $c=d_1d_2^2$, so that $r=d_1d_2$.
By \eqref{EQ:K-relaxed-hybrid-pointwise}, we have
$$\frac{S_c(\bfv)}{c^{(n+j+3)/2}}
\ll (DH)^\ve d_2^{k-j} K(d_1d_2^2,\bfv)
= (DH)^\ve d_2^{k-j}
\prod_{p\mid d_1}
p^{\frac{e(\bfv;p)-j}{2}-1}
\prod_{p\mid d_2}
p^{f(\bfv;p)-k-2},$$
by \eqref{n10:K}, since $e(\bfv;p)\ge j+2$ and $f(\bfv;p)\ge k+2$.
Put
\[
B_{j,k}=\left(\frac{\delta_{1,2}(j)}{j+1}
+2\delta_0(j)\right) \bm1_{k=j}.
\]
By \eqref{n10:progression-counts} if $k=j$,
and by \eqref{n10:Q-progression-counts_square} if $k=j+1$,
a fixed residue class modulo $r$ contains
$$\ll_{F,\ve} (T/r)^{n-j-2+B_{j,k}+\ve}$$
vectors $\bfv\in(P_j\cap Q_k)\cap\mathcal V$.
Moreover, in the light of \eqref{eq:silver},  the Chinese remainder theorem
and the Lang--Weil bound on the varieties $Z_h$ for $j+2\le h\le n-4$,
the number of residues $\bfv\bmod{r}$ for which $e(\bfv;p) = h_p$ for all $p\mid d_1$
and $f(\bfv;p) = \ell_p$ for all $p\mid d_2$ is
\begin{equation*}
\ll_{F,\ve} r^\ve \prod_{p\mid d_1} p^{n-h_p-\delta_{n-5,n-4}(h_p)}
\prod_{p\mid d_2} p^{n-\ell_p-\delta_{n-5,n-4}(\ell_p)}.
\end{equation*}
When $j=k=0$, since $K(c,\bfv) \gg K_j$, it will also be convenient to note that
\begin{equation*}
1 \ll K_j^{-\alpha_n\delta_0(j,k)} K(c,\bfv)^{\alpha_n\delta_0(j,k)},
\end{equation*}
where $\alpha_n=1/(n-8)$, as in the statement of the lemma.
Plugging the bounds above into \eqref{n10:small-radical-allocation} gives
\begin{equation*}
\mathcal L \ll_{F,\ve} (DH)^\ve D^{(n+j+3)/2} T^{n-j-2+B_{j,k}} K_j^{-\alpha_n\delta_0(j,k)}
\sum_{\substack{d_1d_2\le T\\d_1d_2\ \mathrm{square\text{-}free}\\\gcd(d_1d_2,m)=1}}
\sup_{\substack{j+2\le h_p\le n-4 \\ k+2\le \ell_p\le n-4}}
\prod_{p\mid d_1} p^{a_p}
\prod_{p\mid d_2} p^{b_p},
\end{equation*}
where
\begin{equation*}
\begin{split}
a_p
&= (\tfrac{h_p-j}{2}-1) (1+\alpha_n\delta_0(j,k))
- (n-j-2+B_{j,k})
+ (n-h_p-\delta_{n-5,n-4}(h_p)) \\
&\le (h_p-j-2) (1+\alpha_n\delta_0(j,k))
- (n-j-2+B_{j,k})
+ (n-h_p-\delta_{n-5,n-4}(h_p)) \\
&= (h_p-j-2) \alpha_n\delta_0(j,k)
-B_{j,k}
-\delta_{n-5,n-4}(h_p),
\end{split}
\end{equation*}
since $h_p\ge j+2$,
and where
\begin{equation*}
\begin{split}
b_p &= (k-j)
+ (\ell_p-k-2) (1+\alpha_n\delta_0(j,k))
- (n-j-2+B_{j,k})
+ (n-\ell_p-\delta_{n-5,n-4}(\ell_p)) \\
&= (\ell_p-k-2) \alpha_n\delta_0(j,k)
-B_{j,k}
-\delta_{n-5,n-4}(\ell_p).
\end{split}
\end{equation*}

In order to analyse the exponents $a_p$ and $b_p$, let
$$f_{j,k}(h) := (h-j-2) \alpha_n\delta_0(j,k)
-B_{j,k}
-\delta_{n-5,n-4}(h)$$
for $j+2\le h\le n-4$.
Then $a_p\le f_{j,k}(h_p)$ and $b_p\le f_{j,k}(\ell_p)$, since $k\ge j$.
If $n-7\le j\le n-6$, then $j\ge 3$ and $n-5\le h\le n-4$, so
$B_{j,k}=0$ and $f_{j,k}(h) = -1$, whence $a_p,b_p\le -1$ and
\begin{equation*}
\mathcal L \ll_{F,\ve} (DH)^\ve D^{(n+j+3)/2} T^{n-j-2}
\sum_{d_1d_2\le T} (d_1d_2)^{-1}
\ll_{F,\ve} (DH)^{2\ve} D^{(n+j+3)/2} T^{n-j-2},
\end{equation*}
which is satisfactory since $\delta_0(j,k) = 0$ and $\delta_{n-7,n-6}(j) = 1$.
If $0\le j\le n-8$ and $k\ge 1$, then $f_{j,k}(h) \le -B_{j,k}$,
where $0\le B_{j,k}\le 1$, so $a_p,b_p\le -B_{j,k}$ and
\begin{equation*}
\mathcal L \ll_{F,\ve} (DH)^\ve D^{(n+j+3)/2} T^{n-j-2+B_{j,k}}
\sum_{d_1d_2\le T} (d_1d_2)^{-B_{j,k}}
\ll_{F,\ve} (DH)^{2\ve} D^{(n+j+3)/2} T^{n-j-1},
\end{equation*}
which is satisfactory since $\delta_0(j,k) = \delta_{n-7,n-6}(j) = 0$.
Finally, if $j=k=0$, then $B_{j,k}=2$ and
\[
f_{0,0}(h) = \frac{h-2}{n-8}-2-\delta_{n-5,n-4}(h)\le -1
\qquad(2\le h\le n-4).
\]
Indeed, $f_{0,0}(h)$ is maximised either at $h=n-6$, where $f_{0,0}(n-6) = -1$,
or at $h=n-4$, where $f_{0,0}(n-4) = (n-6)/(n-8) - 3 \le -1$
since $n\ge 10$.
Thus $a_p,b_p\le -1$ and
\begin{equation*}
\mathcal L \ll_{F,\ve} (DH)^\ve D^{(n+j+3)/2} T^{n-j} K_j^{-\alpha_n}
\sum_{d_1d_2\le T} (d_1d_2)^{-1}
\ll_{F,\ve} (DH)^{2\ve} D^{(n+j+3)/2} T^{n-j} K_j^{-\alpha_n},
\end{equation*}
which is satisfactory since $\delta_0(j,k) = 1$ and $\delta_{n-7,n-6}(j) = 0$.
\end{proof}

We will also need the following complementary estimate for $R_{j+2,k+2}(d,\bfv)\ge T$.
\begin{lemma}\label{lem:weight-comp}
Let notation be as in \eqref{n10:average-box}.
Suppose $0\le j\le n-6$ and $k\in \{j,j+1\}$.
Then
\begin{equation}\label{n10:weighted-mean-comp}
\sum_{\substack{
d\sim D\\d\ \mathrm{cube\text{-}free}\\
\gcd(d,m)=1}}
\sum_{\substack{
\bfv\in(P_j\cap Q_k)\cap\mathcal V\\
R_{j+2,k+2}(d,\bfv)\ge T}}
|S_d(\bfv)|\ll_{F,\ve}
(DH)^\ve
\sum_{\ell=j+2}^{n-4}
T^{n-\ell+1-\delta_{n-5,n-4}(\ell)}
D^{\frac{n+\ell+1}{2}}.
\end{equation}
\end{lemma}

\begin{proof}
Let $\mathcal L$ be the left-hand side of \eqref{n10:weighted-mean-comp}.
As in the proof of Lemma~\ref{n10:weighted-conductor}, write
\[
r=R_{j+2,k+2}(d,\bfv),\qquad c=\gcd(d,r^2),\qquad b=d/c.
\]
Then $\rad(c)=r\ge T$, $\gcd(b,c)=1$ and
$R_{j+2,k+2}(b,\bfv)=1$.
For $j+2\le h\le n-4$, let $c_h$ be the product of all prime powers
$p^\nu\parallel c$ of depth $h$, where the depth of a prime power is defined to be
$e(\bfv;p)$ if $\nu=1$
and $f(\bfv;p)$ if $\nu=2$.
Put $r_h=\rad(c_h)$.
Then
\[
c=\prod_hc_h,\qquad r=\prod_hr_h\ge T,\qquad
r_h\le c_h\le r_h^2,\qquad c\le2D.
\]
If $c=d_1d_2^2$, then \eqref{EQ:K-relaxed-hybrid-pointwise} gives
$$\frac{S_c(\bfv)}{c^{(n+j+3)/2}}
\ll (DH)^\ve d_2^{k-j} K(d_1d_2^2,\bfv)
\le (DH)^\ve \prod_{h=j+2}^{n-4}c_h^{(h-j-2)/2},$$
by the definition of $K(d_1d_2^2,\bfv)$ in \eqref{n10:K},
since $k\ge j$.
Yet for fixed $c,\bfv$,
we have
$$\sum_{\substack{\mathrm{cube\text{-}free}\ b\sim D/c \\ R_{j+2,k+2}(b,\bfv)=1}}
|S_b(\bfv)| \ll (DH)^\ve (D/c)^{(n+j+3)/2},
$$
by \eqref{eq:fix-cv}.
Therefore,
\begin{equation*}
\mathcal L
\ll \sum_{c_h,r_h}
\sum_{\substack{
\bfv\in(P_j\cap Q_k)\cap\mathcal V\\
\bfv\bmod{r_h} \in Z_h(\ZZ/r_h\ZZ)}}
(DH)^\ve D^{(n+j+3)/2} \prod_{h=j+2}^{n-4}c_h^{(h-j-2)/2}.
\end{equation*}
Let $\mathcal L'$ be the contribution from $c_h\sim C_h$ and $r_h\sim R_h$.
Then $\mathcal L' = 0$ unless
\[
\prod_h C_h\ll D,\qquad
\prod_h R_h\gg T,\qquad R_h\ll C_h\ll R_h^2.
\]
For fixed $r_h$, there are at most
$2^{\omega(r_h)}\ll_\ve D^\ve$ possible choices for $c_h$.
Hence
\begin{equation*}
\mathcal L'
\ll \mathcal N(\bm R)
(DH)^\ve D^{(n+j+3)/2} \prod_{h=j+2}^{n-4} C_h^{(h-j-2)/2},
\end{equation*}
where $\mathcal N(\bm R)$ is the quantity in
Proposition~\ref{lem:stratified-composite-sieve}
with $Y_0=\A_\ZZ^n$,
$\mathcal U=(P_j\cap Q_k)\cap\ZZ^n$,
$Y_i = Z_{j+1+i}$ for $1\le i\le s=n-5-j$,
and with dyadic parameters $R_{j+1+i}$ in place of $R_i$.
Indeed, the set $\mathcal U$
lies outside all the varieties $Z_h$
with $h\ge j+2$.
In fact, $P_j \cap Z_{j+2} = \emptyset$ by \eqref{n10:rational-partition},
since $U_{j+1} \cap Z_{j+2} = \emptyset$ by \eqref{n10:promotion-opens}.

By \eqref{n10:progression-counts},
the hypothesis \eqref{eq:stratified-progression-hypothesis}
of Proposition~\ref{lem:stratified-composite-sieve}
certainly holds with $\alpha = n-j$.
Moreover, \eqref{eq:silver} yields $\dim Z_h \le f(h)$, where $f(h) := n-h-\delta_{n-5,n-4}(h)$.
Since $f$ is (strictly) decreasing on the interval $j+2\le h\le n-4$,
Proposition~\ref{lem:stratified-composite-sieve} implies that
\[
\mathcal N(\bm R) \ll_{F,\ve}(DH)^\ve\Biggl(
1+\frac{T^{n-j}}{\prod_{h=j+2}^{n-4}R_h^{n-j-f(h)-1}}
+\sum_{\ell=j+2}^{n-4}
\frac{T^{f(\ell)+1}}
 {\prod_{h>\ell}R_h^{f(\ell)-f(h)}}
\Biggr).
\]
Let $A_0$ be the middle term in the parentheses,
and let $A_\ell$ be the summand indexed by $\ell$.
Then
\begin{equation*}
\mathcal L'
\ll A (DH)^\ve D^{(n+j+3)/2} \prod_{h=j+2}^{n-4} C_h^{(h-j-2)/2},
\end{equation*}
for some $A\in \{1, A_0, A_\ell\}$.
If $A=1$, then using $\prod_h C_h\ll D$ we find that
\begin{equation*}
\mathcal L'
\ll (DH)^\ve D^{(n+j+3)/2} D^{(n-4-j-2)/2}
= (DH)^\ve D^{n-3/2}
\ll (DH)^\ve T^4 D^{n-3/2},
\end{equation*}
which is the $\ell=n-4$ term of \eqref{n10:weighted-mean-comp}.
If $A=A_0$, then using $C_h\ll R_h^2$ we find that
\begin{equation*}
\mathcal L'
\ll (DH)^\ve D^{(n+j+3)/2} T^{n-j} \prod_{h=j+2}^{n-4}
\frac{R_h^{h-j-2}}{R_h^{n-j-f(h)-1}}.
\end{equation*}
The product over $h$ equals
$\prod_{h=j+2}^{n-4} R_h^{-1-\delta_{n-5,n-4}(h)}
\ll T^{-1-\delta_{n-5,n-4}(j+2)}$,
since $\prod_h R_h\gg T$
and $\delta_{n-5,n-4}(h) \ge \delta_{n-5,n-4}(j+2)$ for all $j+2\le h\le n-4$.
Thus
\begin{equation*}
\mathcal L'
\ll (DH)^\ve D^{(n+j+3)/2} T^{n-j-1-\delta_{n-5,n-4}(j+2)},
\end{equation*}
which is the $\ell=j+2$ term of \eqref{n10:weighted-mean-comp}.
Finally, suppose $A=A_\ell$, where $j+2\le\ell\le n-4$.
Then
\begin{equation*}
\mathcal L'
\ll (DH)^\ve D^{(n+j+3)/2} T^{f(\ell)+1}
\prod_{h=j+2}^{\ell} C_h^{(h-j-2)/2}
\prod_{h=\ell+1}^{n-4} \frac{C_h^{(h-j-2)/2}}{R_h^{f(\ell)-f(h)}}.
\end{equation*}
Here $C_h^{(h-j-2)/2} \ll C_h^{(\ell-j-2)/2}$ for $h\le \ell$.
Also, $R_h^{f(\ell)-f(h)} \gg R_h^{h-\ell} \gg C_h^{(h-\ell)/2}$ for $\ell < h\le n-4$,
since $f$ is strictly decreasing on this interval.
It follows that
\begin{equation*}
\mathcal L'
\ll (DH)^\ve D^{(n+j+3)/2} T^{f(\ell)+1}
\prod_{h=j+2}^{n-4} C_h^{(\ell-j-2)/2}
\ll (DH)^\ve D^{(n+\ell+1)/2} T^{f(\ell)+1}
\end{equation*}
since $\prod_h C_h\ll D$.
This matches the $\ell$ term of \eqref{n10:weighted-mean-comp}, and so we are done.
\end{proof}

For the remainder of this section put $n=10$.
We conclude this section with the following result
which will be useful for combining the cube-free average with the cubefull one.

\begin{prop}\label{prop:n10-weighted-cube-free}
Let $D\ge 1/2$, $L\ge1$, $m\in \ZZ_{\ge 1}$, and $\bfv_0\in\RR^{10}$.
Let $\mathcal V$ be any subset of
$\{\bfv\in\ZZ^{10}\setminus\{\bm0\}:|\bfv-\bfv_0|\le L\}$.
Suppose $w:\mathcal V\to\RR_{\ge0}$ and for every $\bfa\bmod m$ define
\[
P(\bfa):=\sup_{\substack{\bfv\in \mathcal V \\ \bfv\equiv \bfa \bmod m}} |w(\bfv)|,
\]
with the convention that $\sup_{\emptyset} = 0$
if the set in the supremum is empty.
Put
\[
T=1+L/m,\qquad H=2+|\bfv_0|+L+m.
\]
Let
\[
\|P\|_1=\sum_{\bfa\bmod m}P(\bfa),
\qquad \|w\|_1=\sum_{\bfv\in\mathcal V}w(\bfv).
\]
Then
\begin{equation}\label{eq:n10-weighted-cube-free}
\sum_{\substack{d\sim D\ \mathrm{cube\text{-}free}\\\gcd(d,m)=1}}
\sum_{\bfv\in\mathcal V}w(\bfv)|S_d(\bfv)|
\ll_{F,\ve}(DH)^\ve
\left(\mathcal R
+D^{13/2}T^{20/3}\|P\|_1^{2/3}\|w\|_1^{1/3}\right),
\end{equation}
where $\mathcal R = \min \left\{(D^{19/2}T+D^{41/6}T^9)\|P\|_1,\ D^9\|w\|_1\right\}$.
\end{prop}

\begin{proof}
Put $x=T/D^{1/3}$ and let 
$$\mathcal U_d := \{\bfv\in\mathcal V\cap P_0\cap Q_0:
R_{2,2}(d,\bfv)\le T\}.$$
Observe that
\begin{equation*}
\begin{split}
\sum_{\substack{
d\sim D\\
d\ \text{cube-free}\\
\gcd(d,m)=1}}
\sum_{\bfv\in\mathcal V}
w(\bfv)|S_d(\bfv)|
&=
\sum_{\substack{
d\sim D\\
d\ \text{cube-free}\\
\gcd(d,m)=1}}
\sum_{\bfv\in \mathcal U_d}
w(\bfv)|S_d(\bfv)|
+
\sum_{\substack{
d\sim D\\
d\ \text{cube-free}\\
\gcd(d,m)=1}}
\sum_{\bfv\in \mathcal V \setminus \mathcal U_d}
w(\bfv)|S_d(\bfv)|\\
&=M+E,
\end{split}
\end{equation*}
say.
We will prove that $E \ll (DH)^\ve \mathcal R$
and that $M \ll (DH)^\ve D^{13/2}T^{20/3}\|P\|_1^{2/3}\|w\|_1^{1/3}$.

We provide two different bounds for $E$.
First we sort the vector $\bfv$ according to its reduction modulo $m$.
If $\bfv\equiv \bfa\bmod{m}$, then $w(\bfv) \le P(\bfa)$.
Thus
\begin{equation*}
E=\sum_{\substack{\bfa\bmod{m}}}\sum_{\substack{
d\sim D\\
d\ \text{cube-free}\\
\gcd(d,m)=1}}
\sum_{\substack{
\bfv\in \mathcal V \setminus \mathcal U_d\\
\bfv\equiv\bfa\bmod{m}}} w(\bfv)|S_d(\bfv)|
\le \|P\|_{1}\max_{\substack{\bfa\bmod{m}}} f(\bfa),
\end{equation*}
where
\[
f(\bfa) = \sum_{\substack{
d\sim D\\
d\ \text{cube-free}\\
\gcd(d,m)=1}}
\sum_{\substack{
\bfv\in \mathcal V \setminus \mathcal U_d\\
\bfv\equiv\bfa\bmod{m}}} |S_{d}(\bfv)|.
\]
At this stage we recall from the 
start of \S\ref{sec:silver} that 
the set $\{P_j\cap Q_k\}_{j,k=0}^{n-4}$
gives a partition of $\QQ^n$.
After an application of
Lemma~\ref{LEM:n5divisorbound} for $(j,k)=(n-5,n-5)$,
Lemma~\ref{n10:weighted-conductor} for $(j,k)\ne (0,0),(n-5,n-5)$,
and Lemma~\ref{lem:weight-comp} for $(j,k)\ne (n-5,n-5)$,
one gets
\[
\begin{split}
f(\bfa) &\ll (DH)^\ve
\left(
TD^9+T^9D^{13/2}+T^8D^7
+T^7D^{15/2}+T^5D^8+T^4D^{17/2}
\right)\\&\ll(DH)^\ve\left(xD^{28/3}+x^{9}D^{19/2}+x^{8}D^{29/3}+x^{7}D^{59/6}+x^{5}D^{29/3}+x^{4}D^{59/6}\right).
\end{split}
\]
Since $D^{59/6}(x^{4}+x^{7})\leq 2D^{59/6}(x+x^{9})$,
it follows that
\[
\begin{split}
f(\bfa) &\ll  (DH)^\ve\left(xD^{28/3}+x^{9}D^{19/2}+x^{8}D^{29/3}+x^{5}D^{29/3}+(x+x^{9})D^{59/6}\right)\\&\ll (DH)^\ve D^{59/6}(x+x^{9})
= (DH)^\ve (D^{19/2}T+D^{41/6}T^9),
\end{split}
\]
since $D\gg 1$
and $ x=T/D^{1/3}$.
Second,
thanks to the $h=\ell=n-3$ case of \eqref{n10:residual-fixed-point}
we know that for every $\bfv\in P_{j}\cap Q_{k}$, for $j,k\in\{0,\dots,n-5\}$ we have 
\[
\sideset{}{^*}\sum_{d_1\sim D_1}
\sideset{}{^*}\sum_{\substack{
d_2\sim D_2\\\gcd(d_1,d_2)=1}}
|S_{d_1d_2^2}(\bfv)|\ll_{F,\ve} (DH)^\ve
D^{n-3/2} D_1^{\max\{j-n+6,0\}/2}
\ll (DH)^\ve D^{n-1},
\]
since $j\le n-5$.
Indeed, for $j=k=n-5$
this was already observed in the proof of Lemma~\ref{LEM:n5divisorbound}.
Hence
\[
E\le \sum_{\substack{
\bfv\in\mathcal V}} w(\bfv)\sum_{\substack{
d\sim D\\
d\ \text{cube-free}\\
\gcd(d,m)=1}}
|S_{d}(\bfv)|
\ll \|w\|_{1} (DH)^\ve D^9.
\]
Combining the two bounds, we conclude that $E\ll (DH)^\ve \mathcal R$.

It remains to bound
\[
M = \sum_{\substack{
d\sim D\\
d\ \text{cube-free}\\
\gcd(d,m)=1}}
\sum_{\substack{
\bfv\in\mathcal V\cap P_0\cap Q_0 \\ R_{2,2}(d,\bfv)\leq T}}w(\bfv)|S_{d}(\bfv)|.
\]
Let $M(K_0)$ be the contribution to $M$ from $K(d,\bfv) \sim K_0$.
We bound $M(K_0)$ using the same two methods as for $E$.
First, the $(j,k)=(0,0)$ case of Lemma~\ref{n10:weighted-conductor} implies that
\begin{equation*}
M(K_0) \ll \|P\|_{1}
(DH)^\ve T^n D^{\frac{n+3}{2}} K_0^{-1/2}.
\end{equation*}
Second, by \eqref{n10:weighted-fixed-point} for $(j,k)=(0,0)$ we have
\begin{equation*}
M(K_0) \ll
\|w\|_{1} (DH)^\ve D^{\frac{n+3}{2}} K_0.
\end{equation*}
Thus
\begin{equation*}
M(K_0)
= M(K_0)^{2/3} M(K_0)^{1/3}
\ll (DH)^\ve D^{\frac{n+3}{2}}
\left(\|P\|_{1} T^n K_0^{-1/2}\right)^{2/3}
\left(\|w\|_{1} K_0\right)^{1/3}.
\end{equation*}
The powers of $K_0$ cancel out,
leading to $M(K_0) \ll (DH)^\ve D^{(n+3)/2}T^{2n/3}\|P\|_1^{2/3}\|w\|_1^{1/3}$.
Dyadic summation over the interval $1/2 \le K_0 \ll D^{O(1)}$ completes the proof.
\end{proof}

\section{Mixed modulus averages}
\label{sec:n10-mixed-moduli}

Throughout this section $n=10$ and $X(\QQ)=\emptyset$.
Recall the definition \eqref{eq:inel-global-complete-sum} of 
$S_{q,W}(\bfv)$. 
We will combine Lemmas~\ref{Le: Cubefull.finalbound},
\ref{lem:q-smith-cubefull},
\ref{lem:q-fixed-shifted box}
and Proposition~\ref{prop:n10-weighted-cube-free} to prove the following result about the average of 
$|S_{q,W}(\bfv)|$ over $\bfv\in \ZZ^{10}$ and  arbitrary integer moduli $q\ge 1$.

\begin{prop}\label{prop:n10-clean-shifted}
Fix $0<\gamma<1/96$.
There exists a modulus $W$ as in \eqref{large-fixed-W} such that
for every $\varepsilon>0$, $\bfv_0\in \RR^n$ and $Q,L\ge 1/2$, we have
\begin{equation}\label{eq:n10-clean-shifted}
\sum_{q\sim Q}
\sum_{\substack{\bm0\ne\bfv\in\ZZ^{10}\\
|\bfv-\bfv_0|\le L}}
|S_{q,W}(\bfv)|
\ll_{F,\gamma,\ve}
(LQ+|\bfv_0|)^\ve (L+Q^{1/3})^{10}Q^{20/3-\gamma}.
\end{equation}
Here $W$ may depend on $F$ and $\gamma$, but not on $\bfv_0,L,Q,\ve$.
\end{prop}

We start by recalling some notation.
For a fixed $\eta>0$, let $\mathcal P_\eta$
be a finite, sufficiently large set of primes
depending only on $F$ and $\eta$.
We then define $W_\eta = \prod_{p\in \mathcal P_\eta} p^{M_p}$ as in \eqref{large-fixed-W}.

Every positive integer $q$ has a unique decomposition
$q=gdr$, where $g$ is supported on the primes in
$\mathcal P_\eta$, $d$ is cube-free, $r$ is cubefull,
both $d$ and $r$ are supported on primes outside $\mathcal P_\eta$,
and $g,d,r$ are pairwise coprime.
By 
\eqref{eq:q-exact-CRT} and 
Lemma~\ref{lem:restricted-multiplicative},
\[
S_{q,W_\eta}(\bfv)
=
S_{g,W_\eta}(\bfv)S_d(\bfv)S_r(\bfv).
\]

Let us first focus on $S_{g,W_\eta}(\bfv)$.
The estimate in Lemma~\ref{lem:q-fixed-shifted box} gives, for every
$p\in\mathcal P_\eta$ and every $k\geq0$,
\begin{equation*}
\sum_{\bfa\bmod{p^{\lceil k/3\rceil}}}
\max_{\substack{\bfv\bmod{p^{\max(k,M_p)}} \\ \bfv\equiv \bfa\bmod{p^{\lceil k/3\rceil}}}}
|S_{p^k,p^{M_p}}(\bfv)|
\ll_{F,\eta}
p^{k(10+(4-r(F))/6)}.
\end{equation*}
On the other hand,
Lemma~\ref{lem:inel-high-rank-point} implies that
$r(F)\geq8$ when $n=10$. Hence,
\begin{equation*}
\sum_{\bfa\bmod{p^{\lceil k/3\rceil}}}
\max_{\substack{\bfv\bmod{p^{\max(k,M_p)}} \\ \bfv\equiv \bfa\bmod{p^{\lceil k/3\rceil}}}}
|S_{p^k,p^{M_p}}(\bfv)|
\ll_{F,\eta}
p^{28k/3}.
\end{equation*}
By Lemma~\ref{lem:restricted-multiplicative} and the Chinese remainder theorem
over all primes $p\in \mathcal P_\eta$,
we obtain
\begin{equation}
\label{eq:q-fixed-g-rec}
\sum_{\bfa\bmod{A_g}}
\max_{\substack{\bfv\bmod{\lcm(g,W_\eta)} \\ \bfv\equiv \bfa\bmod{A_g}}}
|S_{g,W_\eta}(\bfv)|
\ll_{F,\eta}
g^{28/3},
\end{equation}
where
\[
A_g
=
\prod_{p\mid g}p^{\lceil v_p(g)/3\rceil}
\asymp_{F,\eta}g^{1/3}.
\]

Let us now focus on the cubefull part $r$ and its
interaction with $g$.
Following Lemma~\ref{lem:q-smith-cubefull}, put
\[
A=bz_2,\qquad T=bz_1
\]
where
\begin{equation}
\label{eq:q-buw-rec}
b = \prod_{v_p(r)>0}p^{\lfloor v_p(r)/3\rfloor},
\qquad
z_i =
\prod_{\substack{v_p(r)>0\\
v_p(r)\equiv i\bmod3}}p,
\end{equation}
for $0\leq i\leq2$. 
In particular, $r=b^3z_1z_2^2$ and $z_0z_1z_2\mid b$.
By \eqref{eq:q-fixed-g-rec},
Lemma~\ref{lem:q-smith-cubefull},
and the Chinese remainder theorem,
\begin{equation}
\label{eq:Pgr}
\sum_{\bfa\bmod{A_gA}}
\max_{\substack{\bfv\bmod{\lcm(g,W_\eta)r} \\ \bfv\equiv \bfa\bmod{A_gA}}}
|S_{g,W_\eta}(\bfv)S_r(\bfv)|
\ll_{F,\eta}
g^{28/3} r^{10+\eta}\kappa(r),
\end{equation}
where $\kappa(r)=z_0^{-2}z_1^{-4}z_2^{-1}$.
The following lemma will help us estimate sums over $r$.

\begin{lemma}\label{lem:n10-cubefull-weights}
For $0\le\alpha\le10$ and every $\varepsilon>0$,
\[
\sum_{\substack{r\sim C\\r\ \mathrm{cube\text{-}full}}}
\kappa(r)\, (z_1/z_2)^{\alpha/3}
\ll_{\alpha,\varepsilon}
C^{\max(0,(\alpha-9)/12)+\varepsilon}.
\]
\end{lemma}

\begin{proof}
By Rankin's trick, the left-hand side divided by the right-hand side is
\begin{equation*}
\ll \sum_{r\ \mathrm{cube\text{-}full}}
\kappa(r)\, (z_1/z_2)^{\alpha/3} r^{-s}
\le \prod_p E_p,
\end{equation*}
where $s=\max(0,(\alpha-9)/12)+\varepsilon$ and
\[
E_p = 1+\sum_{e\ge1}p^{-3es-2}
+\sum_{e\ge1}p^{-(3e+1)s-4+\alpha/3}
+\sum_{e\ge1}p^{-(3e+2)s-1-\alpha/3}.
\]
The three sums are respectively
$\ll p^{-3s-2}$,
$\ll p^{-4s-4+\alpha/3}$,
and $\ll p^{-5s-1-\alpha/3}$,
whose exponents are all strictly less than $-1$.
Indeed, this is clear for the first and third exponents,
since $s>0$ and $\alpha\ge 0$.
For the second exponent, we have $-4s-4+\alpha/3 < -1$
because $s > (\alpha-9)/12$.
Thus the Euler product converges.
\end{proof}

\begin{proof}[Proof of Proposition~\ref{prop:n10-clean-shifted}]
Fix $0<\gamma<1/96$, and choose $\eta>0$ sufficiently
small in terms of $1/96-\gamma$, to be fixed at the end.
Let $\mathcal P_\eta$ and $W_\eta$
be chosen in terms of $\eta$ as above.

We first assume that $L\leq Q^{1/3}$ and put
\[
\mathcal V
=
\{\bfv\in\ZZ^{10}\setminus\{\bm0\}:
|\bfv-\bfv_0|\leq L\}.
\]
For $q\sim Q$, write $q=gdr$ as above and fix dyadic ranges
\[
g\sim G,\qquad d\sim D,\qquad r\sim C,
\]
so that $GDC\asymp Q$.
For each fixed $g,r$, define
\[
w_{g,r}(\bfv)
:=
|S_{g,W_\eta}(\bfv)S_r(\bfv)|,
\qquad
m_{g,r}:=A_gA_r,
\qquad
T_{g,r}:=1+\frac{L}{m_{g,r}},
\]
where $A_g$ is as above and $A_r=bz_2$, with $b,z_2$
as in \eqref{eq:q-buw-rec}. Put
\[
H_{g,r}:=2+|\bfv_0|+L+m_{g,r},
\]
and, for every $\bfa\bmod m_{g,r}$, define
\[
P_{g,r}(\bfa)
:=
\sup_{\substack{\bfv\in\mathcal V\\
\bfv\equiv\bfa\bmod m_{g,r}}}
|S_{g,W_\eta}(\bfv)S_r(\bfv)|,
\]
with the convention that the supremum is zero if the
set is empty.
An application of Proposition~\ref{prop:n10-weighted-cube-free}
gives
\begin{equation}
\label{eq:n10-cube-free-applied}
\sum_{\substack{d\sim D\ \mathrm{cube\text{-}free}\\
\gcd(d,m_{g,r})=1}}
\sum_{\bfv\in\mathcal V}
w_{g,r}(\bfv)|S_d(\bfv)|
\ll_{F,\ve} (DH_{g,r})^\ve \left(\min\{\mathcal R_1, \mathcal R_2\} + \mathcal R_3\right)
\end{equation}
where
\begin{equation*}
\begin{split}
\mathcal R_1
&= (D^{19/2}T_{g,r}+D^{41/6}T_{g,r}^{9}) \|P_{g,r}\|_1, \\
\mathcal R_2
&= D^9\|w_{g,r}\|_1, \\
\mathcal R_3
&= D^{13/2}T_{g,r}^{20/3}
\|P_{g,r}\|_1^{2/3}
\|w_{g,r}\|_1^{1/3}.
\end{split}
\end{equation*}
Here $\rad(m_{g,r})=\rad(gr)$, so that
$\gcd(d,m_{g,r})=1$ is equivalent to $\gcd(d,gr)=1$.

In order to further average the estimate \eqref{eq:n10-cube-free-applied}
over $g$ and $r$,
our strategy will be to bound
\begin{equation}
\label{EQ:mixed-endpoint-task-list}
\sum_{g\sim G}\sum_{r\sim C}\|w_{g,r}\|_1
\quad\text{and}\quad
\sum_{g\sim G}\sum_{r\sim C}
T_{g,r}^{\alpha}\|P_{g,r}\|_1
\qquad(\alpha=1,9,10)
\end{equation}
first, before attacking $\min\{\mathcal R_1,\mathcal R_2\}$
and $\mathcal R_3$ by interpolation.
We start with the average of $\|w_{g,r}\|_1$.
It will be convenient to handle the $g$-aspect using a pointwise bound.
By \eqref{fixed-p-adic-stationary-phase}, together with
Lemma~\ref{lem:inel-high-rank-point} and
Lemma~\ref{lem:restricted-multiplicative}, we have
$S_{g,W_\eta}(\bfv)\ll_{F,\eta}g^7$
uniformly in $\bfv$.
Moreover, since $\mathcal P_\eta$ is fixed and finite,
there are $O_{F,\eta,\ve}(G^\ve)$ possible integers
$g\sim G$. Hence, for every $\ve>0$,
\[
\begin{split}
\sum_{g\sim G}\sum_{r\sim C}\|w_{g,r}\|_1
&=
\sum_{g\sim G}\sum_{r\sim C}
\sum_{\bfv\in\mathcal V}
|S_{g,W_\eta}(\bfv)S_r(\bfv)|\\
&\ll_{F,\eta,\ve}
G^{7+\ve}
\sum_{r\sim C}\sum_{\bfv\in\mathcal V}
|S_r(\bfv)|\\
&\ll_{F,\eta,\ve}
G^{7+\ve}C^{32/5+\ve}(L+C^{1/3})^{10}\\
&\ll_{F,\eta,\ve}
G^{7+\ve}C^{32/5+\ve}Q^{10/3}.
\end{split}
\]
Here the penultimate step follows from 
 \eqref{eq:cubefull-average-sharp} in Lemma~\ref{Le: Cubefull.finalbound}, and the
last step uses $L\leq Q^{1/3}$ and $C\ll Q$.

Let us focus on the other term in \eqref{EQ:mixed-endpoint-task-list},
involving the moments of $\|P_{g,r}\|_1$.
All implied constants below may depend on
$F,\eta,\ve$.
Observe that \eqref{eq:Pgr} gives
\[
\|P_{g,r}\|_1
\ll
G^{28/3}r^{10+\eta}\kappa(r).
\]
Moreover, $r=b^3z_1z_2^2$ by \eqref{eq:q-buw-rec}, so we have
\[
A_g\asymp_{F,\eta}g^{1/3},
\qquad
A_r^3=b^3z_2^3 = r\frac{z_2}{z_1}=\frac{r}{u(r)^3},
\qquad
u(r):=\left(\frac{z_1}{z_2}\right)^{1/3}.
\]
Since $L\leq Q^{1/3}$ and $GDC\asymp Q$, we obtain
\[
T_{g,r}
\ll
1+\left(\frac{Q}{gr}\right)^{1/3}u(r)
\ll
1+D^{1/3}u(r).
\]
Lemma~\ref{lem:n10-cubefull-weights}, applied with
exponents $0$ and $\alpha$, gives, for
$\alpha\in\{1,9,10\}$,
\[
\begin{split}
\sum_{g\sim G}\sum_{r\sim C}
T_{g,r}^{\alpha}\|P_{g,r}\|_1
&\ll
G^{28/3+\ve}C^{10+\eta}
\left(
\sum_{r\sim C}\kappa(r)
+
D^{\alpha/3}
\sum_{r\sim C}\kappa(r)u(r)^\alpha
\right)\\
&\ll
D^{\alpha/3}G^{28/3+\ve}
C^{10+\max\{0,(\alpha-9)/12\}+\eta+\ve}\\
&\ll
Q^\ve D^{\alpha/3}G^{28/3}
C^{10+\max\{0,(\alpha-9)/12\}+\eta}.
\end{split}
\]
Here we used $D,C\ge 1/2$ to absorb the contribution
of the exponent $0$, and $GC\ll Q$ in the last step.
Likewise, the preceding estimate for the weight norms
gives
\[
\sum_{g\sim G}\sum_{r\sim C}\|w_{g,r}\|_1
\ll
Q^\ve G^7C^{32/5}Q^{10/3}.
\]

We are now ready to sum the bound in
\eqref{eq:n10-cube-free-applied}.
First, the moment estimates with $\alpha=1,9$ give
\[
\begin{split}
\sum_{g\sim G}\sum_{r\sim C} \mathcal R_1
&=\sum_{g\sim G}\sum_{r\sim C}
\bigl(D^{19/2}T_{g,r}
+D^{41/6}T_{g,r}^{9}\bigr)
\|P_{g,r}\|_1\\
&\ll
Q^\ve
\bigl(D^{19/2+1/3}+D^{41/6+3}\bigr)
G^{28/3}C^{10+\eta}\\
&\ll
Q^\ve D^{59/6}G^{28/3}C^{10+\eta}\\
&\asymp
Q^{10+\ve}C^\eta D^{-1/6}G^{-2/3},
\end{split}
\]
since $GDC\asymp Q$.
On the other hand, the estimate for $\|w_{g,r}\|_1$ gives
\[
\begin{split}
\sum_{g\sim G}\sum_{r\sim C} \mathcal R_2
&\ll
Q^\ve D^9G^7C^{32/5}Q^{10/3}\\
&\ll
Q^\ve D^9G^7C^{13/2}Q^{10/3}\\
&\asymp
Q^{10-1/6+\ve}G^{1/2}D^{5/2}.
\end{split}
\]
Since the sum of the minima $\min\{\mathcal R_1,\mathcal R_2\}$
is at most the minimum
of the sums, we obtain
\[
\begin{split}
\sum_{g\sim G}\sum_{r\sim C}
\min\{\mathcal R_1,\mathcal R_2\}
&\ll
\bigl(Q^{10+\ve}C^\eta D^{-1/6}G^{-2/3}\bigr)^{15/16}
\bigl(Q^{10-1/6+\ve}G^{1/2}D^{5/2}\bigr)^{1/16}\\
&=
Q^{10-1/96+\ve}G^{-19/32}C^{15\eta/16}.
\end{split}
\]

It remains to bound the last term in
\eqref{eq:n10-cube-free-applied}.
By H\"older's inequality,
\[
\begin{split}
\sum_{g\sim G}\sum_{r\sim C} \mathcal R_3
&= D^{13/2}
\sum_{g\sim G}\sum_{r\sim C}
T_{g,r}^{20/3}
\|P_{g,r}\|_1^{2/3}
\|w_{g,r}\|_1^{1/3}\\
&\leq
D^{13/2}
\left(
\sum_{g\sim G}\sum_{r\sim C}
T_{g,r}^{10}\|P_{g,r}\|_1
\right)^{2/3}
\left(
\sum_{g\sim G}\sum_{r\sim C}
\|w_{g,r}\|_1
\right)^{1/3}\\
&\ll
Q^\ve D^{13/2}
\bigl(D^{10/3}G^{28/3}C^{10+1/12+\eta}\bigr)^{2/3}
\bigl(G^7C^{32/5}Q^{10/3}\bigr)^{1/3}\\
&\asymp
Q^{10-1/30+\ve}
G^{-3/10}D^{-2/15}C^{2\eta/3}.
\end{split}
\]
Here the last step uses $GDC\asymp Q$.

Since $G,D\geq1/2$, $C\ll Q$, and $1/30>1/96$,
both contributions are
\[
\ll Q^{10-1/96+\eta+\ve}.
\]
Moreover, $m_{g,r}\mid gr$, so
$m_{g,r}\ll Q$ and $H_{g,r} \ll Q + |\bfv_0|$.
Restoring the factors $(DH_{g,r})^\ve$ from
\eqref{eq:n10-cube-free-applied}, summing over the
$O((1+\log Q)^3)$ dyadic blocks, and rescaling $\ve$,
we therefore obtain
\[
\sum_{q\sim Q}
\sum_{\substack{\bm0\ne\bfv\in\ZZ^{10}\\
|\bfv-\bfv_0|\leq L}}
|S_{q,W_\eta}(\bfv)|
\ll_{F,\eta,\ve}
(LQ+|\bfv_0|)^\ve Q^{10-1/96+\eta}
\qquad
(L\leq Q^{1/3}).
\]
We now fix $0<\eta<1/96-\gamma$, depending only on
$\gamma$, and set $W=W_\eta$.
We conclude that
\[
\sum_{q\sim Q}
\sum_{\substack{\bm0\ne\bfv\in\ZZ^{10}\\
|\bfv-\bfv_0|\leq L}}
|S_{q,W}(\bfv)|
\ll_{F,\gamma,\ve}
(LQ+|\bfv_0|)^\ve Q^{10-\gamma}
\qquad
(L\leq Q^{1/3}).
\]

If instead $L>Q^{1/3}$, cover the original box by
\[
O\left(\left(1+\frac{L}{Q^{1/3}}\right)^{10}\right)
\]
boxes of radius $Q^{1/3}$, with centres of size
$O(|\bfv_0|+L)$. Applying the preceding bound to each box,
and rescaling $\ve$ if necessary, gives
\[
\begin{split}
\sum_{q\sim Q}
\sum_{\substack{\bm0\ne\bfv\in\ZZ^{10}\\
|\bfv-\bfv_0|\leq L}}
|S_{q,W}(\bfv)|
&\ll_{F,\gamma,\ve}
(LQ+|\bfv_0|)^\ve
\left(1+\frac{L}{Q^{1/3}}\right)^{10}
Q^{10-\gamma}\\
&=
(LQ+|\bfv_0|)^\ve
(L+Q^{1/3})^{10}Q^{20/3-\gamma}.
\end{split}
\]
Together with the case $L\leq Q^{1/3}$, this proves
Proposition~\ref{prop:n10-clean-shifted}.
\end{proof}

\section{The circle method}
\label{sec:intro-circle}

Throughout this section we shall assume that $n\geq 10$.
Our basic approach is to assume for a contradiction that 
$X(\QQ)=\emptyset$.
Under this assumption our goal will be to show that 
$$
\#\{\x\in \ZZ^n: |\x|\leq P, ~F(\x)=0\}\gg P^{n-3},
$$
for all $P\gg 1$,  where the implied constants are allowed to depend  on $F$.
This is a contradiction. 
We shall attempt to establish this lower bound for the counting function for $n=10$.

Since $F$ is a square-free cubic form,
there exists a point $\bm{x}_0\in \RR^n$ such that $F(\bm{x}_0)=0$ and $\nabla F(\bm{x}_0)\neq \bm{0}$.
Furthermore, by Lemma~\ref{LEM:my-guy-slop} we can assume that
$\rank(H(\bm{x}_0)) = r(F) \ge 2n/3$.
After a fixed $\QQ$-linear change of variables,
we may assume that $\frac{\partial F}{\partial x_1}(\bm{x}_0)\neq 0$
and that the leading principal $r(F) \times r(F)$ minor of $H(\bm{x}_0)$ is nonzero.
Let
\begin{equation}\label{eq:Psi}
\Psi \colon \RR^n\to \RR^{r(F)},
\qquad \bfy\mapsto
\left(\frac{\partial F(\bfy)}{\partial y_1},\dots,\frac{\partial F(\bfy)}{\partial y_{r(F)}}\right).
\end{equation}
By the inverse function theorem, there exists an open ball $U$ centred at $\bfx_0$ for which
the map $\bfy\mapsto \left(\Psi(\bm{y}),y_{r(F)+1},\dots,y_n\right)$
defines a diffeomorphism from $U$ onto its image,
with Jacobian
\begin{equation}\label{eq:smoothsub''}
\left|\det\left(\frac{\partial^2F}
{\partial y_i\partial y_j}(\bfy)\right)_{1\leq i,j\leq r(F)}\right| \gg 1
\quad \text{for }\bfy\in U.
\end{equation}
Let $\omega\in C_c^\infty(\RR^n)$ be a fixed non-negative smooth weight function supported on  $ [-1,1]^n$, with $\omega(\bm0)>0$. 
Let 
\[
w(\bfy)=\omega((\bm{x}_0-\bfy)/\ell),
\]
where $\ell>0$ is a sufficiently small parameter.
In particular, $w(\bfy)\neq 0$ implies that $|\bfy-\bm{x}_0|\leq \ell$.
In view of Lemma~\ref{lem:Iq'},
the slightly larger box $|\bfy-\bfx_0| \le 2\ell$ will be relevant
when estimating certain oscillatory integrals.
By taking $\ell$ to be sufficiently small, we can assume that 
\begin{equation}\label{eq:smoothsub}
\left|\frac{\partial F}{\partial y_1}(\bfy)\right| \gg 1 \quad
\text{for }|\bfy-\bfx_0| \le 2\ell,
\end{equation}
and that
\begin{equation}
\label{eq:diffU}
\bfy\in U \quad
\text{for }|\bfy-\bfx_0| \le 2\ell.
\end{equation}
All of the estimates in our work are allowed to depend on the coefficients of $F$ and on the small parameter $\ell$ appearing in the definition of $w$. 
Fix $W$ and $\Omega[W]$ as in \eqref{large-fixed-W} and 
\eqref{eq:inel-Omega}, respectively.
For $n=10$ choose the modulus $W$ using Proposition~\ref{prop:n10-clean-shifted},
with a fixed $0<\gamma<1/96$, before choosing $\ve$.
All implicit constants below may also depend on $W$.
With these conventions in place,
our line of attack is to use the circle method to study the counting function
\[
N(P)
= \sum_{\substack{\bfx\in\Z^n\\F(\bfx)=0\\
\bfx\bmod W\in\Omega[W]}}
w(\bfx/P)
\]
that was introduced in \eqref{eq:inel-restricted-count},
as $P\to \infty$.

The precise form of the circle method we use is the following re-interpretation due to Marmon and Vishe~\cite[Proposition~1.2 and Lemma~2.2]{MV19} of Heath-Brown's smooth $\delta$-method \cite{HB}, which in turn is based on pioneering work of Duke, Friedlander and Iwaniec \cite{DFI}.

\begin{prop}\label{p:MV}
Let $Q\geq 1$ and $\eta>0$. Then for all $x\in\ZZ$,
\[
\delta_0(x)=\sum_{q\leq Q}
\sum_{\substack{1\leq a \leq q \\ \gcd(a,q)=1}}
\int_{|\theta|<(qQ)^{-1+\eta}}
p_q(\theta)e((a/q+\theta)x)\dd\theta + O_{N,\eta}\left(Q^{-N\eta}\right),
\]
where
$p_q$ is a smooth function satisfying 
$|p_q(\theta)|\ll 1$
and
\begin{equation}\label{eq:MV-kernel-estimate}
 p_q(\theta)=1+O_N\bigl((1+Q^2|\theta|)^{2N+2}(q/Q)^N\bigr)
 \qquad(\theta\in\RR).
\end{equation}
\end{prop}

It follows from Proposition \ref{p:MV} that 
\begin{equation}\label{Eq:MainCountingFunction}
N(P)=\sum_{q\leq Q}
\sum_{\substack{1\leq a \leq q \\ \gcd(a,q)=1}}
\int_{|\theta|<(qQ)^{-1+\eta}}p_q(\theta)S(a/q+\theta)\dd\theta +O_{N,\eta }(P^nQ^{-N\eta}),
\end{equation}
where 
\[
S(\alpha)=\sum_{\substack{\bm{x}\in \ZZ^n\\\bm{x}\bmod W\in\Omega[W]}}w(\bm{x}/P)e(\alpha F(\bm{x}))
\]
for any $\alpha \in \RR$.  
We shall henceforth take $Q=P^{3/2}$.
In particular, the error term in 
\eqref{Eq:MainCountingFunction} is $O(1)$ on taking $N$ sufficiently large in terms of $\eta$.
It will be convenient to examine the contribution from $q$ and $|\theta|$ restricted to 
dyadic intervals. 
Let $R\geq 1/2$ and let $\phi>0$. We let $N(P;R,\phi)$ be the overall contribution to 
$N(P)$
from $q$ and $\theta$ such that 
$$
R< q\leq 2R \quad \text{ and } \quad \phi<|\theta|\leq 2\phi.
$$
That is, $q\sim R$ and $|\theta|\sim \phi$.
Clearly we have $N(P;R,\phi)=0$ unless $R\ll Q$ and $\phi \ll (RQ)^{-1+\eta}$.

We shall adopt an approach based on Poisson summation to extract an asymptotic formula
for $N(P)$. 
The Poisson summation formula
in residue classes $\bfx \bmod \lcm(q,W)$
yields
\[
\sum_{\substack{1\leq a \leq q \\ \gcd(a,q)=1}}
S(a/q+\theta)
=\frac{1}{\lcm(q,W)^n}
\sum_{\bm{v}\in \ZZ^n}S_{q,W}(\bm{v})I(\theta,\lcm(q,W)^{-1}\bm{v}),
\]
in the integral notation of \eqref{eq:defIq}.
The complete exponential sum $S_{q,W}(\bfv)$
was defined in \eqref{eq:inel-global-complete-sum}.

As usual, the main term will come from the vector $\bm{v}=\bm{0}$, 
and the main task is to show that the vectors $\bm{v}\neq \bm{0}$ make a negligible overall contribution to
$N(P;R,\phi)$. 
Returning to the expression 
\eqref{Eq:MainCountingFunction}, 
and recalling that $|p_q(\theta)|\ll 1$, 
we clearly have 
$$
N(P;R,\phi)=
M(P;R,\phi)+O\bigl(E(P;R,\phi)\bigr),
$$
where
\begin{equation}\label{eq:defMAIN}
M(P;R,\phi)=
\sum_{q\sim R}\frac{S_{q,W}(\bm{0})}{\lcm(q,W)^n}\int_{|\theta|\sim \phi}p_q(\theta)
I(\theta,\bm{0})\dd\theta
\end{equation}
and 
\begin{equation}\label{eq:defERROR2}
E(P;R,\phi)= \sum_{q\sim R} \frac{1}{\lcm(q,W)^n}
\sum_{
\substack{\bm{0}\ne \bm{v}\in \ZZ^n}} |S_{q,W}(\bm{v})|
\int_{|\theta|\sim
\phi } 
|I(\theta,\lcm(q,W)^{-1}\bm{v})|\dd\theta.
\end{equation}
In \eqref{eq:defMAIN} and \eqref{eq:defERROR2},
the restrictions $q\le Q$ and $|\theta|<(qQ)^{-1+\eta}$ are understood;
this is a very minor point,
relevant only at values of $R$ and $\phi$ near the boundary.
Before proceeding, let us quickly dispose of the main term
\begin{equation*}
M(P)
= \sum_{R,\phi} M(P;R,\phi)
= \sum_{q\le Q}\frac{S_{q,W}(\bm{0})}{\lcm(q,W)^n}
\int_{|\theta|<(qQ)^{-1+\eta}} p_q(\theta)
I(\theta,\bm{0})\dd\theta.
\end{equation*}

\begin{lemma}
\label{LEM:pos-main-term}
We have $M(P) \gg P^{n-3}$ for all sufficiently large $P$.
\end{lemma}

\begin{proof}
Let $F_1 = \partial F/\partial x_1$.
Writing $\bfx=P\bfy$ in \eqref{eq:defIq} gives
$I(\theta;\bm0)=P^n J_w(P^3\theta)$, where
\begin{equation*}
J_w(t) = \int_{\RR^n} w(\bfy)e(tF(\bfy))\,\dd\bfy
= \int_{\RR^n} w(\bfy)e(tf)\,\frac{\dd f}{|F_1(\bfy)|} \dd\bfy'
\end{equation*}
by the change of coordinates $\bfy\mapsto (f,\bfy')$
where $f = F(\bfy)$ and $\bfy' = (y_2,\dots,y_n)$,
which is valid by \eqref{eq:smoothsub}.
In particular, $J_w$ is Schwartz in $t$, by integration by parts over $f$.
By Fourier inversion,
\begin{equation*}
\int_{\RR} J_w(t) \,\dd t = \int_{F(\bfy)=0} w(\bfy)\,\frac{\dd\bfy'}{|F_1(\bfy)|} > 0,
\end{equation*}
since $w\ge 0$, $w(\bfx_0)>0$ and $F(\bfx_0) = 0$.

Now write $\theta = t/P^3$ in the definition of $M(P)$ to get
\begin{equation*}
\frac{M(P)}{P^{n-3}}
= \sum_{q\le Q}\frac{S_{q,W}(\bm{0})}{\lcm(q,W)^n}
\int_{|t/P^3|<(qQ)^{-1+\eta}} p_q(t/P^3)
J_w(t)\dd t.
\end{equation*}
Let $B_q = P^3 (qQ)^{-1+\eta} \ge P^3 (qQ)^{-1} = Q/q$.
By the triangle inequality,
\begin{align*}
\left|\int_{|t|<B_q}p_q(t/P^3)J_w(t)\,\dd t-\int_{\RR} J_w(t) \,\dd t\right|
&\le\int_{|t|<B_q}|p_q(t/P^3)-1|\,|J_w(t)|\,\dd t
+\int_{|t|\ge B_q}|J_w(t)|\,\dd t\\
&\ll_w \frac qQ\int_{\RR}(1+|t|)^4|J_w(t)|\,\dd t+B_q^{-1}
\ll_{F,w}\frac qQ,
\end{align*}
where we have applied the $N=1$ case of \eqref{eq:MV-kernel-estimate},
together with the rapid decay of $J_w$.
Consequently, by the triangle inequality
and the absolute convergence of $\mathfrak S_\Omega$
from Lemma~\ref{LEM:localised-singular-series-convergence},
\begin{align*}
\frac{M(P)}{P^{n-3}}-\mathfrak S_\Omega \int_{\RR} J_w(t) \,\dd t
&\ll_{F,w}\sum_{q\le Q}\frac{|S_{q,W}(\bm{0})|}{\lcm(q,W)^n} \frac qQ
+\sum_{q>Q} \frac{|S_{q,W}(\bm{0})|}{\lcm(q,W)^n} \int_{\RR} J_w(t) \,\dd t \\
& \ll Q^{-1/2} \sum_{q\ge 1} \frac{|S_{q,W}(\bm{0})|}{\lcm(q,W)^n}
+ \sum_{q>Q^{1/2}} \frac{|S_{q,W}(\bm{0})|}{\lcm(q,W)^n} = o(1)
\end{align*}
as $P\to \infty$.
Thus $M(P)/P^{n-3} \to \mathfrak S_\Omega \int_{\RR} J_w(t) \,\dd t > 0$,
as $P\to \infty$, where the positivity of $\mathfrak{S}_\Omega$ is a consequence of Lemma~\ref{LEM:localised-singular-series-convergence}.
\end{proof}

We now turn to the error terms $E(P;R,\phi)$.
Let
\begin{equation}  
  \label{eq:V}
  V=RP^{-1} \max\left\{1,\phi P^3\right\}\quad \text{ and }\quad 
  V_0=RP^{-1} \max\left\{1,\sqrt{\phi P^3}\right\}.
  \end{equation}
  Moreover, observe that
$$q\sim R \quad \Longrightarrow \quad \lcm(q,W) \asymp_W R.$$
Take $0<\eta<1/12$ and $P$ sufficiently large; then
for all $R\ll Q$ and $\phi \ll (RQ)^{-1+\eta}$,
we certainly have $\lcm(q,W)\le P^2$ and $|\theta|\le \lcm(q,W)^{-1}P^{-1}$.
It follows directly from
Lemma \ref{lem:Iq'}, with $\ve/3$ in place of $\ve$, that 
$I(\theta,\lcm(q,W)^{-1}\bm{v})\ll_N P^{-N}$ unless 
$
|\bm{v}|\leq P^\ve V.
$
We shall employ the following upper bound for the volume appearing in 
Lemma \ref{lem:Iq'},
with a general modulus $q\asymp_W R$ in place of $\lcm(q,W)$.

\begin{lemma}\label{lem:gradient-volume}
Let $q\asymp_W R$, $|\theta|\sim \phi$, and let $V_0$ be given by \eqref{eq:V}.
Assume that \eqref{eq:diffU} holds.
Then, for every $\varepsilon>0$ and every
$\bfv\in \RR^n$, we have
\[
\meas\left\{
\bfx\in\RR^n:
|\bfx/P-\bm{x}_0|\le 2\ell,\;
|q\theta\nabla F(\bm{x})-\bm{v}|\leq P^\varepsilon V_0
\right\}
\ll
P^{n+O_n(\varepsilon)}\left(\frac{R}{PV_0}\right)^{\rho},
\]
where $\rho=r(F)$.
\end{lemma}

\begin{proof}
We may assume that $\theta \neq 0$, since $\phi>0$.
Write $\x=P\y$.
As $F$ is cubic, we have
$\nabla F(P\y)=P^2\nabla F(\y)$ and the volume in question is
\[
P^n\meas\left\{
\bfy\in\RR^n:
|\bfy-\bm{x}_0|\le 2\ell,\;
|q\theta P^2\nabla F(\y)-\bm{v}|\leq P^\varepsilon V_0
\right\}.
\]
By \eqref{eq:diffU}, the box $|\bfy-\bm{x}_0|\le 2\ell$ is contained in $U$.
Using the diffeomorphism
$$\bfy\mapsto (\Psi(\bfy),y_{\rho+1},\dots,y_n)$$
on $U$ and the Jacobian bound \eqref{eq:smoothsub''},
where $\Psi$ was defined in \eqref{eq:Psi},
we find by Fubini's theorem that
\[
\meas\{\y\in U:|\Psi(\y)-\bm{a}|\leq \Delta\}
\ll \Delta^\rho ,
\]
uniformly for $\bm{a}\in\RR^\rho$ and $0<\Delta\leq 1$.
For $\Delta>1$ we take the trivial bound $O(1)$ for the volume. 
It follows that
\begin{equation}
\label{preslop}
\meas\left\{
\bfy\in\RR^n:
|\bfy-\bm{x}_0|\le 2\ell,\;
|\nabla F(\y)-\bm{b}|\leq \Delta
\right\}
\ll \min\{1,\Delta^\rho\},
\end{equation}
uniformly 
for $\bm{b}\in\RR^n$.

Applying this with $\bm{b}=(q\theta P^2)^{-1}\bm{v}$ and 
$\Delta= (q|\theta|P^2 )^{-1}P^\varepsilon V_0$, we deduce that 
\begin{equation*}
\meas\left\{
\bfx\in\RR^n:
|\bfx/P-\bm{x}_0|\le 2\ell,\;
|q\theta\nabla F(\bm{x})-\bm{v}|\leq P^\varepsilon V_0
\right\}
\ll 
P^{n+O_n(\ve)}
\min\left\{
1,
\frac{ V_0}{R\phi P^2}
\right\}^\rho,
\end{equation*}
since $q\asymp_W R$ and $|\theta|\sim\phi$. The lemma will follow if we are able to show that 
$$
\min\left\{
1,
\frac{ V_0}{R\phi P^2}\right\}
\ll \frac{R}{PV_0},
$$
where we recall that $V_0$ is given by \eqref{eq:V}. 
But this is easily checked on a case distinction between 
the case  $\phi P^3\leq 1$ (so that $V_0=RP^{-1}$) and the case
 $\phi P^3>1$ (so that $V_0=RP^{-1}\sqrt{\phi P^3}$).
\end{proof}

We  can also exploit averaging over $\bfv\in \ZZ^n$ to get alternative estimates. We begin with the following simple estimate.

\begin{lemma}\label{lem:gradient-volume2}
Let $q\asymp_W R$, $|\theta|\sim \phi$, and let $V, V_0$ be given by \eqref{eq:V}.
Then, for every $\varepsilon>0$,
we have
\[
\sum_{\substack{\bm{v}\in \ZZ^n\\ 
|\bm{v}|\leq P^\ve V}}
\meas\left\{
\bfx\in\RR^n:
|\bfx/P-\bm{x}_0|\le 2\ell,\;
|q\theta\nabla F(\bm{x})-\bm{v}|\leq P^\varepsilon V_0
\right\}
\ll
P^{n+O_n(\varepsilon)}(1+V_0)^n.
\]
\end{lemma}

\begin{proof}
Interchanging the integral with the sum, we are left with estimating the number of  integer vectors $\bm{v}$, each of whose entries lie in an interval of length $O(P^\ve V_0)$. The statement of the lemma is now clear,
on integrating over $\bm{x}$ such that $|\bfx/P-\bm{x}_0|\le 2\ell$.
\end{proof}

The following result interpolates between aspects of both of the previous estimates.
Taking $L = 2 P^\ve V$ morally recovers Lemma~\ref{lem:gradient-volume},
whereas taking $L=0$ refines Lemma~\ref{lem:gradient-volume2}.
Lemma~\ref{aislop} will eventually let us use $\ell^\infty$-$\ell^1$ H\"{o}lder's inequality
on shifted boxes of an optimised length $L\ge 0$,
in order to separate the integral from other factors in \eqref{eq:defERROR2}.

\begin{lemma}
\label{aislop}
Fix $\varepsilon>0$.
Let $P,R\ge 1/2$ and $\phi>0$.
Let $V,V_0$ be given by \eqref{eq:V}.
Assume that \eqref{eq:diffU} holds.
Let $\rho=r(F)$.
Let $L\ge 0$ be a real number.
If $|\alpha| \asymp_W R\phi$, then
\begin{equation*}
\sum_{\bm{v}\in \ZZ^n}
M_\alpha(\bm{v},L) \ll
\begin{cases}       
0 &\text{ if $P^\ve V < 1$,}\\
P^{n+O_n(\ve)} (1+\tfrac{L}{V})^\rho (1+V_0+L)^{n-\rho} V_0^\rho & \text{ if $\phi P^3>1$},\\
P^{n+O_n(\ve)} (V+L)^n & \text{ if $\phi P^3\leq 1$},
\end{cases}
\end{equation*}
where
\begin{equation*}
M_\alpha(\bm{v},L) := \sup_{\substack{\bm{a}\in \ZZ^n\\ 
0<|\bm{a}|\leq P^\ve V\\
|\bm{a}-\bm{v}| \le L}}
\meas\left\{
\bfx\in\RR^n:
|\bfx/P-\bm{x}_0|\le 2\ell,\;
|\alpha\nabla F(\bm{x})-\bm{a}|\leq P^\varepsilon V_0
\right\},
\end{equation*}
with the convention that $\sup_{\emptyset} = 0$
if the set in the supremum is empty.

\end{lemma}

\begin{proof}
For each $\bm{v}$
for which the set in the supremum is nonempty,
choose $\bm{a}=\bm{a}(\bm{v})$
with $0<|\bm{a}|\leq P^\ve V$
and $|\bm{a}-\bm{v}| \le L$
such that
\begin{equation}
\label{choose-shifted box-max-vector-a}
M_\alpha(\bm{v},L) =
\meas\left\{
\bfx\in\RR^n:
|\bfx/P-\bm{x}_0|\le 2\ell,\;
|\alpha\nabla F(\bm{x})-\bm{a}|\leq P^\varepsilon V_0
\right\}.
\end{equation}
In particular, the outer sum is zero unless $P^\ve V \ge 1$, which we now assume.
We begin as in the proof of Lemma \ref{lem:gradient-volume} by writing $\bm{x}=P\bm{y}$,
where $\bfy\in U$ by \eqref{eq:diffU}.
By the discussion after \eqref{eq:Psi},
we have a diffeomorphism
\begin{equation}
\label{key-diffeomorphism}
U\xrightarrow{\cong} S\subset \RR^n, \quad 
\bm{y} \mapsto \left(\Psi(\bm{y}),y_{\rho+1},\dots,y_n\right),
\end{equation}
for which the Jacobian and the  Jacobian of the inverse map are uniformly bounded on $U$.
Then, for all vectors $\bm{s},\bm{t}\in \RR^\rho$ and $\bm{y}''\in \RR^{n-\rho}$
with $(\bm{s},\bm{y}''),(\bm{t},\bm{y}'')\in U$, we have
\begin{align}
|\nabla F(\bm{s},\bm{y}'') - \nabla F(\bm{t},\bm{y}'')|
&\ll_U |\bm{s}-\bm{t}|,
\label{easy-grad-smooth} \\
|\Psi(\bm{s},\bm{y}'') - \Psi(\bm{t},\bm{y}'')|
&\gg_U |\bm{s}-\bm{t}|.
\label{inverse-lipschitz-Psi}
\end{align}
The inequality \eqref{easy-grad-smooth} is a general consequence of the smoothness of $\nabla F$;
the inequality \eqref{inverse-lipschitz-Psi} similarly follows from the smoothness of the inverse of the diffeomorphism \eqref{key-diffeomorphism}.

Let us write $\bfv=(\bfv',\bfv'')$
and $\bm{a}=(\bm{a}',\bm{a}'')$
for $\bfv',\bm{a}'\in \ZZ^\rho$ and $\bfv'',\bm{a}''\in \ZZ^{n-\rho}$.
Similarly write $\bm{y}=(\bm{y}',\bm{y}'')$,
where $\bm{y}'\in \RR^\rho$ and $\bm{y}''\in \RR^{n-\rho}$.
By the triangle inequality
$$|\alpha\nabla F(\bm{x})-\bm{v}|
\le |\alpha\nabla F(\bm{x})-\bm{a}| + |\bm{a}-\bm{v}|
\le P^\ve V_0 + L$$
and the change of coordinates $\bm{x}=P\bm{y}$,
\begin{equation}
\label{drop-L-box}
\meas\left\{
\bfx\in\RR^n:
\bfx/P\in U,\;
|\alpha\nabla F(\bm{x})-\bm{a}|\leq P^\varepsilon V_0
\right\}
\le P^n \int_{\substack{\bfy\in U\\
|\alpha P^2\Psi(\bm{y})-\bm{a}'|\leq P^\varepsilon V_0}}
\bm{1}_{A(\bm{y},\bm{v})}\,
\dd\bm{y},
\end{equation}
where $A(\bm{y},\bm{v})$ denotes the condition that
$|\alpha P^2 \frac{\partial F(\bfy)}{\partial y_j}  
-v_j|\leq P^\varepsilon V_0 + L$ for $1\leq j\leq n$.
The condition $A(\bm{y},\bm{v})$ constrains $\bm{v}$ directly,
rather than $\bm{a}$,
at the cost of introducing a summand $L$.
If $\bm{v}'$ and $\bm{y}''$ are fixed,
then any two possible choices $\bm{s},\bm{t}\in \RR^\rho$ for the vector
$\bm{y}'$ satisfy $$|\bm{s}-\bm{t}|\ll_U \frac{P^\varepsilon V_0 + L}{|\alpha P^2|},$$
by $A(\bm{y},\bm{v})$ applied with $1\le j\le \rho$,
together with the inequality \eqref{inverse-lipschitz-Psi}.
Then, \eqref{easy-grad-smooth} and
the case $\rho+1\le j\le n$ of $A(\bm{y},\bm{v})$
constrain the vector $\bm{v}''$ to lie in a box
$B(\bm{v}',\bm{y}'')$
of diameter
$$\ll_U \sup_{\bm{s},\bm{t}} |\alpha P^2| |\bm{s}-\bm{t}|
+ P^\varepsilon V_0 + L
\ll_U |\alpha P^2| \frac{P^\varepsilon V_0 + L}{|\alpha P^2|}
+ P^\varepsilon V_0 + L
\asymp P^\varepsilon V_0 + L.$$
The strategy is now to
choose $\bm{v}'$, $\bm{y}''$, $\bm{v}''$ and $\bm{y}'$ in that order,
where $$|\bm{v}'|\le |\bm{v}|\le |\bm{a}| + |\bm{v}-\bm{a}| \le P^\ve V + L.$$
Plugging \eqref{drop-L-box} into \eqref{choose-shifted box-max-vector-a}
and summing over $\bfv$,
it follows from our analysis of $\bm{v}'$ and $\bm{v}''$ above that
\begin{equation*}
\begin{split}
\sum_{\bm{v}\in \ZZ^n}
M_\alpha(\bm{v},L)
&\le P^n \sum_{\bm{v}\in \ZZ^n} \int_{\substack{\bm{y}\in U\\
|\alpha P^2\Psi(\bm{y})-\bm{a}'|\leq P^\varepsilon V_0}}
\bm{1}_{|\bm{v}'| \le P^\ve V + L}
\bm{1}_{\bm{v}''\in B(\bm{v}',\bm{y}'')}\,
\dd\bm{y} \\
&= P^n \sum_{\substack{\bm{v}'\in \ZZ^\rho\\ 
|\bm{v}'|\le P^\ve V + L}}
\int_{\substack{\bm{y}''\in \RR^{n-\rho}\\ |\bm{y}''|\ll_U 1}}
\sum_{\bm{v}''\in B(\bm{v}',\bm{y}'')}
\int_{\substack{\bm{y}'\in \RR^\rho\\ \bm{y}\in U\\
|\alpha P^2\Psi(\bm{y})-\bm{a}'|\leq P^\varepsilon V_0}}
\dd\bm{y}'\, \dd\bm{y}'',
\end{split}
\end{equation*}
by Fubini's theorem.
For any fixed $\bm{v}'$, $\bm{y}''$ and $\bm{v}''$, the inner volume is 
$$
\ll_U P^{O_n(\ve)} \min\left\{1, \frac{ V_0}{R\phi P^2}\right\}^\rho,
$$
arguing via the diffeomorphism \eqref{key-diffeomorphism}
as in the proof of \eqref{preslop}.
The number of available vectors $\bm{v}'$ and $\bm{v}''$
being $O((1+P^\ve V + L)^\rho)$ and $O((1+P^\ve V_0 + L)^{n-\rho})$, respectively, we get
\begin{equation*}
\sum_{\bm{v}\in \ZZ^n}
M_\alpha(\bm{v},L)
\ll P^{n+O_n(\ve)} (V + L)^\rho (1+V_0+L)^{n-\rho}
\min\left\{1, \frac{ V_0}{R\phi P^2}\right\}^\rho,
\end{equation*}
since $P^\ve V\ge 1$.

If $\phi P^3\leq 1$ then $V=V_0=RP^{-1}$ and the overall bound is clearly
$O(P^{n+O_n(\ve)} (V+L)^n)$,
since $P^\ve V\ge 1$.
On the other hand, if 
 $\phi P^3>1$ then $V=R\phi P^2$ and
 $V_0=R\phi^{1/2}P^{1/2}$. But then
 $V_0/(R\phi P^2)=(\phi P^3)^{-1/2}< 1$, leading to the overall  contribution
$$
\ll P^{n+O_n(\ve)} (V+L)^\rho (1+V_0+L)^{n-\rho} (V_0/V)^\rho
= P^{n+O_n(\ve)} (1+\tfrac{L}{V})^\rho (1+V_0+L)^{n-\rho} V_0^\rho,
$$ 
as required.
\end{proof}

Let $V,V_0$ be given by \eqref{eq:V}.
Plugging Lemma~\ref{lem:Iq'} into \eqref{eq:defERROR2}
and writing $\tau = \lcm(q,W)\theta/R$,
we get
\begin{equation}\label{eq:bislop}
E(P;R,\phi)\ll \phi + \frac{1}{L^n} \int_{|\tau|\asymp_W \phi}
\sum_{\bm{v}\in \ZZ^n} M_{R\tau}(\bm{v},L)
\sum_{q\sim R} \frac{1}{R^n}
\sum_{\substack{\bm{a}\in \ZZ^n\\ 
0<|\bm{a}|\leq P^\ve V\\
|\bm{a}-\bm{v}| \le L}} |S_{q,W}(\bm{a})|
\,\dd\tau
\end{equation}
for any integer
$1\le L=L(P,R,\phi)\ll P^{100}$,
since by averaging over shifted boxes of length $L$
and applying $\ell^\infty$-$\ell^1$ H\"{o}lder's inequality,
if $f,g\ge 0$ then
\begin{equation*}
(2L+1)^n \sum_{\bm{v}\in \ZZ^n} f(\bm{v}) g(\bm{v})
= \sum_{\bm{v}\in \ZZ^n}
\sum_{\substack{\bm{a}\in \ZZ^n\\ |\bm{a}-\bm{v}|\le L}} f(\bm{a}) g(\bm{a})
\le \sum_{\bm{v}\in \ZZ^n}
\sup_{\substack{\bm{a}\in \ZZ^n\\ |\bm{a}-\bm{v}|\le L}} g(\bm{a})
\sum_{\substack{\bm{a}\in \ZZ^n\\ |\bm{a}-\bm{v}|\le L}} f(\bm{a}).
\end{equation*}
It is clear from \eqref{eq:bislop} that we now need strong  bounds for the average of 
the exponential sum $|S_{q,W}(\bm{a})|$  over all integer vectors $\bm{a}$
in shifted boxes and over all moduli $q$ in a dyadic interval. 
We expect the following to be true. 

\begin{conjecture}
\label{conj:clean-shifted}
Fix $n\ge 10$.
There exists a constant $b < 2n/3$ such that
for all real $\ve>0$,
$B,L,R\ge 1/2$,
and vectors $\bm{v}\in \ZZ^n$ with $|\bm{v}|\le B$,
we have
\begin{equation*}
\sum_{q\sim R} 
\sum_{\substack{\bm{0}\ne \bm{a}\in \ZZ^n\\ 
|\bm{a}-\bm{v}| \le L}} |S_{q,W}(\bm{a})|
\ll_\ve (BLR)^\ve (L + R^{1/3})^n R^b.
\end{equation*}
\end{conjecture}
For $n=10$ this assertion is Proposition~\ref{prop:n10-clean-shifted},
with $b=20/3-\gamma$ and the fixed $W$ chosen above.

\begin{prop}
\label{prop:remove-integral}
Assume Conjecture~\ref{conj:clean-shifted} for some fixed $n\ge 10$.
Let $\rho=r(F)$.
Then
\begin{equation*}
E(P;R,\phi) \ll \min\{1,\phi P^3\} \left(1 + P^{n-3+O(\eta+\ve)-\Delta}\right),
\end{equation*}
where $\Delta = \min(\rho/4-3/2, \rho/2-1, 2n/3-b) > 0$ is admissible.
\end{prop}

\begin{proof}
Plugging Lemma~\ref{aislop} with $\alpha = R\tau$
and Conjecture~\ref{conj:clean-shifted} with $B \ll P^{100}$
into \eqref{eq:bislop},
the desired bound is trivial unless
$P^\ve V\ge 1$,
in which case
\begin{equation*}
E(P;R,\phi) \ll \phi + (1 + \tfrac{R^{1/3}}{L})^n \phi R^{b-n} P^{n+O_n(\ve)}
\begin{cases}       
(1+\tfrac{L}{V})^\rho (V_0+L)^{n-\rho} V_0^\rho & \text{ if $\phi P^3>1$},\\
(V+L)^n & \text{ if $\phi P^3\leq 1$},
\end{cases}
\end{equation*}
for any real $1\le L\ll P^{100}$,
since $L\asymp \lceil{L}\rceil\in \ZZ$.
Here $1/2\le R\ll P^{3/2}$ and $\phi \ll (RP^{3/2})^{-1+\eta}$.

If $\phi P^3\leq 1$ then $V=RP^{-1}$ and $P^\ve V\gg 1$,
so $R^{1/3} \gg R/P = V \gg P^{-\ve}$, since $R\ll P^{3/2}$.
Taking $L \asymp R^{1/3} \gg V$ gives
\begin{equation*}
E(P;R,\phi) \ll \phi + \phi R^{b - 2n/3} P^{n+O_n(\ve)}
\ll \phi + \phi P^{b - 2n/3} P^{n+O(\ve)},
\end{equation*}
since $R\gg P^{1-\ve}$.
Thus $\Delta = 2n/3 - b > 0$ is admissible in this case.

Now suppose $\phi P^3>1$.
Then $V=R\phi P^2$ and
$V_0=R\phi^{1/2}P^{1/2}$,
with $P^\ve V\gg 1$,
$V>V_0$
and $V_0 \ll (RP^{3/2})^{\eta/2} R^{1/2} P^{-1/4} \ll (RP^{3/2})^{\eta/2} R^{1/3}$.
Taking $L \asymp R^{1/3}$ gives
\begin{equation*}
E(P;R,\phi) \ll 1 + \phi R^{b-n} P^{n+O_n(\ve)}  
(1+\tfrac{R^{1/3}}{V})^\rho (V_0+R^{1/3})^{n-\rho} V_0^\rho.
\end{equation*}
Writing $V = (\phi P^3)^{1/2}V_0$ and $V_0+R^{1/3} \ll (RP^{3/2})^{\eta/2} R^{1/3}$, this is
\begin{equation*}
\ll 1 + \phi R^{b-2n/3-\rho/3} P^{n+O(\eta+\ve)}  
(V_0+\tfrac{R^{1/3}}{(\phi P^3)^{1/2}})^\rho.
\end{equation*}
Since $\rho\ge 2$, optimising over the interval $\max(P^{-3}, R^{-1} P^{-2-\ve})
\ll \phi \ll (RP^{3/2})^{-1+\eta}$ gives
\begin{equation*}
\begin{split}
E(P;R,\phi) &\ll 1 + R^{b-2n/3-\rho/3} P^{n+O(\eta+\ve)}  
(R^{-1}P^{-3/2} R^{\rho/2}P^{-\rho/4} + P^{-3} \max(1, P/R)^{1 - \rho/2} R^{\rho/3}) \\
&= 1 + R^{b-2n/3} P^{n-3+O(\eta+\ve)}  
((R/P^{3/2})^{\rho/6 - 1} + \max(1, P/R)^{1 - \rho/2}).
\end{split}
\end{equation*}
By writing each term on the right-hand side as
$P$ raised to a piecewise linear function of the quantity $\log_P R$,
we find that this is optimised at either $R\asymp 1$, $R\asymp P$, or $R\asymp P^{3/2}$.
Thus
\begin{equation*}
E(P;R,\phi) \ll 1 + P^{n-3+O(\eta+\ve)}  
((1/P^{3/2})^{\rho/6 - 1} + P^{1 - \rho/2} + P^{b-2n/3}).
\end{equation*}
Since $\rho = r(F)\ge \lceil 2n/3\rceil \ge 7$,
the saving $\Delta = \min(\rho/4-3/2, \rho/2-1, 2n/3-b) > 0$ is admissible.
\end{proof}

Dyadic summation of Proposition~\ref{prop:remove-integral}
over $R$ and $\phi$
gives $N(P) = M(P) + o(P^{n-3})$,
unconditionally for $n=10$.
Thus $N(P) \gg P^{n-3}$ for all sufficiently large $P$,
by Lemma~\ref{LEM:pos-main-term}.
Theorem~\ref{thm:main}
follows first for $n=10$,
and then for $n\ge 11$,
by setting the last $n-10$ variables of $F$
equal to zero.

\appendix
\section{Applying singular support}
\label{sec:n10-singular-support}

Let $n\geq 2$ and let $F\in\ZZ[x_1,\dots,x_n]$ be a nonzero homogeneous
polynomial of degree $d\geq1$.
For a finite field $\FF_q$ of characteristic $p$,
we define a nontrivial additive character $\psi_q\colon \FF_q\to \CC^\ast$
by $\psi_q(t)=e_p(\Tr_{\FF_q/\FF_p}(t))$.
For vectors $\bfv\in \FF_q^n$, we are interested in the exponential sum
\[
T_q(\bfv)
= \sum_{\substack{\bfx\in\FF_q^n\\F(\bfx)=0}}
\psi_q(\bfx\cdot\bfv)
-q^{n-1}\bm1_{\bfv=\bm0}
\]
that we defined in \eqref{n10:T}.
The general Fouvry--Katz--Laumon stratification \cite[Theorem~1.1]{FK01}
gives exceptional loci for such sums
that measure the failure of square-root cancellation.
In our homogeneous setting,
we give an alternative construction for such a stratification using
the singular support of Beilinson \cite{Bei}.
Let $\mathcal X\subset \PP^{n-1}_{\ZZ}$ be the hypersurface defined by $F=0$.

\begin{thm}\label{n10:ss}
There exist an integer $N\geq1$ and closed subschemes
\[
\A^n_\ZZ=\mathcal Z_0\supset \mathcal Z_1\supset\cdots\supset \mathcal Z_n=\{\bm0\}
\]
defined by homogeneous polynomials such that the following hold:
\begin{enumerate}
\item \label{item:strat-geom-dim-bound}
We have $\dim (\mathcal Z_i \otimes \QQ) \le n-i$ for every $0\le i\le n$.

\item \label{item:strat-fourier-bound}
For every finite field $\FF_q$
of characteristic $p\nmid N$, every $0\leq i\leq n-1$, and every
vector $\bfv\in \FF_q^n \setminus \mathcal Z_{i+1}(\FF_q)$, we have
\begin{equation}\label{eq:main-bound}
T_q(\bfv)\ll_{n,d} q^{(n-1+i)/2}.
\end{equation}

\item If $\mathcal X \otimes \QQ$ is geometrically integral
and $d\ge 2$, then for every $0\leq i\leq n-1$, we have
\begin{equation}
\label{EQ:nice-sandwich}
\cone W_i \subset \mathcal Z_{i+1} \otimes \QQ \subset \cone B_i,
\end{equation}
in the notation of \eqref{EQ:general-bad-hyperplanes}
and \eqref{EQ:general-large-gauss-fibres}.

\item Let $\bar{\mathcal Z_i(\QQ)}$ denote the Zariski closure of
$\mathcal Z_i(\QQ)$ in the variety $\mathcal Z_i \otimes \QQ$.
If $\mathcal X(\QQ)=\emptyset$, then for every $1\leq i\leq n-1$, we have
\begin{equation}\label{eq:rational-main}
\dim \bar{\mathcal Z_i(\QQ)}
\le n-i-1.
\end{equation}
\end{enumerate}
\end{thm}

\begin{rem}
As in the main body of the paper,
we let $\cone Y$ denote the affine cone of a projective variety $Y$,
where $\cone \emptyset = \{\bm0\}$.
The reader may notice that \eqref{eq:main-bound} holds
trivially for all $\bfv\in \FF_q^n$ if $i=n-1$, even if $\bfv\in \mathcal Z_n(\FF_q)$.
Nonetheless, \eqref{EQ:nice-sandwich} for $i=n-1$ requires us to include the origin
$\bm0$ in $\mathcal Z_n$.
Including the origin
is also natural from the perspective of projective geometry,
which is pervasive throughout our paper.
\end{rem}

\noindent \textbf{Notation and conventions.} 
Over a field $k$, write $X=\mathcal X_k$. We write $P=\PP^{n-1}$ and let $P^\vee$ denote its dual projective space. It will always be clear from the context over which ring we are working. For $h=[\bfv]\in P^\vee$ we write $H_h$ for the
hyperplane $\bfv\cdot\bfx=0$ and $X_h=X\cap H_h$. Given a scheme $Y$ over $S$ and a morphism $\Spec(R) \to S$, we write $Y_R$ for the base change of $Y$ to $\Spec(R)$.
We put $\dim\emptyset=-1$.
We use the notation $\subset$ for subsets,
not necessarily proper;
the notation $\subsetneq$ is reserved for proper subsets.

\subsection{The finite-field argument}

\subsubsection{Reduction to a Gysin isomorphism}
Throughout this section, we fix an auxiliary prime $\ell$ and work with a finite field $k=\FF_q$ of characteristic $p\neq \ell$.
We assume $F \in k[x_1,\dots, x_n]$ is nonzero and homogeneous of degree $d$,
and continue to write $X\subset P$ for the projective hypersurface it defines.
\begin{lemma}\label{app:ss:lem:homogeneity}
For every $\bfv\neq\bm0$, with $h=[\bfv]$, one has
\begin{equation*}
 T_q(\bfv)=1+q\#X_h(k)-\#X(k).
\end{equation*}
\end{lemma}
\begin{proof}
When $q=p$, this follows from
\eqref{eq:Sp-projective-v-nonzero}
together with the identity $S_p(\bfv) = pT_p(\bfv)$
observed in \eqref{eq:pST}.
For general prime powers $q$, the same underlying algebraic manipulations apply.
\end{proof}

Applying the Grothendieck--Lefschetz trace formula in \'etale cohomology
to the $\bar k$-varieties $X_{h,\bar k} = (X_h)_{\bar k}$ and $X_{\bar k}$,
the previous lemma gives for every $\bfv\neq\bm0$,
\begin{equation}\label{app:ss:eq:trace}
 T_q(\bfv)=1+\sum_{i\geq0}(-1)^i\left(
 \Tr(\Fr_q\mid H^{i-2}(X_{h,\bar k},\QQ_\ell)(-1))
 -\Tr(\Fr_q\mid H^i(X_{\bar k},\QQ_\ell))\right),
\end{equation}
where $\Fr_q$ denotes the geometric Frobenius.
By Deligne's general form of the Riemann hypothesis \cite{Del},
an eigenvalue $\alpha$ of the action of $\Fr_q$
on $H^{i-2}(X_{h,\bar k},\QQ_\ell)(-1)$ or $H^i(X_{\bar k}, \QQ_\ell)$
satisfies $$|\alpha|\leq q^{i/2}.$$
The main idea of the proof of Theorem~\ref{n10:ss} is now
to construct strata on $P^\vee$ measuring in which degrees
the groups $H^{i-2}(X_{h,\bar k},\QQ_\ell)(-1)$ and $H^i(X_{\bar k}, \QQ_\ell)$ are isomorphic. 

For this we will use the \emph{Gysin map} in a form that applies even when $X$ or $X_h$
is singular, following the construction in \cite[\S2.1]{RS}.
Let $R = \ZZ/\ell^r\ZZ$ for some integer $r\geq1$.
Let
$$\phi:X\hookrightarrow P,
\qquad i_h:H_h\hookrightarrow P,$$
be the obvious inclusion maps into $P$,
and let $\mathcal F=\phi_*R_X$,
where $R_X$ denotes the constant sheaf on $X$ valued in $R$.
Although \cite[\S2.1]{RS} only explicitly discusses the case $r=1$,
its content directly generalises to arbitrary $r\ge 1$.
Purity for the smooth divisor $H_h\subset P$ identifies
$i_h^!R_P$ with the complex $R_{H_h}(-1)[-2]$
in the derived category $D^b_c(H_h,R)$.
Since $R_{H_h}(-1)$ is locally free of rank one,
the natural morphism
$i_h^*\mathcal F\otimes_R^L i_h^!R_P\to i_h^!\mathcal F$,
where $\otimes_R^L$ denotes the derived tensor product in $D_c(H_h,R)$,
therefore gives a morphism
\begin{equation}
\label{EQ:derived-gysin-map}
i_h^*\mathcal F(-1)[-2]\longrightarrow i_h^!\mathcal F.
\end{equation}
We can compose its pushforward by $i_h$
with the counit $(i_h)_*i_h^!\mathcal F\to \mathcal F$
and take geometric global sections, so that proper base change
identifies the resulting map in degree $a\in \ZZ$ with
\begin{equation*}
G_h^a(R):H^{a-2}(X_{h,\bar k},R)(-1)
\longrightarrow H^a(X_{\bar k},R).
\end{equation*}
The construction is compatible
with the Galois action
and with the coefficient reduction maps $\ZZ/\ell^{r+1}\ZZ\to\ZZ/\ell^r\ZZ$.
Thus the maps $G_h^a(\ZZ/\ell^r\ZZ)$
lift to the $\ell$-adic coefficients $\ZZ_\ell$ and $\QQ_\ell$.
The key connection between the $\ell$-adic Gysin maps $G_h^a(\QQ_\ell)$
and the size of $T_q(\bfv)$ is recorded in the following lemma.
\begin{lemma}\label{app:ss:lem:reduction}
If $t\geq0$
and $G_h^i(\QQ_\ell)$ is an isomorphism for all $i>t$, then,
for $\bfv\neq\bm0$ with $[\bfv]=h$,
\[
 T_q(\bfv)\ll_{n,d} q^{t/2}.
\]
\end{lemma}
\begin{proof}
In \eqref{app:ss:eq:trace}, the terms with $i>t$ vanish, since by assumption $G^i_h(\QQ_\ell)$ is a Galois-equivariant isomorphism for $i>t$.
Deligne's general form of the Riemann hypothesis \cite[Th\'eor\`eme 3.3.1]{Del} bounds the
absolute values of the eigenvalues of the nonzero terms by
$q^{i/2}\leq q^{t/2}$.
Katz's bounds~\cite[Theorem~3 and its corollary]{KatzBetti}
bound the total Betti numbers of $X$ and $X_h$ in terms of $n,d$:
they are cut out by at most two equations of degree at most $d$.
The triangle inequality in \eqref{app:ss:eq:trace}
therefore gives the desired bound for $T_q(\bfv)$.
\end{proof}

Consequently, we would like to have an algebraic criterion
for a given Gysin map $G_h^i(\QQ_\ell)$ to be an isomorphism.
We will first give such a criterion with $\FF_\ell$-coefficients
in Proposition~\ref{app:ss:prop:RS} below,
before lifting it to $\ell$-adic coefficients $\ZZ_\ell$ and $\QQ_\ell$
in Lemma~\ref{app:ss:lem:coefficients}.

\subsubsection{Singular support and the exceptional loci}
We first work over a perfect field $k$ with coefficients
$\Lambda=\FF_\ell$, where $\ell\neq\operatorname{char}k$.
For a smooth equidimensional variety $M$
and a constructible \'etale sheaf $\mathcal F\in D^b_c(M,\Lambda)$,
Beilinson defines
the singular support $\operatorname{SS}(\mathcal F)\subset T^*M$ as the smallest closed
conical set on which $\mathcal F$ is micro-supported
\cite[\S\S1.2--1.3]{Bei}.
The precise definition of $\operatorname{SS}(\mathcal F)$ will not be important for us;
all that matters is that the same definition is consistently used in
the key references \cite{Bei,HY,RS}.
In any case, the rough idea is that the singular support records the cotangent directions in which local acyclicity can fail.

The properties needed here are the following. Every irreducible
component of $\operatorname{SS}(\mathcal F)$ has dimension $\dim M$;
the image of $\operatorname{SS}(\mathcal F)$ in $M$ is
the closed support $\operatorname{Supp} \mathcal F$;
and formation of singular support commutes with extensions
of the ground field. All these properties are recorded in \cite[Theorems 1.3, 1.4(iii), and
Lemma 2.3(iii)]{Bei}.

We apply this construction with $M=P=\PP^{n-1}_k$ and $\mathcal F=\phi_*\Lambda_X$,
so that $\operatorname{Supp} \mathcal F=X$.
Borrowing notation from \cite[Theorem~2.2 and its proof]{RS},
consider the universal hyperplane
\[
 \Phi=\{(x,h)\in P\times P^\vee:x\in H_h\}
       \cong \PP(T^*P),
\]
where the identification $\cong$ assigns to $(x,h)$ the
conormal line to $H_h$ at $x$. Put
\begin{equation}\label{app:ss:eq:Sigma}
 \Sigma=\PP(\operatorname{SS}(\mathcal F))\subset \Phi,
 \qquad \pi:\Sigma\longrightarrow P^\vee.
\end{equation}
The map $\pi$ is proper and $\dim\Sigma\leq n-2$. In fact, every
irreducible component of $\Sigma$ has dimension exactly equal to
$\dim M - 1 = n-2$.
Indeed, a component of $\operatorname{SS}(\mathcal F)$
contained in the zero section of $T^*M$
(and thus disappearing after projectivisation)
would have base of dimension $\dim M = n-1$,
whereas $\operatorname{Supp} \mathcal F=X$ has dimension $n-2$.

\begin{defin}\label{app:ss:def:Yi}
For $1\leq i\leq n$, define
\begin{equation}\label{app:ss:eq:Yi}
 Y_i=\{h\in P^\vee:\dim\Sigma_{\bar h}\geq i-1\},
\end{equation}
where $\bar h$ is a geometric point over $h$.
\end{defin}

\begin{lemma}\label{app:ss:lem:dimension}
The loci $Y_i$ satisfy $\codim_{P^\vee}Y_i\geq i$.
In particular $Y_n=\emptyset$.
\end{lemma}
\begin{proof}
For an irreducible component
$E$ of $Y_i$ over an algebraic closure of $k$, a component of $\pi^{-1}(E)$
dominating $E$ has dimension at least $\dim E+i-1$. Hence
$\dim E+i-1\leq\dim\Sigma\leq n-2$, giving
$\dim E\leq n-1-i$. Thus $\codim_{P^\vee}Y_i\geq i$ and $Y_n=\emptyset$.
\end{proof}

The key property of the loci $Y_i$ is that they measure in which degree the Gysin map is an isomorphism.
The following result is assembled from the work of Raskin--Smith \cite{RS}.
\begin{prop}\label{app:ss:prop:RS}
Let $1\le i\le n$.
For $h\notin Y_i$, the map $G_h^a(\Lambda)$
is an isomorphism if $a>n-2+i$, and is surjective if $a=n-2+i$.
\end{prop}
\begin{proof}
We apply \cite[Theorem~1.2 and its proof]{RS} to
the inclusion morphism $\phi:X\hookrightarrow P$,
with parameters $r=\dim(X\times_P X)=\dim X=n-2$, $c=i$ and $q=a$.
The construction in \cite[Theorem~2.2, Corollary~2.3 and their proofs]{RS} chooses
precisely the loci $Y_i$ defined in \eqref{app:ss:eq:Yi}:
the cone of the morphism $i_h^*\mathcal F(-1)[-2]\to i_h^!\mathcal F$
defined in \eqref{EQ:derived-gysin-map}
is supported on
the image of $\Sigma_h$ in $H_h$.
Outside $Y_i$ this support has dimension at most $i-2$,
by the definition of $Y_i$.
The desired conclusion follows from the statement of \cite[Theorem~1.2]{RS},
with $Y_i$ in place of $Z_c$.
\end{proof}

\subsubsection{Passage from finite to rational coefficients}
Next, we pass from finite to $\ell$-adic coefficients using an elementary argument involving coefficient
long exact sequences and inverse limits.

\begin{lemma}
\label{app:ss:lem:coefficients}
Fix a hyperplane $h$ and an integer $t$. Suppose $G_h^a(\FF_\ell)$ is
an isomorphism for $a>t$ and is surjective for $a=t$. Then the same
statements hold with coefficients $\ZZ/\ell^r\ZZ$ for every $r\geq1$,
and with coefficients $\ZZ_\ell$ and $\QQ_\ell$.
\end{lemma}
\begin{proof}
Write $\Lambda_r=\ZZ/\ell^r\ZZ$ and abbreviate
\[
 A_r^a=H^{a-2}(X_{h,\bar k},\Lambda_r)(-1),\qquad
 B_r^a=H^a(X_{\bar k},\Lambda_r).
\]
For $r\geq2$, the coefficient sequence
\[
 0\longrightarrow\Lambda_{r-1}\xrightarrow{\ell}\Lambda_r
       \longrightarrow\Lambda_1\longrightarrow0
\]
gives long exact sequences compatible with the Gysin maps,
by standard properties of derived tensor products with respect to distinguished triangles
in the derived category of $\Lambda_r$-modules.
To prove the statement of Lemma~\ref{app:ss:lem:coefficients} for $\ZZ/\ell^r\ZZ$,
we induct on $r\ge 1$.
The case $r=1$ is precisely our assumption,
so assume that $r\ge 2$ and the statement already holds for $r-1$.
For $a>t$, we obtain a commutative diagram that is exact on its horizontal arrows:
\[
\begin{tikzcd}[column sep=small]
 A_1^{a-1}\arrow[r]\arrow[d]
 & A_{r-1}^{a}\arrow[r]\arrow[d]
 & A_r^a\arrow[r]\arrow[d]
 & A_1^a\arrow[r]\arrow[d]
 & A_{r-1}^{a+1}\arrow[d]\\
 B_1^{a-1}\arrow[r]
 & B_{r-1}^{a}\arrow[r]
 & B_r^a\arrow[r]
 & B_1^a\arrow[r]
 & B_{r-1}^{a+1}.
\end{tikzcd}
\]
The first vertical map is surjective, and the second, fourth, and
fifth are isomorphisms. By the five lemma the middle map is thus an
isomorphism.
For $a=t$, the last four terms instead give surjectivity
of the middle map:
the maps on $A_{r-1}^t$ and $A_1^t$ are surjective,
and the map on $A_{r-1}^{t+1}$ is injective,
so the map on $A_r^t$ is surjective by the four lemma.

Since the finite-coefficient cohomology groups are finite, the inverse
systems of kernels of the surjective maps $G_h^t(\Lambda_r)$ satisfy
the Mittag--Leffler condition. Taking inverse limits therefore preserves
surjectivity at $a=t$, as well as the isomorphisms for $a>t$.
Now $H^b(Z,\ZZ_\ell)=\varprojlim_r H^b(Z,\Lambda_r)$,
giving the assertions for $\ZZ_\ell$.
Finally, the identity
$H^b(Z,\QQ_\ell)=H^b(Z,\ZZ_\ell)\otimes_{\ZZ_\ell}\QQ_\ell$
and the flatness of $\QQ_\ell$ over $\ZZ_\ell$
lead to the result for $\QQ_\ell$.
\end{proof}

\begin{cor}\label{app:ss:cor:Ql}
Let $1\le i\le n$.
For $h\notin Y_i$, the map $G_h^a(\QQ_\ell)$
is an isomorphism for $a>n-2+i$
and is surjective for $a=n-2+i$.
If $k=\FF_q$ and
$\bfv\neq\bm0$ satisfies $[\bfv]\notin Y_i(k)$, then
$$T_q(\bfv)\ll_{n,d} q^{(n-2+i)/2}.$$
\end{cor}
\begin{proof}
The first sentence follows from Lemma~\ref{app:ss:lem:coefficients} with $t=n-2+i$,
in view of Proposition~\ref{app:ss:prop:RS}.
We can use Lemma~\ref{app:ss:lem:reduction},
with $G_h^a(\QQ_\ell)$ in place of $G_h^i(\QQ_\ell)$,
to deduce the second sentence.
\end{proof}

\subsection{A construction over the integers}
We use the relative singular support of Hu and Yang \cite{HY} to
construct the exceptional loci over an open subscheme of $\Spec\ZZ$.
\subsubsection{Relative singular support}
Fix a prime $\ell$.
Let $N\ge 1$ be an integer multiple of $\ell$
and all primes dividing the content of $F$.
We will work over the base $B=\Spec\ZZ[1/N]$,
enlarging $N$ finitely many times without changing the notation.
Let
\[
 \mathcal P=\PP^{n-1}_B,\qquad
 \mathcal X_B=V(F)\subset\mathcal P,\qquad
 \phi_B:\mathcal X_B \hookrightarrow \mathcal P,\qquad
 \phi_{\bar b}:(\mathcal X_B)_{\bar b} \hookrightarrow \mathcal P_{\bar b},
\]
for every geometric point $\bar b\to B$.
Every fibre $(\mathcal X_B)_{\bar b}$ is a hypersurface of dimension $n-2$.

We use the following consequence of Hu--Yang's relative theory.
Part (2) of Theorem~5.2 and  and Theorem~5.3 of \cite{HY} imply, after shrinking $B$, the
existence of a closed conical subset
\begin{equation*}
 \mathcal C=\operatorname{SS}((\phi_B)_\ast\FF_\ell,\mathcal P/B)
                         \subset T^*(\mathcal P/B),
\end{equation*}
the singular support relative to $B$, such that, for every geometric point $\bar b\to B$,
\begin{equation}\label{app:ss:eq:relative-fibre}
 (\mathcal C_{\bar b})_{\mathrm{red}}
       =\operatorname{SS}((\phi_{\bar b})_\ast\FF_\ell).
\end{equation}
In other words, the reduction of the singular support is the singular support of the reduction away from finitely many primes.

Relativising the definitions of $\Sigma$ and $\pi$ in \eqref{app:ss:eq:Sigma},
we obtain
\[
 \Sigma_B=\PP(\mathcal C)
       \subset\mathcal P\times_B\mathcal P^\vee,
 \qquad \pi_B:\Sigma_B\longrightarrow\mathcal P^\vee.
\]
For $1\leq i\leq n$, define
\begin{equation*}
 \mathcal Y_i=\{h\in\mathcal P^\vee:
                    \dim(\pi_B^{-1}(\bar h))\geq i-1\}.
\end{equation*}
These sets are closed, by upper semicontinuity of fibre dimension for
proper morphisms \cite[\href{https://stacks.math.columbia.edu/tag/0D4I}{Tag 0D4I}]{TSta}.
By \eqref{app:ss:eq:relative-fibre}, their
geometric fibres are exactly the loci of Definition~\ref{app:ss:def:Yi}. Thus
\begin{equation}\label{app:ss:eq:fibre-codim}
 \codim_{\mathcal P^\vee_{\bar b}}\,(\mathcal Y_i)_{\bar b}\geq i
 \quad\text{for all }\bar b\to B,
 \qquad\mathcal Y_n=\emptyset,
\end{equation}
by Lemma~\ref{app:ss:lem:dimension}.
We give $\mathcal Y_i$ the reduced scheme structure.
Clearly $\mathcal Y_{i+1} \subset \mathcal Y_i$.

\subsubsection{Affine cones and specialization}
For $1\leq i\leq n$, let $\mathcal Z_{i,B}\subset\A^n_B$ be
the affine cone over $\mathcal Y_i$,
with the origin included even when $\mathcal Y_i=\emptyset$.
Set $\mathcal Z_{0,B}=\A^n_B$.
Formation of these cones on underlying sets commutes with field
extension, and \eqref{app:ss:eq:fibre-codim} gives
$\dim (\mathcal Z_{i,B})_{\bar b}\leq n-i$.

For $1\leq i\leq n$, define $\mathcal Z_i\subset\A^n_\ZZ$ to be the
closure of $\mathcal Z_{i,B}$, and set $\mathcal Z_0=\A^n_\ZZ$.
These are nested closed conical subschemes defined over $\ZZ$.
Each irreducible component is the closure of a component over $B$.
A component dominating $\Spec\ZZ$ has dimension at most $n+1-i$ by
the dimension formula \cite[\href{https://stacks.math.columbia.edu/tag/02JU}{Tag 02JU}]{TSta}; a vertical component
lies over a prime in $B$ and has dimension at most $n-i$.
Hence $\codim_{\A^n_\ZZ}\mathcal Z_i\geq i$.
Also, $(\mathcal Z_i)_B = \mathcal Z_{i,B}$
by \cite[\href{https://stacks.math.columbia.edu/tag/081I}{Tag 081I}]{TSta},
since $\mathcal Z_{i,B}$ is closed
and the open inclusion morphism $B\hookrightarrow \Spec\ZZ$ is flat.

\begin{prop}\label{app:ss:prop:global}
The schemes $\mathcal Z_i$ satisfy (\ref{item:strat-geom-dim-bound})
and (\ref{item:strat-fourier-bound}) of Theorem~\ref{n10:ss}.
\end{prop}
\begin{proof}
The dimension assertions in (\ref{item:strat-geom-dim-bound})
have just been proved.
Suppose $p\nmid N$,
$0\leq i\leq n-1$, and $\bfv\notin \mathcal Z_{i+1}(\FF_q)$.
Then $\bfv\neq\bm0$ and $[\bfv]\notin\mathcal Y_{i+1}(\FF_q)$, so
Corollary~\ref{app:ss:cor:Ql} with $i+1$ in place of $i$
gives the required estimate for (\ref{item:strat-fourier-bound}),
since $\mathcal Y_{i+1}(\FF_q) = Y_{i+1}(\FF_q)$ by \eqref{app:ss:eq:relative-fibre}.
\end{proof}

\subsubsection{Comparison with the geometric exceptional loci}
Suppose that $X = \mathcal X_\QQ$ is geometrically integral and $d\ge 2$.
Working over $\QQ$ rather than $\Qbar$,
recall the correspondences
$\Gamma$ and $S$ from \eqref{def:closed-gauss-correspondence} and
\eqref{EQ:universal-section-singular}. We claim that
\begin{equation}\label{app:ss:geometric-incidences}
 \Gamma\subset\Sigma_{\QQ}\subset S.
\end{equation}
At a smooth point of $X$, the singular support of
$\phi_*\FF_\ell$ is the ordinary conormal bundle of $X$ in $P$;
see \cite[Lemmas~2.3(i) and~2.5(ii)]{Bei}.
Taking closures gives the first inclusion. For the second inclusion,
let $(x,h)\in\Sigma_{\QQ}$ be a geometric point. The incidence
containment in \eqref{app:ss:eq:Sigma} gives $x\in H_h$, and the support property
of the image of $\Sigma_{\QQ}$ in $P$ gives
$x\in X$. If $x$ is smooth on $X$, then $(x,h)\in \Gamma$, so $(x,h)\in S$ by \eqref{EQ:gauss-section-containment}.
If $x$ is singular on $X$, all the first partial derivatives of $F$ vanish at $x$,
so $(x,h)\in S$ holds trivially by \eqref{EQ:universal-section-singular}.
This concludes the proof of \eqref{app:ss:geometric-incidences}.

Since $X_{\Qbar}$ is integral of degree $d\ge 2$,
the restriction of $F$ to every hyperplane is nonzero.
Thus $X\cap H_h$ is a hypersurface in $H_h$,
and the Jacobian criterion identifies the fibre $S_h$
with the scheme-theoretic singular locus of $X\cap H_h$.
Taking geometric fibre dimensions in \eqref{app:ss:geometric-incidences} gives
\[
 W_i\subset (\mathcal Y_{i+1})_\QQ \subset B_i
 \qquad(0\leq i\leq n-1).
\]
Taking affine cones, with the convention
$\cone\emptyset=\{\bm0\}$, proves assertion \eqref{EQ:nice-sandwich} of
Theorem~\ref{n10:ss}.

\subsection{Rational points on the exceptional loci}
To complete the proof of Theorem~\ref{n10:ss}, it remains to prove the
assertion \eqref{eq:rational-main}
about the Zariski closures $\bar{\mathcal Z_i(\QQ)}$,
assuming $X(\QQ)=\emptyset$.

\subsubsection{Fibre jumps and dual varieties}
We begin with the following elementary observation. 
\begin{lemma}\label{app:ss:lem:fibre-jump} 

Let $f:A\to B$ be a proper dominant morphism of integral varieties, and
put $c=\dim A-\dim B$. For $1\le j\le \dim B$, the closed locus
\[
 E_j=\{b\in B:\dim A_{\bar b}\geq c+j\}
\]
has codimension at least $j+1$ in $B$.
\end{lemma}
\begin{proof}
This is trivial if $E_j=\emptyset$, so assume $E_j\ne \emptyset$.
The generic fibre dimension is $c$, so $E_j$ is a proper closed subset of $B$,
since $j\ge 1$.
Since $f$ is dominant, $f^{-1}(E_j)$ is a proper closed subset of $A$.
Since $A$ is integral,
$$\dim f^{-1}(E_j) \le \dim A - 1 = \dim B + c-1.$$
However, for any irreducible component $E$ of $E_j$, some component of
$f^{-1}(E)$ dominating $E$ has dimension at least $\dim E+c+j$.
Thus $\dim f^{-1}(E_j) \ge \dim E_j+c+j$.
It follows that $\dim E_j \le \dim B - j-1$, as desired.
\end{proof}

For an integral proper subvariety $V\subsetneq\PP^m$ over a field of characteristic
zero, its projective conormal variety is
\[
 \mathcal N_V=\bar{\{(x,h):x\in V_{\mathrm{sm}},\
                       \mathbb T_xV\subset H_h\}}
             \subset\PP^m\times(\PP^m)^\vee,
\]
where $\mathbb T_xV$ denotes the embedded projective tangent space.
Its second projection is the dual variety $V^\vee$.
If $V$ is not geometrically integral, its geometric components and
their distinct duals are Galois conjugate.
Biduality identifies $\mathcal N_V$ with $\mathcal N_{V^\vee}$ after
interchanging the two factors, and gives $(V^\vee)^\vee=V$;
see \cite[Theorem~1.7]{Tevelev}.

\begin{lemma}\label{app:ss:lem:dual-rational}
Let $V\subsetneq\PP^m_\QQ$ be integral, but not necessarily geometrically
integral. If $V(\QQ)=\emptyset$, then
\[
 V^\vee(\QQ)\subset\Sing(V^\vee)(\QQ).
\]
\end{lemma}
\begin{proof}
Suppose that $h\in(V^\vee)_{\mathrm{sm}}(\QQ)$.
The annihilator
\[
 (\mathbb T_hV^\vee)^\perp\subset\PP^m_\QQ
\]
of the embedded tangent space is a projective linear space over $\QQ$
of dimension $m-\dim V^\vee-1\geq0$. By biduality it is the fibre
over $h$ of $\mathcal N_V\to V^\vee$, viewed in the first factor,
and is therefore contained in $V$. It has a rational point,
contrary to the hypothesis $V(\QQ)=\emptyset$.
\end{proof}

\subsubsection{Conormal components and the codimension gain}
We continue to use the notation $\mathcal F=\phi_*\FF_\ell$ and $\Sigma=\PP(\operatorname{SS}(\mathcal F))$ on $P_\QQ$ from the previous sections.
In characteristic zero, the components of singular support are conical
Lagrangian subvarieties \cite[Remark after Theorem~1.3]{Bei},
hence closures of conormal bundles to their images
\cite[Theorem~1.9]{Tevelev}. Those images lie in $X$, since the base
of singular support is $\operatorname{Supp} \mathcal F=X$. Therefore, for each irreducible component
$\Sigma_\alpha$ of $\Sigma$ over $\QQ$, there is an integral subvariety
$V_\alpha\subset X$ with $\Sigma_\alpha=\mathcal N_{V_\alpha}$.

\begin{prop}\label{app:ss:prop:rational-projective}
Assume $X(\QQ)=\emptyset$. For the generic-fibre loci
$Y_i=(\mathcal Y_i)_\QQ$ and $1\leq i\leq n-1$,
\begin{equation*}
 \codim_{P^\vee_\QQ}\bar{Y_i(\QQ)}\geq i+1.
\end{equation*}
\end{prop}
\begin{proof}
The projection $\Sigma_\alpha\to V_\alpha^\vee$ has generic relative
dimension $c_\alpha=n-2-\dim V_\alpha^\vee$. Put
\[
 E_{\alpha,i}=\{h\in V_\alpha^\vee:
                       \dim(\Sigma_\alpha)_{\bar h}\geq i-1\}.
\]
Then $Y_i=\bigcup_\alpha E_{\alpha,i}$. We distinguish three cases.

If $c_\alpha>i-1$, then
$\dim V_\alpha^\vee=n-2-c_\alpha\leq n-i-2$, and the same bound
holds for $\bar{E_{\alpha,i}(\QQ)}$.

If $c_\alpha<i-1$, Lemma~\ref{app:ss:lem:fibre-jump} with
$1\le j=i-1-c_\alpha\le n-2-c_\alpha=\dim V_\alpha^\vee$ gives
\[
 \dim E_{\alpha,i}
 \leq\dim V_\alpha^\vee-(i-1-c_\alpha)-1=n-i-2.
\]

If $c_\alpha=i-1$, then $\dim V_\alpha^\vee=n-i-1$.
Since $V_\alpha\subset X$ has no rational point,
Lemma~\ref{app:ss:lem:dual-rational} implies that $V_\alpha^\vee(\QQ)$ lies in the proper
closed subset $\Sing(V_\alpha^\vee)$, of dimension at most $n-i-2$.
The closure of $E_{\alpha,i}(\QQ)$ satisfies the same bound.
\end{proof}

\begin{proof}[Completion of the proof of Theorem~\ref{n10:ss}]
Let $1\le i\le n-1$ and assume $X(\QQ)=\emptyset$.
Every nonzero $\QQ$-point of $\mathcal Z_i$ projects to a $\QQ$-point of $Y_i$.
Thus $\bar{\mathcal Z_i(\QQ)}$ lies in the affine cone over $\bar{Y_i(\QQ)}$,
so $\dim \bar{\mathcal Z_i(\QQ)} \le 1 + \dim \bar{Y_i(\QQ)}
\le n-i-1$ by Proposition~\ref{app:ss:prop:rational-projective}.
This proves \eqref{eq:rational-main}.
\end{proof}

\end{document}